\documentclass[reqno,12pt]{amsart}
\usepackage{amssymb}
\usepackage{mathrsfs}
\usepackage{amsmath,amssymb}
\usepackage{paralist}
\usepackage{graphics} 
\usepackage{epsfig} 
\usepackage[colorlinks = true]{hyperref}
\hypersetup{urlcolor = red, citecolor = blue}
\usepackage[top = 1in,bottom = 1in,left = 1in,right = 1in]{geometry}
\usepackage{cite}
\usepackage{lineno}
\newtheorem{thm}{Theorem}[section]

\newtheorem{lem}[thm]{Lemma}

\theoremstyle{definition}
\newtheorem{defn}{Definition}[section]

\numberwithin{equation}{section}

\def\d{\,\mathrm{d}}
\def\dt{\frac{\mathrm{d}}{\mathrm{d}t}}
\def\r{\!\!\!}

\allowdisplaybreaks[3]

\begin{document}


\title[ZHANG AND LI] 
{Global boundedness for a two-dimensional doubly degenerate nutrient taxis system}
\author[Zhang]{Zhiguang Zhang }%
\address{School of Mathematics, Southeast University, Nanjing 211189, P. R. China; and School of Mathematics and Statistics, Chongqing Three Gorges University, Wanzhou 404020 , P. R. China}
\email{guangz$\_$z@163.com}

\author[Li]{Yuxiang Li$^{\star}$}
\address{School of Mathematics, Southeast University, Nanjing 211189, P. R. China}
\email{lieyx@seu.edu.cn}

\thanks{$^{\star}$Corresponding author.}
\thanks{Supported in part by National Natural Science Foundation of China (No. 12271092, No. 11671079) and the  Science and Technology Research Program of Chongqing Municipal Education Commission (No. KJQN202201226).}



\begin{abstract}
This paper is concerned with the doubly degenerate nutrient taxis system $u_t=\nabla \cdot(u^{l-1} v \nabla u)- \nabla \cdot\left(u^{l} v \nabla v\right)+ uv$ and $v_t=\Delta v-u v$ for some $l \geqslant 1$, subjected to homogeneous Neumann boundary conditions in a smooth bounded convex domain $\Omega \subset \mathbb{R}^n$ $(n \leqslant 2)$. Through distinct approaches, we establish that for sufficiently regular initial data, in two-dimensional contexts, if $l \in[1,3]$, then the system possesses global weak solutions, and in one-dimensional settings, the same conclusion holds for $l \in[1,\infty)$. Notably, the solution remains uniformly bounded when $l \in[1,\infty)$ in one dimension or $l \in(1,3]$ in two dimensions. 
\end{abstract}
\maketitle

\section{Introduction and main results}\label{section1}

In the present work, we consider a doubly degenerate nutrient taxis system
\begin{equation}\label{SYS:MAIN}
\begin{cases}u_t=\nabla \cdot(u^{l-1} v \nabla u)- \nabla \cdot\left(u^{l}v \nabla v\right)+  u v, & x \in \Omega, t>0, \\ v_t=\Delta v-u v, & x \in \Omega, t>0, \\ \left(u v \nabla u- u^{2} v \nabla v\right) \cdot \nu=\nabla v \cdot \nu=0, & x \in \partial \Omega, t>0, \\ u(x, 0)=u_0(x), \quad v(x, 0)=v_0(x), & x \in \Omega, \end{cases}
\end{equation}
where $\Omega \subset \mathbb{R}^n$ $(n \leqslant 2)$ is a bounded domain with smooth boundary, $\nu$ is the outward unit normal vector at a point of $\partial \Omega$ and $l \geqslant 1$. The scalar functions $u$ and $v$ denote the concentration of the nutrient, respectively.

Fujikawa \cite{1992-PASMaiA-Fujikawa}, Fujikawa and Matsushita \cite{1989-JotPSoJ-FujikawaMatsushita}, and Matsushita and Fujikawa \cite{1990-PASMaiA-MatsushitaFujikawa} showed that \textit{Bacillus subtilis} forms complex patterns, including snowflake-like distributions, under nutrient scarcity. To model these dynamics, Kawasaki et al. \cite{1997-JoTB-Kawa} proposed the diffusion model
\begin{align}\label{sys:Kaw}
\begin{cases}
u_t=\nabla \cdot (D_u(u,v) \nabla u) + uv,  \\ 
v_t=D_v \Delta v - uv, 
\end{cases}
\end{align}
where $D_u(u,v)$ is the bacterial diffusion coefficient and $D_v$ is the constant nutrient diffusion coefficient. Evidences presumably suggest that bacteria are immotile when either $u$ or $v$ is low and become active as $u$ or $v$
increases, so they proposed the simplest diffusion coefficient as 
\begin{align*}
D_u(u,v)=uv.
\end{align*} 
This system associated with no-flux boundary condition in a smooth bounded convex domain $\Omega \subset \mathbb{R}^n$ was studied by Winkler \cite{2022-CVPDE-Winkler}. It is proved that for nonnegative initial data satisfying certain conditions, there exists a global weak solution $(u, v)$ which satisfies $\sup _{t>0}\|u(t)\|_{L^p(\Omega)}<\infty$ for all $p \in [1, p_0)$ with $p_0=\frac{n}{(n-2)_{+}}$, and one can find $u_{\infty} \in \bigcap_{p \in\left[1, p_0\right)} L^p(\Omega)$ such that, within an appropriate topological setting, $(u(t), v(t))$ approaches the equilibrium $\left(u_{\infty}, 0\right)$ in the large time limit.

\vskip 1mm

To describe the formation of the said aggregation pattern with more accuracy, Leyva et al. \cite{2013-PA-LeyvaMalagaPlaza} extended the degenerate diffusion model \eqref{sys:Kaw} to the following doubly degenerate nutrient taxis system 
\begin{align}\label{SYS:LEYVA}
\begin{cases}
u_t=\nabla \cdot(u v \nabla u)-\nabla \cdot\left(u^{2}v \nabla v\right)+ u v,  \\ 
v_t=\Delta v- u v. 
\end{cases}
\end{align}
The term ``doubly degenerate" means the diffusion coefficient of the bacteria becomes degenerate when $ u \text { or } v \rightarrow 0$. In the two-dimensional setting, numerical simulations in \cite{1997-JoTB-Kawa, 2000-AiP-Ben-JacobCohenLevine, 2013-PA-LeyvaMalagaPlaza} indicated that, depending on the initial data and parameter conditions, the model \eqref{SYS:LEYVA} could generate rich branching pattern, which is very close to that observed in biological experiments.
In \cite{2021-TAMS-Winkler}, Winkler studied \eqref{SYS:LEYVA} in the one-dimensional setting, i.e. the following cross-diffusion system 
\begin{align}\label{ONESYS:WINKLER}
\begin{cases}
u_t=\left(u v u_x\right)_x-\left(u^2 v v_x\right)_x+u v, & x \in \Omega, t>0, \\ 
v_t=v_{x x}-u v, & x \in \Omega, t>0, \\ 
u v u_x-u^2 v v_x=0, \quad v_x=0, & x \in \partial \Omega, t>0, \\ 
u(x, 0)=u_0(x), \quad v(x, 0)=v_0(x), & x \in \Omega,
\end{cases}
\end{align}
where the initial data in \eqref{ONESYS:WINKLER} are assumed to be such that 
\begin{align}\label{ONESYS:WINKLERa}
\begin{cases}
u_0 \in C^{\vartheta}(\overline{\Omega}) \text { for some } \vartheta \in(0,1), \text { with } u_0 \geqslant 0 \text { and } \int_{\Omega} \ln u_0>-\infty, \quad \text { and that } \\ 
v_0 \in W^{1, \infty}(\Omega) \text { satisfies } v_0>0 \text { in } \overline{\Omega}.
\end{cases}
\end{align}
Based on the method of energy estimates, the author demonstrated that the system \eqref{ONESYS:WINKLER} admits a global weak solution, which is uniform-in-time bounded and converges to some equilibrium within a defined topological space. Later, Li and Winkler \cite{2022-CPAA-LiWinklera} removed $\int_{\Omega} \ln u_0>-\infty$ in \eqref{ONESYS:WINKLERa} and obtained similar results.

The associated no-flux type initial-boundary value problem for \eqref{SYS:LEYVA} is considered by Winkler \cite{2022-NARWA-Winkler} in smooth bounded convex domain of the plane. It is shown that for any $p>2$ and each fixed nonnegative $u_0 \in W^{1, \infty}(\Omega)$, a smallness condition exclusively involving $v_0$ can be identified as sufficient to ensure that this problem with $\left.(u, v)\right|_{t=0}=\left(u_0, v_0\right)$ admits a global weak solution satisfying $\mathop{\mathrm{ess\,sup}}\limits_{t>0}\|u(t)\|_{L^p(\Omega)}<\infty$. 
Recently, in order to obtain the large-data solutions, the model \eqref{SYS:LEYVA} is modified via replacing $\nabla \cdot\left(u^2 v \nabla v\right)$ by $\nabla \cdot\left(u^l v \nabla v\right)$ and is transformed to the following system
\begin{align}\label{SYS:LIWINKLER}
\begin{cases}
u_t=\nabla \cdot(u v \nabla u)- \nabla \cdot\left(u^l v \nabla v\right)+ u v, \\
v_t=\Delta v-u v.
\end{cases}
\end{align}
It is shown in \cite{2022-JDE-Li} that the associated no-flux type initial-boundary value problem for \eqref{SYS:LIWINKLER} in a smooth bounded convex domain $\Omega \subset \mathbb{R}^n$, $n=2,3$, admits a global weak solution if either $l \in \left(1, \frac{3}{2}\right)$ when $n=2$ or $l \in \left(\frac{7}{6}, \frac{13}{9}\right)$ when $n=3$. 

Very recently, Winkler \cite{2024-JDE-Winkler} considered a more general variant of the system \eqref{SYS:LEYVA}, including the system \eqref{SYS:LIWINKLER} as a special case. Global bounded weak solutions were obtained in a bounded convex planar domain when either $\alpha<2$ and the initial data are reasonably regular as well as arbitrary large, or $\alpha=2$ but the initial data are such that $v_0$ satisfies an appropriate smallness condition.

\vskip 2mm

Let us point out that, in the past decade, the following relatively simpler nutrient taxis system      
\begin{equation}\label{sys-1.5}
\begin{cases}
u_t=\Delta u-\nabla \cdot(u \nabla v), \\ v_t=\Delta v-u v,
\end{cases}
\end{equation}
has been studied extensively by many authors. It is proved that the associated no-flux  initial-boundary value problem for the system \eqref{sys-1.5} admits global classical solutions in bounded planar domains \cite{2012-CPDE-Winkler}, and admits global weak solutions in bounded three-dimensional domains that eventually become smooth in the large time \cite{2012-jde-taoyou}. In the higher dimensional setting, Tao \cite{2011-JMAA-Tao} obtained global bounded classical solutions under a smallness assumption on the initial signal concentration $v_0$. Wang and Li \cite{2019-EJDE-WangLi} showed that this model possesses at least one global renormalized solution. For more related results, we refer readers to the surveys \cite{2015-MMMAS-BellomoBellouquidTaoWinkler,2023-SAM-LankeitWinkler} and the references therein. Recently, blow-up problems for nutrient taxis system have been investigated by some authors. We refer the reader to \cite{2024-JNS-Jin} for singular attraction-consumption systems, and to \cite{2023-PRSESA-WangWinkler, 2023-CVPDE-AhnWinkler} for repulsion-consumption systems.  

Motivated by the aforementioned works, this paper aims to investigate the existence of global bounded weak solutions for a doubly degenerate nutrient taxi system \eqref{SYS:MAIN}.

\vskip 3mm

\textbf{Main results.} 
Throughout our analysis, we will assume that the initial data satisfy 
\begin{align}\label{assIniVal}
\begin{cases}
u_0 \in \left\{\begin{array}{lll}
W^{1, \infty}(\Omega) \text {  with  } u_0 \geqslant 0   & \text {  when  } & 1 \leqslant l<2 , \\
W^{1, \infty}(\Omega) \text { with } u_0 \geqslant 0 \text {  and  } \int_{\Omega} \ln u_0>-\infty    & \text {   when   } & l=2, \\
W^{1, \infty}(\Omega) \text {   with  } u_0 \geqslant 0 \text {  and   } \int_{\Omega} u_{0}^{2-l}<\infty    & \text {   when   } & l>2  \quad \text {  as well as   }  \\  
\end{array}\right.\\ v_0 \in W^{1, \infty}(\Omega) \text {    satisfies   }  v_0 > 0  \text { in } \overline{\Omega}.
\end{cases}
\end{align}

First, we introduce the definition of a weak solution to the system \eqref{SYS:MAIN}.

\begin{defn} \label{def-weak-sol}
Let $l\geqslant1$ and $\Omega \subset \mathbb{R}^n$ $(n \leqslant 2)$ be a bounded domain with smooth boundary. Suppose that $u_0 \in L^1(\Omega)$ and $v_0 \in L^1(\Omega)$ are nonnegative. By a global weak solution of the system (\ref{SYS:MAIN}) we mean a pair $(u, v)$ of functions satisfying
\begin{equation}\label{-2.1}
\begin{cases}
u \in L_{\mathrm{loc}}^1(\overline{\Omega} \times[0, \infty)) \quad \text { and } \\
v \in L_{\mathrm{loc}}^{\infty}(\overline{\Omega} \times[0, \infty)) \cap L_{\mathrm{loc}}^1\left([0, \infty) ; W^{1,1}(\Omega)\right)
\end{cases}
\end{equation}
and
\begin{equation}\label{-2.2}
u^l \in L_{\mathrm{loc}}^1\left([0, \infty) ; W^{1,1}(\Omega)\right) \quad \text { and } \quad u^{l} \nabla v \in L_{\mathrm{loc}}^1\left(\overline{\Omega} \times[0, \infty) ; \mathbb{R}^2\right),
\end{equation}
which are such that
\begin{align}\label{-2.3}
-\int_0^{\infty} \int_{\Omega} u \varphi_t-\int_{\Omega} u_0 \varphi(\cdot, 0)=&-\frac{1}{l} \int_0^{\infty} \int_{\Omega} v  \nabla u^{l} \cdot \nabla \varphi
+ \int_0^{\infty} \int_{\Omega} u^{l} v \nabla v \cdot \nabla \varphi \nonumber\\
& + \int_0^{\infty} \int_{\Omega} u v \varphi, 
\end{align}
and
\begin{align}\label{-2.4}
\int_0^{\infty} \int_{\Omega} v \varphi_t+\int_{\Omega} v_0 \varphi(\cdot, 0)=\int_0^{\infty} \int_{\Omega} \nabla v \cdot \nabla \varphi+\int_0^{\infty} \int_{\Omega} u v \varphi
\end{align}
for all $\varphi \in C_0^{\infty}(\overline{\Omega} \times[0, \infty))$.
\end{defn}

Next, we state the global existence of weak solutions for the system \eqref{SYS:MAIN}.

\begin{thm} \label{thm-1.1}
Let $1<l\leqslant3$ and let $\Omega \subset \mathbb{R}^2$ be a bounded convex domain with smooth boundary. Assume that the initial value $\left(u_0, v_0\right)$ satisfies \eqref{assIniVal}. Then there exist functions
\begin{align}\label{solu:property2}
\begin{cases}
u \in C^{0}(\overline{\Omega} \times[0, \infty)) \quad \text { and } \\
v \in C^0(\overline{\Omega} \times[0, \infty)) \cap C^{2,1}(\overline{\Omega} \times(0, \infty))
\end{cases}
\end{align}
such that $u > 0$ a.e in $\Omega \times(0, \infty)$ and $v>0$ in $\overline{\Omega} \times[0, \infty)$, and that $(u, v)$ solves \eqref{SYS:MAIN} in the sense of Definition \ref{def-weak-sol}, and that $(u, v)$  has the property that for each $p \geq 1$ there exists $C(p)>0$ fulfilling
\begin{align*}
\|u(t)\|_{L^p(\Omega)}+\|v(t)\|_{W^{1, \infty}(\Omega)} \leqslant C(p) \text { for all } t>0
\end{align*} 
with some constant $C>0$.
\end{thm}

\begin{thm} \label{thm-1.1xzc}
Let $l=1$ and let $\Omega \subset \mathbb{R}^2$ be a bounded convex domain with smooth boundary. Assume that the initial value $\left(u_0, v_0\right)$ satisfies \eqref{assIniVal}. Then there exist functions
\begin{align}\label{solu:property1}
\begin{cases}
u \in C^{0}(\overline{\Omega} \times[0, \infty)) \quad \text { and } \\
v \in C^0(\overline{\Omega} \times[0, \infty)) \cap C^{2,1}(\overline{\Omega} \times(0, \infty))
\end{cases}
\end{align}
such that $u\geqslant 0$ and $v>0$ in $\overline{\Omega} \times[0, \infty)$, and that $(u, v)$ solves \eqref{SYS:MAIN} in the sense of Definition \ref{def-weak-sol}, and that $(u, v)$  has the property that for each $T>0$ there exists $C(T)>0$ fulfilling
\begin{align*}
\|u(t)\|_{L^{\infty}(\Omega)}+\|v(t)\|_{L^{\infty}(\Omega)} \leqslant C(T)\quad \text { for all } t \in(0, T).
\end{align*}
\end{thm}

\begin{thm} \label{thm-1.1a}
Let $l\geqslant1$ and let $\Omega$ be an open bounded interval in $\mathbb{R}^1$. Assume that the initial value $\left(u_0, v_0\right)$ satisfies \eqref{assIniVal}. Then there exist functions
\begin{align}\label{solu:property3}
\begin{cases}
u \in C^{0}(\overline{\Omega} \times[0, \infty)) \quad \text { and } \\
v \in C^0(\overline{\Omega} \times[0, \infty)) \cap C^{2,1}(\overline{\Omega} \times(0, \infty))
\end{cases}
\end{align}
such that $u \geqslant 0$ and $v>0$ in $\overline{\Omega} \times[0, \infty)$, and that $(u, v)$ solves \eqref{SYS:MAIN} in the sense of Definition \ref{def-weak-sol}, and that $(u, v)$ is bounded in the sense that
\begin{align*}
\|u(t)\|_{L^\infty(\Omega)}+\|v(t)\|_{W^{1, \infty}(\Omega)} \leqslant C \quad \text { for all } t>0
\end{align*} 
with some constant $C>0$.
\end{thm}

\vskip 1mm

The remainder of this paper is organized as follows. 
In Section \ref{section2}, we define the global weak solutions of the system \eqref{SYS:MAIN} and provide some initial findings regarding the local-in-time existence of the problem \eqref{SYS:MAIN}. In Section \ref{sect-3x}, we give some a priori estimates and state straightforward consequences.  In Section \ref{sect-3} and
\ref{sect-5}, we prove Theorems \ref{thm-1.1}, \ref{thm-1.1xzc} and \ref{thm-1.1a} respectively.



\section{Preliminaries}\label{section2}

In order to construct a weak solution of the system \eqref{SYS:MAIN}, we consider the following regularized problem for $\varepsilon \in(0,1)$ and $l\geqslant1$,
\begin{align}\label{sys-regul}
\begin{cases}
u_{\varepsilon t}=\nabla \cdot\left(u_{\varepsilon}^{l-1} v_{\varepsilon} \nabla u_{\varepsilon}\right)
  - \nabla \cdot\left(u_{\varepsilon}^{l} v_{\varepsilon} \nabla v_{\varepsilon}\right)+ u_{\varepsilon}v_{\varepsilon}, 
    & x \in \Omega, t>0, \\ 
v_{\varepsilon t}=\Delta v_{\varepsilon}-u_{\varepsilon} v_{\varepsilon}, 
   & x \in \Omega, t>0, \\ 
\frac{\partial u_{\varepsilon}}{\partial \nu}=\frac{\partial v_{\varepsilon}}{\partial \nu}=0, 
   & x \in \partial \Omega, t>0, \\ 
u_{\varepsilon}(x, 0)=u_{0 \varepsilon}(x):=u_0(x)+\varepsilon, \quad 
   v_{\varepsilon}(x, 0)=v_{0 \varepsilon}(x):=v_0(x), 
   & x \in \Omega.
\end{cases}
\end{align}

For any such problem, standard theory yields the following result on local existence and extensibility of smooth solutions.

\begin{lem}\label{lemma-2.1}
Let $l\geqslant1$ and $\Omega \subset \mathbb{R}^n$ $(n \leqslant 2)$ be a bounded domain with smooth boundary, and suppose that (\ref{assIniVal}) holds. Then for each $\varepsilon \in(0,1)$, there exist $T_{\max, \varepsilon} \in(0, \infty]$ and at least one pair $\left(u_{\varepsilon}, v_{\varepsilon}\right)$ of functions
\begin{align}\label{-2.6}
\begin{cases}
u_{\varepsilon} \in C^0\left(\overline{\Omega} \times\left[0, T_{\max, \varepsilon}\right)\right) \cap C^{2,1}\left(\overline{\Omega} \times\left(0, T_{\max, \varepsilon}\right)\right) \quad \text{ and } \\
v_{\varepsilon} \in \cap_{q > 2} C^0\left(\left[0, T_{\max, \varepsilon}\right); W^{1, q}(\Omega)\right) \cap C^{2,1}\left(\overline{\Omega} \times\left(0, T_{\max, \varepsilon}\right)\right)
\end{cases}
\end{align}
which are such that $u_{\varepsilon}>0$ and $v_{\varepsilon}>0$ in $\overline{\Omega} \times\left[0, T_{\max, \varepsilon}\right)$, that $\left(u_{\varepsilon}, v_{\varepsilon}\right)$ solves \eqref{sys-regul} in the classical sense, and that
\begin{align}\label{-2.7}
if \text{  } T_{\max, \varepsilon}<\infty, \quad \text { then } \quad \limsup _{t \nearrow T_{\max, \varepsilon}}\left\|u_{\varepsilon}(t)\right\|_{L^{\infty}(\Omega)}=\infty.
\end{align}
In addition, this solution satisfies
\begin{align}\label{-2.8}
\int_{\Omega} u_{0 \varepsilon} \leqslant \int_{\Omega} u_{\varepsilon}(t) \leqslant \int_{\Omega} (u_{0}+ \varepsilon)+ \int_{\Omega} v_{0}, \quad \text { for all } t \in\left(0, T_{\max, \varepsilon}\right)
\end{align}
and
\begin{align}\label{-2.9}
\left\|v_{\varepsilon}(t)\right\|_{L^{\infty}(\Omega)} \leqslant \|v_0\|_{L^{\infty}(\Omega)}, \quad \text { for all } t \in\left(0, T_{\max, \varepsilon}\right)
\end{align}
as well as
\begin{align}\label{-2.10}
\int_{0}^{T_{\max, \varepsilon}}\!\!\! \int_{\Omega} u_{\varepsilon} v_{\varepsilon} \leqslant \int_{\Omega} v_0.
\end{align}
\end{lem}
\begin{proof}
Since the initial data of \eqref{sys-regul} is strictly positive in $\overline{\Omega}$, the statement on existence, extensibility
and basic properties can be proved according to well-established approaches, see \cite[Lemma 2.1]{2022-NARWA-Winkler} for details. 
\end{proof}

Now, we derive below two functional inequalities, which play key roles in the two dimensional setting.  

\begin{lem}\label{lemma-3.4}
Let $\Omega \subset \mathbb{R}^2$ be a bounded domain with smooth boundary and $p \geqslant 1$. For any $\varphi, \psi \in C^1(\overline{\Omega})$ satisfying $\varphi,\psi>0$ in $\overline{\Omega}$, there holds
\begin{align}\label{eq-6.1}
\int_{\Omega} \varphi^{p +1} \psi 
  \leqslant  c \left\{\int_{\Omega} \frac{\varphi}{\psi}|\nabla \psi|^2\
    +\int_{\Omega} \frac{\psi}{\varphi}|\nabla \varphi|^2
    +\int_{\Omega} \varphi \psi  \right\}\cdot \int_{\Omega} \varphi^p
\end{align}
for some constant $c=c(p)>0$.
\end{lem}
\begin{proof}
Recall the Sobolev imbedding inequality in $\Omega \subset \mathbb{R}^2$,
\begin{align}\label{eq-6.2}
\int_{\Omega} \rho^2 \leqslant C\|\nabla \rho\|_{L^1(\Omega)}^2 + C\|\rho\|_{L^{1}(\Omega)}^2, \quad 
  \rho \in W^{1,1}(\Omega).
\end{align}
For any $\varphi, \psi \in C^1(\overline{\Omega})$ satisfying $\varphi,\psi>0$ in $\overline{\Omega}$, we apply \eqref{eq-6.2} with $\rho=\varphi^{\frac{p +1}{2}} \psi^{\frac{1}{2}}$ to infer that
\begin{align}\label{eq-6.3}
\int_{\Omega} \varphi^{p +1} \psi \leqslant 
&\ C\left\{\int_{\Omega}\left|\frac{p+1}{2}\varphi^{\frac{p -1}{2}}\psi^{\frac{1}{2}}\nabla\varphi
    +\frac{1}{2}\varphi^{\frac{p +1}{2}} \psi^{-\frac{1}{2}} \nabla \psi\right|\right\}^2
 +C\cdot\left\{\int_{\Omega}\varphi^{\frac{p+1}{2}} \psi^{\frac{1}{2}}\right\}^{2}\nonumber\\
\leqslant &\ \frac{(p+1)^2 C}{2} \left\{\int_{\Omega} \varphi^{\frac{p -1}{2}} \psi^{\frac{1}{2}}|\nabla \varphi|\right\}^2
  +\frac{C}{2}\left\{\int_{\Omega} \varphi^{\frac{p +1}{2}} \psi^{-\frac{1}{2}}|\nabla \psi|\right\}^{2} \nonumber\\
&\ +C\left\{\int_{\Omega}\varphi^{\frac{p +1}{2}} \psi^{\frac{1}{2}}\right\}^{2}.
\end{align}
By Cauchy-Schwarz inequality, we have
\begin{align*}
\left\{\int_{\Omega} \varphi^{\frac{p -1}{2}} \psi^{\frac{1}{2}}|\nabla \varphi|\right\}^2 =\left\{\int_{\Omega} \varphi^{\frac{p}{2}}\cdot \frac{\psi^{\frac{1}{2}}}{\varphi^{\frac{1}{2}}} |\nabla \varphi|\right\}^2 \leqslant \int_{\Omega} \varphi^p  \cdot \int_{\Omega} \frac{\psi}{\varphi}|\nabla \varphi|^2
\end{align*}
and
\begin{align*}
\left\{\int_{\Omega} \varphi^{\frac{p +1}{2}} \psi^{-\frac{1}{2}}|\nabla \psi|\right\}^2 =
\left\{\int_{\Omega} \varphi^{\frac{p}{2}}\cdot \frac{\varphi^{\frac{1}{2}}}{\psi^{\frac{1}{2}}} |\nabla \psi|\right\}^2
\leqslant \int_{\Omega} \varphi^p  \cdot \int_{\Omega} \frac{\varphi}{\psi}|\nabla \psi|^2.
\end{align*}
Based on H\"{o}lder's inequality, we see that
\begin{align*}
\left\{\int_{\Omega}\varphi^{\frac{p +1}{2}} \psi^{\frac{1}{2}}\right\}^{2} & 
=\left\{\int_{\Omega} \varphi^{\frac{p }{2}}  \cdot (\varphi \psi)^{\frac{1}{2}}\right\}^{2} \leqslant \int_{\Omega} \varphi^p  \cdot \int_{\Omega} \varphi \psi.
\end{align*}
Therefore, \eqref{eq-6.1} results from \eqref{eq-6.3} if we let $c=c(p)=\max \left\{\frac{(p +1)^2 C}{2}, C\right\}$.
\end{proof}

\begin{lem}\label{lemma-3.5}
Let $\Omega \subset \mathbb{R}^2$ be a bounded domain with smooth boundary and $p \geqslant 1$. For each $\eta>0$ and any $\varphi, \psi \in C^1(\overline{\Omega})$ satisfying $\varphi,\psi>0$ in $\overline{\Omega}$, there holds
\begin{align}\label{eq-6.4}
\int_\Omega \varphi^{p+1} \psi |\nabla \psi|^2
\leqslant & \eta \int_\Omega \varphi^{p-1} \psi|\nabla \varphi|^2
  + c \left\{\left\|\psi \right\|_{L^{\infty}(\Omega)}+\frac{\left\|\psi \right\|^3_{L^{\infty}(\Omega)}}{\eta}\right\} 
    \cdot \int_{\Omega} \varphi^{p+1} \psi  \cdot \int_\Omega \frac{|\nabla \psi|^4}{\psi^3} \nonumber\\
& + c \left\|\psi \right\|^2_{L^{\infty}(\Omega)} \cdot\left\{\int_\Omega \varphi\right\}^{2 p+1} 
    \cdot \int_\Omega \frac{|\nabla \psi|^4}{\psi^3} 
    + c \left\|\psi \right\|^2_{L^{\infty}(\Omega)} \cdot \int_{\Omega}\varphi \psi
\end{align}
for some constant $c=c(p)>0$.
\end{lem}

\begin{proof}
Applying \eqref{eq-6.2}, we easily obtain
\begin{align}\label{eq-6.5}
\|\rho\|_{L^2(\Omega)} 
\leqslant C\|\nabla \rho\|_{L^1(\Omega)}+C\|\rho\|_{L^{\frac{1}{p+1}}(\Omega)}, \quad \rho \in W^{1,1}(\Omega)
\end{align}
for some $C=C(p)>0$. 

We use H\"{o}lder's inequality to yield
\begin{align*}
\int_\Omega \varphi^{p+1} \psi |\nabla \psi|^2 \leqslant\left\{\int_\Omega \frac{|\nabla \psi|^4}{\psi^3}\right\}^{\frac{1}{2}} \cdot\left\{\int_\Omega \varphi^{2 (p+1)} \psi^{5}\right\}^{\frac{1}{2}}.
\end{align*}
For any $\varphi, \psi \in C^1(\overline{\Omega})$ satisfying $\varphi,\psi>0$ in $\overline{\Omega}$, we apply \eqref{eq-6.5} with $\rho=\varphi^{p+1} \psi^{\frac{5}{2}}$ to infer that
\begin{align*}
\left\{\int_\Omega \varphi^{2 (p+1)} \psi^{5}\right\}^{\frac{1}{2}} 
= &\ \left\|\varphi^{p+1} \psi^{\frac{5}{2}}\right\|_{L^2(\Omega)} \nonumber\\
\leqslant &\  C \int_\Omega\left|(p+1) \varphi^{p} \psi^{\frac{5}{2}} \nabla \varphi
    +\frac{5}{2} \varphi^{p+1}\psi^\frac{3}{2}\nabla \psi\right|
    +C\left\{\int_\Omega \varphi \psi^{\frac{5}{2 (p+1)}}\right\}^{p+1} \nonumber\\
= &\ (p+1) C \int_\Omega \varphi^{p} \psi^\frac{5}{2}|\nabla \varphi|
     +\frac{5 C}{2} \int_\Omega \varphi^{p+1}\psi^\frac{3}{2}|\nabla \psi| \nonumber\\
  &\ + C \left\{\int_\Omega \varphi \psi^{\frac{5}{2 (p+1)}}\right\}^{p+1}.
\end{align*}
Applying H\"{o}lder's inequality, we show that
\begin{align*}
\int_G \varphi^{p} \psi^\frac{5}{2}|\nabla \varphi| 
\leqslant \left\|\psi\right\|_{L^{\infty}(\Omega)}^\frac{3}{2} \cdot\left\{\int_\Omega \varphi^{p+1}\psi\right\}^{\frac{1}{2}} \cdot\left\{\int_\Omega \varphi^{p-1} \psi|\nabla \varphi|^2\right\}^{\frac{1}{2}}
\end{align*}
and
\begin{align*}
\int_\Omega \varphi^{p+1}\psi^\frac{3}{2}|\nabla \psi| 
\leqslant \left\|\psi\right\|_{L^{\infty}(\Omega)}^\frac{1}{2} \cdot\left\{\int_\Omega \varphi^{p+1}\psi\right\}^{\frac{1}{2}} \cdot\left\{\int_\Omega \varphi^{p+1} \psi|\nabla \psi|^2\right\}^{\frac{1}{2}}
\end{align*}
as well as
\begin{align*}
\left\{\int_\Omega \varphi \psi^{\frac{5}{2 (p+1)}}\right\}^{p+1} 
& \leqslant \left\|\psi \right\|_{L^{\infty}(\Omega)}^2 \cdot\left\{\int_\Omega\varphi^{\frac{2 p+1}{2 (p+1)}}\cdot(\varphi \psi)^{\frac{1}{2 (p+1)}}\right\}^{p+1} \\
& \leqslant  \left\|\psi \right\|_{L^{\infty}(\Omega)}^2 \cdot\left\{\int_\Omega \varphi\right\}^{\frac{2 p+1}{2}} \cdot\left\{\int_\Omega \varphi \psi\right\}^{\frac{1}{2}}.
\end{align*}
Therefore, for all $\eta>0$, it follows that
\begin{align*}
\int_\Omega \varphi^{p+1} \psi |\nabla \psi|^2 
\leqslant & (p+1) C\left\|\psi\right\|_{L^{\infty}(\Omega)}^\frac{3}{2} \cdot\left\{\int_\Omega \frac{|\nabla 
    \psi|^4}{\psi^3}\right\}^{\frac{1}{2}} 
    \cdot\left\{\int_\Omega \varphi^{p+1}\psi\right\}^{\frac{1}{2}} 
    \cdot\left\{\int_\Omega \varphi^{p-1} \psi|\nabla \varphi|^2\right\}^{\frac{1}{2}}\\
& +\frac{5 C\left\|\psi\right\|_{L^{\infty}(\Omega)}^\frac{1}{2}}{2} 
    \cdot \left\{\int_\Omega \frac{|\nabla \psi|^4}{\psi^3}\right\}^{\frac{1}{2}}
    \cdot\left\{\int_\Omega \varphi^{p+1}\psi\right\}^{\frac{1}{2}} 
    \cdot\left\{\int_\Omega \varphi^{p+1} \psi|\nabla \psi|^2\right\}^{\frac{1}{2}} \\
& + C\left\|\psi \right\|_{L^{\infty}(\Omega)}^2 \cdot\left\{\int_\Omega \frac{|\nabla \psi|^4}{\psi^3}\right\}^{\frac{1}{2}} 
    \cdot\left\{\int_\Omega \varphi\right\}^{\frac{2 p+1}{2}} \cdot\left\{\int_\Omega \varphi \psi\right\}^{\frac{1}{2}} \\
\leqslant &   \frac{\eta}{2} \int_\Omega \varphi^{p-1} \psi|\nabla \varphi|^2
  +\frac{(p+1)^2 C^2\left\|\psi\right\|_{L^{\infty}(\Omega)}^3}{2 \eta} \cdot \int_\Omega \varphi^{p+1} \psi 
    \cdot \int_\Omega \frac{|\nabla \psi|^4}{\psi^3} \\
& + \frac{1}{2} \int_\Omega \varphi^{p+1} \psi |\nabla \psi|^2
  + \frac{25 C^2\left\|\psi\right\|_{L^{\infty}(\Omega)}}{8} 
  \cdot \int_\Omega \varphi^{p+1} \psi \cdot \int_\Omega \frac{|\nabla \psi|^4}{\psi^3}\\
& +C\left\|\psi \right\|_{L^{\infty}(\Omega)}^2 \cdot\left\{\int_\Omega \varphi\right\}^{2 p+1} 
  \cdot \int_\Omega \frac{|\nabla \psi|^4}{\psi^3}
  +C\left\|\psi \right\|_{L^{\infty}(\Omega)}^2 \int_\Omega \varphi \psi
\end{align*}
which implies \eqref{eq-6.4} if we set $c=c(p)=\max \left\{2C, (p+1)^2 C^2, \frac{25 C^2}{4}\right\}$.
\end{proof}

\section{Fundamental a Priori Estimates for \eqref{sys-regul}}\label{sect-3x}
In \cite{2021-TAMS-Winkler}, Winkler found a gradient-like structure of the system \eqref{sys-regul}. We use such a structure to derive a uniform estimate of the temporal-spatial integral of $\frac{v_{\varepsilon}}{u_{\varepsilon}}|\nabla  u_{\varepsilon }|^2$.

\begin{lem}\label{lem-1st-est}
Let $l >1$ and assume that \eqref{assIniVal} holds. 
Then we have 
\begin{align}\label{-3.5aa} 
\int_0^{T_{\max, \varepsilon}}\r \int_{\Omega} \frac{v_{\varepsilon}}{u_{\varepsilon}}|\nabla u_{\varepsilon }|^2  \leqslant C  \quad \text { for all } \varepsilon \in(0,1),
\end{align}
where $C$ is a positive constant depending on $\int_{\Omega} u_0$, $\int_{\Omega} v_0$ and $\int_{\Omega} |\nabla v_{0 }|^2$, but independent of $\varepsilon$.
\end{lem}

\begin{proof}
Multiplying the second equation of \eqref{sys-regul} by $-\Delta v$ and integrating over $\Omega$ yield that
\begin{align}\label{1006-1659}
\frac{1}{2} \frac{d}{d t} \int_{\Omega}|\nabla v_{\varepsilon}|^2+\int_{\Omega}|\Delta v_{\varepsilon}|^2 + \int_{\Omega}u_{\varepsilon} | \nabla v_{\varepsilon }|^2  =- \int_{\Omega}v_{\varepsilon} \nabla u_{\varepsilon }\cdot \nabla v_{\varepsilon }
\end{align}
for all $t \in(0, T_{\max, \varepsilon})$ and $\varepsilon \in(0,1)$.

Case I: $1<l<2$. We multiply the first equation in \eqref{sys-regul} by $u^{1-l}_{\varepsilon}$ and  integrate by parts to obtain
\begin{align}\label{1006-2240}
- \dt \int_{\Omega} u^{2-l}_{\varepsilon} 
= & -(2-l)\int_{\Omega} u^{1-l}_{\varepsilon} \cdot \left\{\nabla \cdot \left(u^{l-1}_{\varepsilon} v_{\varepsilon} \nabla u_{\varepsilon }-u^{l}_{\varepsilon} v_{\varepsilon} \nabla v_{\varepsilon }\right)+ u_{\varepsilon }  v_{\varepsilon }\right\} \nonumber\\
= &  -(2-l)(l-1)\int_{\Omega} \frac{v_{\varepsilon}}{u_{\varepsilon}}|\nabla u_{\varepsilon }|^2+(2-l)(l-1) \int_{\Omega} v_{\varepsilon} \nabla u_{\varepsilon }\cdot \nabla v_{\varepsilon }-(2-l)\int_{\Omega}u^{2-l}_{\varepsilon} v_{\varepsilon}\nonumber\\
\leqslant & -(2-l)(l-1)\int_{\Omega} \frac{v_{\varepsilon}}{u_{\varepsilon}}|\nabla u_{\varepsilon }|^2+(2-l)(l-1) \int_{\Omega} v_{\varepsilon} \nabla u_{\varepsilon }\cdot \nabla v_{\varepsilon }
\end{align}
for all $t \in(0, T_{\max, \varepsilon})$ and $\varepsilon \in(0,1)$. Combining \eqref{1006-1659} with \eqref{1006-2240}, we conclude that
\begin{align*}
\dt \left\{-\int_{\Omega} u^{2-l}_{\varepsilon} +  \frac{(2-l)(l-1)}{2} \int_{\Omega} |\nabla v_{\varepsilon }|^2\right\} + (2-l)(l-1)\int_{\Omega} \frac{v_{\varepsilon}}{u_{\varepsilon}}|\nabla u_{\varepsilon }|^2
\leqslant  0 
\end{align*}
for all $t \in(0, T_{\max, \varepsilon})$ and $\varepsilon \in(0,1)$. Integrating the above differential equality on $(0,t)$, we have
\begin{align}\label{1006-2249}
-\int_{\Omega} u^{2-l}_{\varepsilon}(t)+ \frac{(2-l)(l-1)}{2}  \int_{\Omega} |\nabla v_{\varepsilon }(t)|^2 & +(2-l)(l-1) \int_0^{t} \int_{\Omega} \frac{v_{\varepsilon}}{u_{\varepsilon}}|\nabla u_{\varepsilon }|^2 \nonumber \\  
& \leqslant - \int_{\Omega} (u_0+\varepsilon)^{2-l}+\frac{(2-l)(l-1)}{2} \int_{\Omega} | \nabla v_{0 }|^2 \nonumber \\
& \leqslant \frac{(2-l)(l-1)}{2} \int_{\Omega} | \nabla v_{0 }|^2
\end{align}
for all $t \in(0, T_{\max, \varepsilon})$ and $\varepsilon \in(0,1)$. In view of \eqref{-2.8} and Young's inequality, we see that
\begin{align}\label{0914-1042}
\int_{\Omega} u^{2-l}_{\varepsilon}(t) \leqslant \int_{\Omega} u_{\varepsilon}(t) + |\Omega| \leqslant \int_{\Omega} (u_0+2)+ \int_{\Omega} v_0
\end{align}
for all $t \in(0, T_{\max, \varepsilon})$ and $\varepsilon \in(0,1)$. Owing to \eqref{1006-2249}, we infer that
\begin{align*}
\frac{(2-l)(l-1)}{2}  \int_{\Omega} |\nabla v_{\varepsilon }(t)|^2 & +(2-l)(l-1) \int_0^{t} \int_{\Omega} \frac{v_{\varepsilon}}{u_{\varepsilon}}|\nabla u_{\varepsilon }|^2 \nonumber \\ \leqslant & \frac{(2-l)(l-1)}{2} \int_{\Omega} | \nabla v_{0 }|^2 + \int_{\Omega} (u_0+2)+ \int_{\Omega} v_0
\end{align*}
for all $t \in(0, T_{\max, \varepsilon})$ and $\varepsilon \in(0,1)$, which implies \eqref{-3.5aa}.

Case II: $l=2$. We multiply the first equation in \eqref{sys-regul} by $-\frac{1}{u_{\varepsilon}}$ and  integrate by parts to obtain
\begin{align}\label{1006-1700}
- \dt \int_{\Omega} \ln u_{\varepsilon}
= & -\int_{\Omega} \frac{1}{u_{\varepsilon}} \cdot 
  \left\{\nabla \cdot \left(u_{\varepsilon} v_{\varepsilon} \nabla u_{\varepsilon }- u_{\varepsilon}^2 v_{\varepsilon} \nabla v_{\varepsilon }\right)+ u_{\varepsilon} v_{\varepsilon}\right\} \nonumber\\
= & - \int_{\Omega}\frac{v_{\varepsilon}}{u_{\varepsilon}} |\nabla u_{\varepsilon }|^2 +\int_{\Omega} 
  v_{\varepsilon} \nabla u_{\varepsilon }\cdot \nabla v_{\varepsilon }-  \int_{\Omega} v_{\varepsilon}
\end{align}
for all $t \in(0, T_{\max, \varepsilon})$ and $\varepsilon \in(0,1)$. Combining \eqref{1006-1659} with \eqref{1006-1700}, we conclude that
\begin{align*}
\dt\left\{-\int_{\Omega} \ln u_{\varepsilon} +  
  \frac{1}{2} \int_{\Omega} |\nabla v_{\varepsilon }|^2\right\} + \int_{\Omega} \frac{v_{\varepsilon}}{u_{\varepsilon}} | \nabla u_{\varepsilon }|^2 
  +  \int_{\Omega} v_{\varepsilon}
  + \int_{\Omega}  |\Delta v_{\varepsilon }|^2
  + \int_{\Omega} u_{\varepsilon} | \nabla v_{\varepsilon }|^2=0 
\end{align*}
for all $t \in(0, T_{\max, \varepsilon})$ and $\varepsilon \in(0,1)$. Integrating the above differential equality on $(0,t)$ implies that
\begin{align}\label{1006-1728}
& -\int_{\Omega} \ln u_{\varepsilon}(t) 
   +\frac{1}{2} \int_{\Omega} |\nabla v_{\varepsilon}(t)|^2 
   + \int_0^{t} \int_{\Omega} \left(\frac{v_{\varepsilon}}{u_{\varepsilon}} | \nabla u_{\varepsilon }|^2 + v_{\varepsilon} +  |\Delta v_{\varepsilon }|^2 + u_{\varepsilon} | \nabla v_{\varepsilon }|^2\right) \nonumber \\ 
   = & - \int_{\Omega} \ln (u_0+\varepsilon)+\frac{1}{2} \int_{\Omega} | \nabla v_{0 }|^2
\end{align}
for all $t \in(0, T_{\max, \varepsilon})$ and $\varepsilon \in(0,1)$. In view of \eqref{-2.8} and $\ln \xi \leqslant \xi$ for all $\xi>0$, we see that
\begin{align}\label{0914-1042}
\int_{\Omega} \ln u_{\varepsilon}(t) \leqslant \int_{\Omega} u_{\varepsilon}(t) \leqslant \int_{\Omega} (u_0+1)+ \int_{\Omega} v_0
\end{align}
for all $t \in(0, T_{\max, \varepsilon})$ and $\varepsilon \in(0,1)$. Owing to \eqref{1006-1728}, we infer that
\begin{align*}
&\frac{1}{2} \int_{\Omega} |\nabla v_{\varepsilon}(t)|^2 + \int_0^{t} \int_{\Omega} \left(\frac{v_{\varepsilon}}{u_{\varepsilon}} | \nabla u_{\varepsilon }|^2 + v_{\varepsilon} +  |\Delta v_{\varepsilon }|^2 +  u_{\varepsilon} | \nabla v_{\varepsilon }|^2\right) \\
\leqslant & -\int_{\Omega} \ln u_0 + \int_{\Omega} \ln u_{\varepsilon}(t) +\frac{1}{2} \int_{\Omega} | \nabla v_{0 }|^2 \\
\leqslant & -\int_{\Omega} \ln u_0 + \int_{\Omega} (u_0+1)+ \int_{\Omega} v_0 +\frac{1}{2} \int_{\Omega} | \nabla v_{0 }|^2 
\end{align*}
for all $t \in(0, T_{\max, \varepsilon})$ and $\varepsilon \in(0,1)$, which implies \eqref{-3.5aa}.

Case III: $l>2$. We multiply the first equation in \eqref{sys-regul} by $u^{1-l}_{\varepsilon}$ and  integrate by parts to obtain
\begin{align}\label{1006-2201}
\dt \int_{\Omega} u^{2-l}_{\varepsilon} 
= & (2-l)\int_{\Omega} u^{1-l}_{\varepsilon} \cdot \left\{\nabla \cdot \left(u^{l-1}_{\varepsilon} v_{\varepsilon} \nabla u_{\varepsilon }-u^{l}_{\varepsilon} v_{\varepsilon} \nabla v_{\varepsilon }\right)+ u_{\varepsilon }  v_{\varepsilon }\right\} \nonumber\\
= &  -(l-2)(l-1)\int_{\Omega} \frac{v_{\varepsilon}}{u_{\varepsilon}}|\nabla u_{\varepsilon }|^2+(l-2)(l-1) \int_{\Omega} v_{\varepsilon} \nabla u_{\varepsilon }\cdot \nabla v_{\varepsilon }-(l-2)\int_{\Omega}u^{2-l}_{\varepsilon} v_{\varepsilon}\nonumber\\
\leqslant & -(l-2)(l-1)\int_{\Omega} \frac{v_{\varepsilon}}{u_{\varepsilon}}|\nabla u_{\varepsilon }|^2+(l-2)(l-1) \int_{\Omega} v_{\varepsilon} \nabla u_{\varepsilon }\cdot \nabla v_{\varepsilon }
\end{align}
for all $t \in(0, T_{\max, \varepsilon})$ and $\varepsilon \in(0,1)$. Combining \eqref{1006-1659} and \eqref{1006-2201}, we conclude that
\begin{align*}
\dt \left\{\int_{\Omega} u^{2-l}_{\varepsilon} +  \frac{(l-2)(l-1)}{2} \int_{\Omega} |\nabla v_{\varepsilon }|^2\right\} + (l-2)(l-1)\int_{\Omega} \frac{v_{\varepsilon}}{u_{\varepsilon}}|\nabla u_{\varepsilon }|^2
\leqslant  0 
\end{align*}
for all $t \in(0, T_{\max, \varepsilon})$ and $\varepsilon \in(0,1)$. Integrating the above differential equality on $(0,t)$, we have
\begin{align*}
\int_{\Omega} u^{2-l}_{\varepsilon}(t)+ \frac{(l-2)(l-1)}{2}  \int_{\Omega} |\nabla v_{\varepsilon }(t)|^2 & +(l-2)(l-1) \int_0^{t} \int_{\Omega} \frac{v_{\varepsilon}}{u_{\varepsilon}}|\nabla u_{\varepsilon }|^2 \nonumber \\  
& \leqslant \int_{\Omega} u_0^{2-l}+\frac{(l-2)(l-1)}{2} \int_{\Omega} | \nabla v_{0 }|^2 
\end{align*}
for all $t \in(0, T_{\max, \varepsilon})$ and $\varepsilon \in(0,1)$, which implies \eqref{-3.5aa}.
\end{proof}

In \cite{2022-CVPDE-Winkler}, Winkler derived a uniform estimate of the temporal-spatial integral of $\frac{u_{\varepsilon}}{v_{\varepsilon}}|\nabla  v_{\varepsilon }|^2$ for a doubly degenerate reaction-diffusion system. Based on Lemma \ref{lem-1st-est}, using the similar method, we also obtain such estimate for the system \eqref{sys-regul}.


\begin{lem}\label{lem-2nd-est}
Let $l>1$ and assume that \eqref{assIniVal} holds.
Then we have 
\begin{align}\label{-3.9}
\int_0^{T_{\max, \varepsilon}}\r\int_{\Omega} \frac{u_{\varepsilon}}{v_{\varepsilon}}|\nabla  v_{\varepsilon }|^2 \leqslant C \quad \text { for all } \varepsilon \in(0,1)
\end{align}
and
\begin{align}\label{-3.10}
\int_0^{T_{\max, \varepsilon}}\r\int_{\Omega} \frac{|\nabla v_{\varepsilon }|^4}{v_{\varepsilon}^3} \leqslant C \quad \text { for all } \varepsilon \in(0,1),
\end{align}
where $C$ is a positive constant depending on $\int_{\Omega} u_0$, $\int_{\Omega} v_0$, $\int_{\Omega} |\nabla v_{0 }|^2$ and $\int_{\Omega} \frac{|\nabla v_{0 }|^2}{v_0}$, but independent of $\varepsilon$. 
\end{lem}


\begin{proof}
Since $v_{\varepsilon}$ is positive in $\overline{\Omega} \times(0, T_{\max, \varepsilon})$ and actually belongs to $C^3(\overline{\Omega} \times(0, T_{\max, \varepsilon}))$ thanks to standard parabolic regularity theory, using the second equation in the system \eqref{sys-regul} we compute
\begin{align}\label{-3.11}
\frac{1}{2} \dt \int_{\Omega} \frac{\left|\nabla v_{\varepsilon}\right|^2}{v_{\varepsilon}}
= & \int_{\Omega} \frac{1}{v_{\varepsilon}} \nabla v_{\varepsilon} \cdot \nabla\left\{\Delta v_{\varepsilon}
    -u_{\varepsilon} v_{\varepsilon}\right\}
    -\frac{1}{2} \int_{\Omega} \frac{1}{v_{\varepsilon}^2}\left|\nabla v_{\varepsilon}\right|^2 \cdot\left\{\Delta v_{\varepsilon}-u_{\varepsilon} v_{\varepsilon}\right\} \nonumber\\
= & \int_{\Omega} \frac{1}{v_{\varepsilon}} \nabla v_{\varepsilon} \cdot \nabla\Delta v_{\varepsilon}
    -\int_{\Omega} \frac{u_{\varepsilon}}{v_{\varepsilon}}\left|\nabla v_{\varepsilon}\right|^2
    -\int_{\Omega} \nabla u_{\varepsilon} \cdot \nabla v_{\varepsilon}\nonumber\\
  & -\frac{1}{2} \int_{\Omega} \frac{1}{v_{\varepsilon}^2}\left|\nabla v_{\varepsilon}\right|^2 \cdot\Delta v_{\varepsilon}
    +\frac{1}{2} \int_{\Omega} \frac{u_{\varepsilon}}{v_{\varepsilon}}\left|\nabla v_{\varepsilon}\right|^2\nonumber\\
= & \int_{\Omega} \frac{1}{v_{\varepsilon}} \nabla v_{\varepsilon} \cdot \nabla\Delta v_{\varepsilon}
    -\frac{1}{2} \int_{\Omega} \frac{1}{v_{\varepsilon}^2}\left|\nabla v_{\varepsilon}\right|^2 \cdot\Delta v_{\varepsilon}\nonumber\\
  & -\frac{1}{2} \int_{\Omega} \frac{u_{\varepsilon}}{v_{\varepsilon}}\left|\nabla v_{\varepsilon}\right|^2
    -\int_{\Omega} \nabla u_{\varepsilon} \cdot \nabla v_{\varepsilon}
\end{align}
for all $t \in\left(0, T_{\max, \varepsilon}\right)$ and $\varepsilon \in(0,1)$. 
Based on \cite[Lemma~3.2]{2012-CPDE-Winkler}, we have the integral identity
\begin{align*}
\int_{\Omega} \frac{1}{v_{\varepsilon}} \nabla v_{\varepsilon} \cdot \nabla \Delta v_{\varepsilon}
-\frac{1}{2} \int_{\Omega} \frac{1}{v_{\varepsilon}^2}|\nabla v_{\varepsilon}|^2 \Delta v_{\varepsilon}
= -\int_{\Omega} v_{\varepsilon}\left|D^2 \ln v_{\varepsilon}\right|^2
  +\frac{1}{2} \int_{\partial \Omega} \frac{1}{v_{\varepsilon}} \frac{\partial|\nabla v_{\varepsilon}|^2}{\partial \nu}.
\end{align*}
Therefore, \eqref{-3.11} becomes
\begin{align}\label{-3.11-2a}
\frac{1}{2} \dt \int_{\Omega} \frac{\left|\nabla v_{\varepsilon}\right|^2}{v_{\varepsilon}}
  +\int_{\Omega} v_{\varepsilon}\left|D^2 \ln v_{\varepsilon}\right|^2
   & +\frac{1}{2} \int_{\Omega} \frac{u_{\varepsilon}}{v_{\varepsilon}}\left|\nabla v_{\varepsilon}\right|^2\nonumber\\
   &=  \frac{1}{2} \int_{\partial \Omega} \frac{1}{v_{\varepsilon}} \frac{\partial|\nabla v_{\varepsilon}|^2}{\partial \nu}
    -\int_{\Omega} \nabla u_{\varepsilon} \cdot \nabla v_{\varepsilon}
\end{align}
for all $t \in\left(0, T_{\max, \varepsilon}\right)$ and $\varepsilon \in(0,1)$. By \cite[Lemma 3.3]{2012-CPDE-Winkler}, there holds
\begin{equation}\label{-3.12}
\int_{\Omega} v_{\varepsilon}\left|D^2 \ln v_{\varepsilon}\right|^2 
\geqslant \frac{c_1}{2}  \int_{\Omega} \frac{\left|\nabla v_{\varepsilon}\right|^4}{v_{\varepsilon}^3}
\end{equation}
for some $c_1>0$.
By Young’s inequality, we infer that 
\begin{align}\label{-3.13a}
-\int_{\Omega} \nabla u_{\varepsilon} \cdot \nabla v_{\varepsilon}& \leqslant   \frac{1}{4}\int_{\Omega} \frac{u_{\varepsilon}}{v_{\varepsilon}}|\nabla v_{\varepsilon }|^2
+\int_{\Omega} \frac{v_{\varepsilon}}{u_{\varepsilon}} |\nabla u_{\varepsilon }|^2. 
\end{align}
From \eqref{-3.11-2a}-\eqref{-3.13a} and using that $\frac{\partial|\nabla \varphi|^2}{\partial v} \leqslant 0$ on $\partial \Omega$ by convexity of $\Omega$ (cf. \cite{1980-ARMA-Lions}), we obtain 
\begin{align}\label{-3.14}
\dt \int_{\Omega} \frac{|\nabla v_{\varepsilon}|^2}{v_{\varepsilon}}
+c_1 \int_{\Omega} \frac{|\nabla v_{\varepsilon}|^4}{v_{\varepsilon}^3}
+ \frac{1}{2}\int_{\Omega} \frac{u_{\varepsilon}}{v_{\varepsilon}} |\nabla v_{\varepsilon }|^2 
\leqslant   2 \int_{\Omega} \frac{v_{\varepsilon}}{u_{\varepsilon}} |\nabla u_{\varepsilon }|^2 
\end{align}
for all $t \in(0, T_{\max, \varepsilon})$ and $\varepsilon \in(0,1)$.
Integrating \eqref{-3.14} on $(0,t)$, we show that
\begin{align*}
\int_{\Omega} \frac{|\nabla v_{\varepsilon}(t)|^2}{v_{\varepsilon}(t)}
  &+c_1 \int_0^{t} \int_{\Omega} \frac{|\nabla v_{\varepsilon}|^4}{v_{\varepsilon}^3}
  +\frac{1}{2} \int_0^{t} \int_{\Omega} \frac{u_{\varepsilon}}{v_{\varepsilon}} |\nabla v_{\varepsilon }|^2 \\
&\leqslant \int_{\Omega} \frac{|\nabla v_{0 }|^2}{v_0} 
  + 2 \int_0^{t} \int_{\Omega} \frac{v_{\varepsilon}}{u_{\varepsilon}}|\nabla  u_{\varepsilon }|^2 
\end{align*}
for all $t \in(0, T_{\max, \varepsilon})$ and $\varepsilon \in(0,1)$, which yields \eqref{-3.9} and \eqref{-3.10} thanks to \eqref{-3.5aa}.
\end{proof}

Using Lemma \ref{lemma-3.4} with $p=1$ and the estimates obtained in Lemmas \ref{lem-1st-est} and \ref{lem-2nd-est}, we can establish the following a temporal-spatial $L^{2}$ bound for $u_{\varepsilon}$, weighted by the factor $v_{\varepsilon}$.

\begin{lem}\label{lemma-3.7}
Let $l>1$ and assume that \eqref{assIniVal} holds.
Then we have 
\begin{align}\label{-3.25xxx}
\int_0^{T_{\max, \varepsilon}}\r \int_{\Omega} u^2_{\varepsilon} v_{\varepsilon} \leqslant C\quad \text { for all } \varepsilon \in(0,1),
\end{align}
where $C$ is a positive constant depending on $-\int_{\Omega} \ln u_0$, $\int_{\Omega} u_0$, $\int_{\Omega} v_0$, $\int_{\Omega} |\nabla v_{0 }|^2$ and $\int_{\Omega} \frac{|\nabla v_{0 }|^2}{v_0}$, but independent of $\varepsilon$. 
\end{lem}
\begin{proof}
Letting $p=1$ in \eqref{eq-6.1}, we have
\begin{align*}
\int_{\Omega} u_{\varepsilon}^{2} v_{\varepsilon} 
 \leqslant c \left\{\int_{\Omega} \frac{u_{\varepsilon}}{v_{\varepsilon}}|\nabla v_{\varepsilon}|^2
  +\int_{\Omega} \frac{v_{\varepsilon}}{u_{\varepsilon}}|\nabla u_{\varepsilon}|^2
  +\int_{\Omega} u_{\varepsilon} v_{\varepsilon}  \right\} \cdot \int_{\Omega} u_{\varepsilon} 
\end{align*}
for some $c>0$. Summing up \eqref{-2.10}, \eqref{-3.5aa} and \eqref{-3.9}, we infer that
\begin{align*}
\int_0^{T_{\max, \varepsilon}}\r  \left\{\int_{\Omega} \frac{u_{\varepsilon}}{v_{\varepsilon}}|\nabla v_{\varepsilon}|^2
+\int_{\Omega} \frac{v_{\varepsilon}}{u_{\varepsilon}}|\nabla u_{\varepsilon}|^2
+\int_{\Omega} u_{\varepsilon} v_{\varepsilon}\right\} \leqslant C  \quad \text { for all } \varepsilon \in(0,1)
\end{align*}
with some $C>0$ independent of $\varepsilon$. From \eqref{-2.8}, there holds
\begin{align*}
\int_{\Omega} u_{\varepsilon}(t) \leqslant \int_{\Omega}\left(u_0+1\right)+ \int_{\Omega} v_0, \quad \text { for all } t \in(0, T_{\max, \varepsilon}) \text { and } \varepsilon \in(0,1).
\end{align*}
Thus, we complete the proof of \eqref{-3.25xxx}.
\end{proof}


\section{The Two-Dimensional Case}\label{sect-3}
\subsection{Uniform estimates in Lebesgue spaces}\label{sect-3xx}
We first derive a uniform estimate for $\|u_\varepsilon(t)\|_{L^p(\Omega)}$ independent of $\varepsilon$ for any $p > 1$ in $(0,T_{\max, \varepsilon})$. 
We first deduce a differential inequality for $\int_\Omega u^p_{\varepsilon}(t)$ using Lemma \ref{lemma-3.5}.

\begin{lem}\label{lemma-3.9xx}
Let $l > 1$ and assume that \eqref{assIniVal} holds.
Then for all $p \geqslant 2$ we have
\begin{small} 
\begin{align}\label{0704-0024}
\dt \int_{\Omega} u_{\varepsilon}^p 
+ \frac{p(p-1)}{4} \int_{\Omega} u_{\varepsilon}^{p+l-3} v_{\varepsilon} \left|\nabla 
u_{\varepsilon}\right|^2  
\leqslant &   A \int_{\Omega} u_{\varepsilon}^{p+l-1} v_{\varepsilon}  \cdot \int_\Omega \frac{|\nabla v_{\varepsilon}|^4}{v_{\varepsilon}^3} \nonumber\\
& + A\left\{\int_{\Omega} \frac{u_{\varepsilon}}{v_{\varepsilon}}|\nabla v_{\varepsilon}|^2
  +\int_{\Omega} \frac{v_{\varepsilon}}{u_{\varepsilon}}|\nabla u_{\varepsilon}|^2
  +\int_{\Omega} u_{\varepsilon} v_{\varepsilon}  \right\} \cdot \int_{\Omega} u_{\varepsilon}^p \nonumber\\
& + A \left\{\int_\Omega u_{\varepsilon} \right\}^{2 p+2l-3} \cdot \int_\Omega \frac{|\nabla v_{\varepsilon}|^4}{v_{\varepsilon}^3}+ A \int_{\Omega} u_{\varepsilon} v_{\varepsilon} 
\end{align}
\end{small}
where $A=\max \left\{c_1\left\|v_{0} \right\|_{L^{\infty}(\Omega)}+2 c_1 \left\|v_{0} \right\|^3_{L^{\infty}(\Omega)}, c_1\left\|v_{0} \right\|^2_{L^{\infty}(\Omega)}+p, c_2 p \right\}$ with some constants $c_1=c_1(p)>0$ and $c_2=c_2(p)>0$.
\end{lem}
\begin{proof}
We multiply the first equation of \eqref{sys-regul} by $u_{\varepsilon}^{p-1}$, integrate by parts and use Young's inequality to obtain 
\begin{align}\label{0703-1903xx}
 \dt \int_{\Omega} u_{\varepsilon}^p
= &\ p \int_{\Omega} u_{\varepsilon}^{p-1}\left\{\nabla \cdot(u^{l-1}_{\varepsilon} v_{\varepsilon} \nabla u_{\varepsilon})- \nabla \cdot\left(u_{\varepsilon}^{l} v_{\varepsilon} \nabla v_{\varepsilon}\right)+ u_{\varepsilon} v_{\varepsilon}\right\} \nonumber\\
= &-(p-1)p \int_{\Omega} u_{\varepsilon}^{p+l-3} v_{\varepsilon}\left|\nabla u_{\varepsilon}\right|^2+ (p-1)p \int_{\Omega} u_{\varepsilon}^{p+l-2} v_{\varepsilon} \nabla u_{\varepsilon} \cdot \nabla v_{\varepsilon} + p \int_{\Omega} u_{\varepsilon}^{p} v_{\varepsilon} \nonumber\\
\leqslant & -\frac{(p-1)p}{2} \int_{\Omega} u_{\varepsilon}^{p+l-3} v_{\varepsilon} \left|\nabla u_{\varepsilon}\right|^2 +\frac{(p-1)p}{2} \int_{\Omega} u_{\varepsilon}^{p+l-1} v_{\varepsilon}\left|\nabla v_{\varepsilon}\right|^2 \nonumber\\
& + p \int_{\Omega} u_{\varepsilon}^{p+1}v_{\varepsilon} + p \int_{\Omega} u_{\varepsilon} v_{\varepsilon} \quad \text { for all } t \in(0, T_{\max, \varepsilon}) \text { and } \varepsilon \in(0,1).
\end{align}
Letting $\eta =\frac{1}{2}$ in \eqref{eq-6.4}  and using \eqref{-2.9}, we have
\begin{align*}
& \frac{p(p-1)}{2} \int_{\Omega} u_{\varepsilon}^{p+l-1} v_{\varepsilon}\left|\nabla v_{\varepsilon}\right|^2 \nonumber\\
\leqslant &    \frac{p(p-1) }{4} \int_\Omega u_{\varepsilon}^{p+l-3} v_{\varepsilon}|\nabla u_{\varepsilon}|^2 +  \left\{ c_1   \left\|v_{0} \right\|_{L^{\infty}(\Omega)}+2c_1  \left\|v_{0} \right\|^3_{L^{\infty}(\Omega)}\right\} \cdot \int_{\Omega} u_{\varepsilon}^{p+l-1} v_{\varepsilon}  \cdot \int_\Omega \frac{|\nabla v_{\varepsilon}|^4}{v_{\varepsilon}^3} \nonumber\\
&+c_1  \left\|v_{0} \right\|^2_{L^{\infty}(\Omega)} \cdot\left\{\int_\Omega u_{\varepsilon} \right\}^{2 p+2l-3} \cdot \int_\Omega \frac{|\nabla v_{\varepsilon}|^4}{v_{\varepsilon}^3} +  c_1   \left\|v_{0} \right\|^2_{L^{\infty}(\Omega)} \cdot \int_{\Omega} u_{\varepsilon} v_{\varepsilon}
\end{align*}
for some $c_1=c_1(p)>0$. This together with \eqref{0703-1903xx} yields 
\begin{align}\label{jia-2}
\dt \int_{\Omega} u_{\varepsilon}^p 
 + & \frac{p(p-1)}{4} \int_{\Omega} u_{\varepsilon}^{p+l-3} v_{\varepsilon} \left|\nabla 
u_{\varepsilon}\right|^2  \nonumber\\
\leqslant &   \left\{ c_1   \left\|v_{0} \right\|_{L^{\infty}(\Omega)}+2c_1 \left\|v_{0} \right\|^3_{L^{\infty}(\Omega)}\right\} \cdot \int_{\Omega} u_{\varepsilon}^{p+l-1} v_{\varepsilon}  \cdot \int_\Omega \frac{|\nabla v_{\varepsilon}|^4}{v_{\varepsilon}^3} \nonumber\\
& +  c_1 \left\|v_{0} \right\|^2_{L^{\infty}(\Omega)} \cdot \left\{\int_\Omega u_{\varepsilon} \right\}^{2 p+2l-3} \cdot \int_\Omega \frac{|\nabla v_{\varepsilon}|^4}{v_{\varepsilon}^3}\nonumber\\
& + p \int_{\Omega} u_{\varepsilon}^{p+1}  v _{\varepsilon} +  \left\{c_1 \left\|v_{0} \right\|^2_{L^{\infty}(\Omega)}+p \right\} \cdot \int_{\Omega} u_{\varepsilon} v_{\varepsilon} 
\end{align}
for all $t \in(0, T_{\max, \varepsilon})$ and $\varepsilon \in(0,1)$. It follows from \eqref{eq-6.1} that
\begin{align}\label{jia-1}
\int_{\Omega} u_{\varepsilon}^{p+1} v_{\varepsilon} 
 \leqslant c_2 \left\{\int_{\Omega} \frac{u_{\varepsilon}}{v_{\varepsilon}}|\nabla v_{\varepsilon}|^2
  +\int_{\Omega} \frac{v_{\varepsilon}}{u_{\varepsilon}}|\nabla u_{\varepsilon}|^2
  +\int_{\Omega} u_{\varepsilon} v_{\varepsilon}  \right\} \cdot \int_{\Omega} u_{\varepsilon}^p
\end{align}
for some $c_2=c_2(p)>0$. 
Together with \eqref{jia-1}, \eqref{jia-2} yields the desired \eqref{0704-0024}.
\end{proof}

From \eqref{-2.10}, \eqref{-3.5aa} and \eqref{-3.9}, we next know that $\int_{\Omega} \frac{u_{\varepsilon}}{v_{\varepsilon}}|\nabla v_{\varepsilon}|^2+\int_{\Omega} \frac{v_{\varepsilon}}{u_{\varepsilon}}|\nabla u_{\varepsilon}|^2+ \int_{\Omega} u_{\varepsilon} v_{\varepsilon}\in L^1(0,T_{\max, \varepsilon})$, and from \eqref{-3.10}, we also know that $\int_\Omega \frac{|\nabla v_{\varepsilon}|^4}{v_{\varepsilon}^3}\in L^1(0,T_{\max, \varepsilon})$. However, this is not sufficient to solve the differential inequality \eqref{0704-0024} for $\int_{\Omega} u_{\varepsilon}^p$. To overcome this difficulty, we shall prove $\int_\Omega \frac{|\nabla v_{\varepsilon}|^4}{v_{\varepsilon}^3}\in L^\infty(0,T_{\max, \varepsilon})$. To this aim, we derive a differential inequality for the following energy-like functional $G_{\varepsilon}(t)$ and the estimate of $\int_{\Omega} u_{\varepsilon}^{p+l-1} v_{\varepsilon}$ (cf. \cite[Lemma 6.1]{2024-JDE-Winkler}).

\begin{lem}\label{lemma-3.6}
Let $l > 1$ and assume that \eqref{assIniVal} holds, and let
\begin{align*}
G_{\varepsilon}(t) =\left\{\begin{array}{lll}
\frac{4b}{(l-3)(l-2)} \int_{\Omega} u_{\varepsilon}^{3-l}+\int_{\Omega} \frac{\left|\nabla v_{\varepsilon}\right|^4}{v_{\varepsilon}^3} & \text { when } & 1 < l<2  \text {   or   }  l>3, \\
-\frac{4b}{(3-l)(l-2)} \int_{\Omega} u_{\varepsilon}^{3-l}+\int_{\Omega} \frac{\left|\nabla v_{\varepsilon}\right|^4}{v_{\varepsilon}^3} & \text { when } & 2<l<3,\\
4b \int_{\Omega} u_{\varepsilon} \ln u_{\varepsilon}+\int_{\Omega} \frac{\left|\nabla v_{\varepsilon}\right|^4}{v_{\varepsilon}^3} & \text { when } & l=2, \\
-4b \int_{\Omega} \ln u_{\varepsilon}+\int_{\Omega} \frac{\left|\nabla v_{\varepsilon}\right|^4}{v_{\varepsilon}^3} & \text { when } & l=3. 
\end{array}\right.
\end{align*}
Then we have
\begin{align}\label{0922-1558}
G_{\varepsilon}^{\prime}(t) +2 \int_{\Omega} v_{\varepsilon}^{-1}\left|\nabla v_{\varepsilon}\right|^2\left|D^2 \ln v_{\varepsilon}\right|^2+ b \int_{\Omega} v_{\varepsilon}\left|\nabla u_{\varepsilon}\right|^2 
\leqslant &  4b  \int_{\Omega} u_{\varepsilon}^{2} v_{\varepsilon}\left|\nabla v_{\varepsilon}\right|^2  -\frac{4 b}{l-2}  \int_{\Omega} u^{3-l}_{\varepsilon} v_{\varepsilon}   \nonumber\\  
 & \quad \quad \text { when }  1 < l<2  \text {   or   }    l>3
\end{align}
and
\begin{align}\label{0922-1561}
G_{\varepsilon}^{\prime}(t)+2 \int_{\Omega} v_{\varepsilon}^{-1}\left|\nabla v_{\varepsilon}\right|^2\left|D^2 \ln v_{\varepsilon}\right|^2 + b \int_{\Omega} v_{\varepsilon}\left|\nabla u_{\varepsilon}\right|^2 
& \leqslant   4b  \int_{\Omega} u_{\varepsilon}^{2} v_{\varepsilon}\left|\nabla v_{\varepsilon}\right|^2  \nonumber\\  
& \text { when } 2<l<3
\end{align}
and
\begin{align}\label{0922-1559}
G_{\varepsilon}^{\prime}(t) +2 \int_{\Omega} v_{\varepsilon}^{-1}\left|\nabla v_{\varepsilon}\right|^2\left|D^2 \ln v_{\varepsilon}\right|^2 + b \int_{\Omega} v_{\varepsilon}\left|\nabla u_{\varepsilon}\right|^2 
\leqslant &  4b \int_{\Omega} u_{\varepsilon}^{2} v_{\varepsilon}\left|\nabla v_{\varepsilon}\right|^2 + 4b \int_{\Omega} u_{\varepsilon} v_{\varepsilon}  \nonumber\\ 
& +   4b \int_{\Omega} u_{\varepsilon}^{2} v_{\varepsilon} \quad \text { when } l=2
\end{align}
as well as
\begin{align}\label{0922-1560}
G_{\varepsilon}^{\prime}(t) +2 \int_{\Omega} v_{\varepsilon}^{-1}\left|\nabla v_{\varepsilon}\right|^2\left|D^2 \ln v_{\varepsilon}\right|^2 + b \int_{\Omega} v_{\varepsilon}\left|\nabla u_{\varepsilon}\right|^2 
\leqslant &  4b \int_{\Omega} u_{\varepsilon}^{2} v_{\varepsilon}\left|\nabla v_{\varepsilon}\right|^2   \text { when } l=3
\end{align}
for all $t \in(0, T_{\max, \varepsilon})$ and $\varepsilon \in(0,1)$, where $b$ and $c$ are some positive constants independent of $\varepsilon$.
\end{lem}

\begin{proof}
Based on \cite[Lemma 3.4]{2022-DCDSSB-Winkler}, we see that there exists a constant $b>0$ such that
\begin{align}\label{-3.21a}
\int_{\Omega} \frac{|\nabla \phi|^6}{\phi^5} \leqslant b \int_{\Omega} \phi^{-1}|\nabla \phi|^2\left|D^2 \ln \phi\right|^2.
\end{align}
In view of \cite[Lemma 2.3]{2022-JDE-Li} and $\frac{\partial|\nabla \varphi|^2}{\partial v} \leqslant 0$ on $\partial \Omega$ by convexity of $\Omega$ (cf. \cite{1980-ARMA-Lions}), we obtain 
\begin{align}\label{-3.22a}
\frac{d}{d t} \int_{\Omega} \frac{\left|\nabla v_{\varepsilon}\right|^4}{v_{\varepsilon}^3} +  4 \int_{\Omega} v_{\varepsilon}^{-1}\left|\nabla v_{\varepsilon}\right|^2\left|D^2 \ln v_{\varepsilon}\right|^2  
\leqslant  -4 \int_{\Omega} v_{\varepsilon}^{-2}\left|\nabla v_{\varepsilon}\right|^2\left(\nabla u_{\varepsilon} \cdot \nabla v_{\varepsilon}\right)
\end{align}
for all $t \in(0, T_{\max, \varepsilon})$ and $\varepsilon \in(0,1)$. Invoking Young's inequality and \eqref{-3.21a}, we deduce that
\begin{align}\label{-3.23a}
-4 \int_{\Omega} v_{\varepsilon}^{-2}\left|\nabla v_{\varepsilon}\right|^2\left(\nabla u_{\varepsilon} \cdot \nabla v_{\varepsilon}\right) & \leqslant \frac{2}{b} \int_{\Omega} \frac{\left|\nabla v_{\varepsilon}\right|^6}{v_{\varepsilon}^5}+2 b \int_{\Omega} v_{\varepsilon}\left|\nabla u_{\varepsilon}\right|^2 \nonumber\\
& \leqslant 2 \int_{\Omega} v_{\varepsilon}^{-1}\left|\nabla v_{\varepsilon}\right|^2\left|D^2 \ln v_{\varepsilon}\right|^2 +2 b \int_{\Omega} v_{\varepsilon}\left|\nabla u_{\varepsilon}\right|^2.
\end{align}
Combining \eqref{-3.22a} and \eqref{-3.23a}, we see that
\begin{align}\label{-3.22aa}
\frac{d}{d t} \int_{\Omega} \frac{\left|\nabla v_{\varepsilon}\right|^4}{v_{\varepsilon}^3} +2 \int_{\Omega} v_{\varepsilon}^{-1}\left|\nabla v_{\varepsilon}\right|^2\left|D^2 \ln v_{\varepsilon}\right|^2
\leqslant  2 b \int_{\Omega} v_{\varepsilon}\left|\nabla u_{\varepsilon}\right|^2
\end{align}
for all $t \in(0, T_{\max, \varepsilon})$ and $\varepsilon \in(0,1)$. 

Case I: $l\neq2$ and $l\neq3$. Multiplying the first equation in \eqref{sys-regul} by $u^{2-l}_{\varepsilon}$, we infer that
\begin{align}\label{0921-1843a}
\dt \int_{\Omega} u^{3-l}_{\varepsilon} 
= & (3-l)\int_{\Omega} u^{2-l}_{\varepsilon} \cdot \left\{\nabla \cdot \left(u^{l-1}_{\varepsilon} v_{\varepsilon} \nabla u_{\varepsilon }-u^{l}_{\varepsilon} v_{\varepsilon} \nabla v_{\varepsilon }\right)+ u_{\varepsilon} v_{\varepsilon} \right\} \nonumber\\
= &  (3-l)(l-2)\int_{\Omega} v_{\varepsilon}|\nabla u_{\varepsilon }|^2-(3-l)(l-2) \int_{\Omega} u_{\varepsilon} v_{\varepsilon} \nabla u_{\varepsilon }\cdot \nabla v_{\varepsilon } \nonumber\\
& +(3-l)\int_{\Omega}u^{3-l}_{\varepsilon}v_{\varepsilon } \quad \text { for all } t \in(0, T_{\max, \varepsilon}) \text { and } \varepsilon \in(0,1).
\end{align} 

In this case, we need to consider two subcases. 

Subcase 1: $1 \leqslant l<2$ or $l>3$. From \eqref{0921-1843a} and Young’s inequality, we obtain
\begin{small}
\begin{align*}
\dt \int_{\Omega} u^{3-l}_{\varepsilon} + (l-3)(l-2)\int_{\Omega} v_{\varepsilon}|\nabla u_{\varepsilon }|^2 
= &  (l-3)(l-2) \int_{\Omega} u_{\varepsilon} v_{\varepsilon} \nabla u_{\varepsilon }\cdot \nabla v_{\varepsilon }+(3-l)\int_{\Omega}u^{3-l}_{\varepsilon}v_{\varepsilon } \nonumber\\
\leqslant  & \frac{(l-3)(l-2)}{4}\int_{\Omega} v_{\varepsilon}|\nabla u_{\varepsilon }|^2  + (l-3)(l-2) \int_{\Omega} u^{2}_{\varepsilon} v_{\varepsilon}\left|\nabla v_{\varepsilon}\right|^2 \nonumber\\
& + (3-l) \int_{\Omega}u^{3-l}_{\varepsilon} v_{\varepsilon},
\end{align*} 
\end{small}
so that
\begin{align}\label{0921-1104}
\frac{4b}{(l-3)(l-2)}  \dt \int_{\Omega} u^{3-l}_{\varepsilon} +3 b \int_{\Omega} v_{\varepsilon}\left|\nabla u_{\varepsilon}\right|^2  \leqslant & 4b  \int_{\Omega} u_{\varepsilon}^{2} v_{\varepsilon}\left|\nabla v_{\varepsilon}\right|^2  -\frac{4 b}{l-2}  \int_{\Omega} u^{3-l}_{\varepsilon} v_{\varepsilon}
\end{align} 
for all $t \in(0, T_{\max, \varepsilon})$ and $\varepsilon \in(0,1)$. We set
\begin{align*}
G_{\varepsilon}(t)= \frac{4b}{(l-3)(l-2)}\int_{\Omega} u^{3-l}_{\varepsilon} +\int_{\Omega} \frac{\left|\nabla v_{\varepsilon}\right|^4}{v_{\varepsilon}^3}\quad \text { for all } t \in(0, T) \text { and } \varepsilon \in(0,1).
\end{align*}
Combining \eqref{-3.22aa} and \eqref{0921-1104}, we conclude that
\begin{align*}
G_{\varepsilon}^{\prime}(t) +2 \int_{\Omega} v_{\varepsilon}^{-1}\left|\nabla v_{\varepsilon}\right|^2\left|D^2 \ln v_{\varepsilon}\right|^2 + b \int_{\Omega} v_{\varepsilon}\left|\nabla u_{\varepsilon}\right|^2
\leqslant & 4b  \int_{\Omega} u_{\varepsilon}^{2} v_{\varepsilon}\left|\nabla v_{\varepsilon}\right|^2  -\frac{4 b}{l-2}  \int_{\Omega} u^{3-l}_{\varepsilon} 
\end{align*}
for all $t \in(0, T_{\max, \varepsilon})$ and $\varepsilon \in(0,1)$, which implies \eqref{0922-1558}.

Subcase 2: $2 < l<3$. From \eqref{0921-1843a} and Young’s inequality, we see that
\begin{small}
\begin{align*}
- \dt \int_{\Omega} u^{3-l}_{\varepsilon} + (3-l)(l-2)\int_{\Omega} v_{\varepsilon}|\nabla u_{\varepsilon }|^2 
= &  (3-l)(l-2) \int_{\Omega} u_{\varepsilon} v_{\varepsilon} \nabla u_{\varepsilon }\cdot \nabla v_{\varepsilon } -(3-l)\int_{\Omega}u^{3-l}_{\varepsilon}v_{\varepsilon} \nonumber\\
\leqslant  & \frac{(3-l)(l-2)}{4}\int_{\Omega} v_{\varepsilon}|\nabla u_{\varepsilon }|^2  + (3-l)(l-2) \int_{\Omega} u^{2}_{\varepsilon} v_{\varepsilon}\left|\nabla v_{\varepsilon}\right|^2 \nonumber\\
& - (3-l) \int_{\Omega}u^{3-l}_{\varepsilon}v_{\varepsilon},
\end{align*} 
\end{small}
so that
\begin{align}\label{0921-2120}
-\frac{4b}{(3-l)(l-2)} \dt \int_{\Omega} u^{3-l}_{\varepsilon} + 3b \int_{\Omega} v_{\varepsilon}|\nabla u_{\varepsilon }|^2 
\leqslant &  4 b \int_{\Omega} u^{2}_{\varepsilon} v_{\varepsilon}\left|\nabla v_{\varepsilon}\right|^2 - \frac{4 b}{l-2}   \int_{\Omega}u^{3-l}_{\varepsilon}v_{\varepsilon}\nonumber\\
\leqslant &  4 b \int_{\Omega} u^{2}_{\varepsilon} v_{\varepsilon}\left|\nabla v_{\varepsilon}\right|^2
\end{align}
for all $t \in(0, T_{\max, \varepsilon})$ and $\varepsilon \in(0,1)$. We set
\begin{align*}
G_{\varepsilon}(t)=- \frac{4b}{(3-l)(l-2)}\int_{\Omega} u^{3-l}_{\varepsilon} +\int_{\Omega} \frac{\left|\nabla v_{\varepsilon}\right|^4}{v_{\varepsilon}^3}\quad \text { for all } t \in(0, T_{\max, \varepsilon}) \text { and } \varepsilon \in(0,1).
\end{align*}
Combining \eqref{-3.22aa} and \eqref{0921-2120}, we infer that
\begin{align*}
G_{\varepsilon}^{\prime}(t)+2 \int_{\Omega} v_{\varepsilon}^{-1}\left|\nabla v_{\varepsilon}\right|^2\left|D^2 \ln v_{\varepsilon}\right|^2  + b \int_{\Omega} v_{\varepsilon}\left|\nabla u_{\varepsilon}\right|^2 
\leqslant &  4b  \int_{\Omega} u_{\varepsilon}^{2} v_{\varepsilon}\left|\nabla v_{\varepsilon}\right|^2 
\end{align*}
for all $t \in(0, T_{\max, \varepsilon})$ and $\varepsilon \in(0,1)$, which implies \eqref{0922-1561}.

Case II: $l=2$. Multiplying the first equation in \eqref{sys-regul} by $1+\ln u_{\varepsilon}$, we use Cauchy-Schwarz inequality and $\ln \xi \leqslant \xi$ for all $\xi>0$ to obtain
\begin{align*}
\dt \int_{\Omega} u_{\varepsilon} \ln u_{\varepsilon}+\int_{\Omega} v_{\varepsilon}\left|\nabla u_{\varepsilon}\right|^2 
= &  \int_{\Omega} u_{\varepsilon} v_{\varepsilon} \nabla u_{\varepsilon} \cdot \nabla v_{\varepsilon}+ \int_{\Omega} u_{\varepsilon}v_{\varepsilon}     + \int_{\Omega} u_{\varepsilon} v_{\varepsilon} \ln u_{\varepsilon}\nonumber\\
\leqslant &   \frac{1}{4} \int_{\Omega} v_{\varepsilon}\left|\nabla u_{\varepsilon}\right|^2+  \int_{\Omega} u_{\varepsilon}^{2} v_{\varepsilon}\left|\nabla v_{\varepsilon}\right|^2 + \int_{\Omega} u_{\varepsilon}v_{\varepsilon} + \int_{\Omega} u^2_{\varepsilon} v_{\varepsilon}
\end{align*}
so that
\begin{align}\label{0921-1256}
4b \dt \int_{\Omega} u_{\varepsilon} \ln u_{\varepsilon}+3b \int_{\Omega} v_{\varepsilon}\left|\nabla u_{\varepsilon}\right|^2    
\leqslant  4b \int_{\Omega} u_{\varepsilon}^{2} v_{\varepsilon}\left|\nabla v_{\varepsilon}\right|^2 + 4b \int_{\Omega} u_{\varepsilon} v_{\varepsilon} + 4b \int_{\Omega} u_{\varepsilon}^{2} v_{\varepsilon}
\end{align}
for all $t \in(0, T_{\max, \varepsilon})$ and $\varepsilon \in(0,1)$. We set
\begin{align*}
G_{\varepsilon}(t)= 4b \int_{\Omega} u_{\varepsilon} \ln u_{\varepsilon} +\int_{\Omega} \frac{\left|\nabla v_{\varepsilon}\right|^4}{v_{\varepsilon}^3}\quad \text { for all } t \in(0, T) \text { and } \varepsilon \in(0,1).
\end{align*}
Combining \eqref{-3.22aa} and \eqref{0921-1256}, we conclude that
\begin{align*}
G_{\varepsilon}^{\prime}(t) +2 \int_{\Omega} v_{\varepsilon}^{-1}\left|\nabla v_{\varepsilon}\right|^2\left|D^2 \ln v_{\varepsilon}\right|^2 + b \int_{\Omega} v_{\varepsilon}\left|\nabla u_{\varepsilon}\right|^2   
\leqslant & 4b \int_{\Omega} u_{\varepsilon}^{2} v_{\varepsilon}\left|\nabla v_{\varepsilon}\right|^2 + 4b \int_{\Omega} u_{\varepsilon} v_{\varepsilon} \nonumber\\
& + 4b \int_{\Omega} u_{\varepsilon}^{2} v_{\varepsilon},
\end{align*}
for all $t \in(0, T_{\max, \varepsilon})$ and $\varepsilon \in(0,1)$, which implies \eqref{0922-1559}.

Case III: $l=3$. Multiplying the first equation in \eqref{sys-regul} by $-u^{-1}_{\varepsilon}$, we use Cauchy-Schwarz inequality to show that
\begin{align*}
- \dt \int_{\Omega}\ln u_{\varepsilon} 
= & - \int_{\Omega} u^{-1}_{\varepsilon} \cdot \left\{\nabla \cdot \left(u^{2}_{\varepsilon} v_{\varepsilon} \nabla u_{\varepsilon }-u^{3}_{\varepsilon} v_{\varepsilon} \nabla v_{\varepsilon }\right)+ u_{\varepsilon}  v_{\varepsilon} \right\} \nonumber\\
= &  -\int_{\Omega} v_{\varepsilon}|\nabla u_{\varepsilon }|^2 + \int_{\Omega} u_{\varepsilon} v_{\varepsilon} \nabla u_{\varepsilon }\cdot \nabla v_{\varepsilon } 
-\int_{\Omega}v_{\varepsilon}\nonumber\\
\leqslant &  - \frac{3}{4} \int_{\Omega} v_{\varepsilon}\left|\nabla u_{\varepsilon}\right|^2+  \int_{\Omega} u_{\varepsilon}^{2} v_{\varepsilon}\left|\nabla v_{\varepsilon}\right|^2, 
\end{align*} 
so that
\begin{align}\label{0921-2309}
- 4b \dt \int_{\Omega}\ln u_{\varepsilon} + 3b \int_{\Omega} v_{\varepsilon}\left|\nabla u_{\varepsilon}\right|^2
\leqslant  4b \int_{\Omega} u_{\varepsilon}^{2} v_{\varepsilon}\left|\nabla v_{\varepsilon}\right|^2 
\end{align}
for all $t \in(0, T_{\max, \varepsilon})$ and $\varepsilon \in(0,1)$. We set
\begin{align*}
G_{\varepsilon}(t)= -4b \int_{\Omega} \ln u_{\varepsilon} +\int_{\Omega} \frac{\left|\nabla v_{\varepsilon}\right|^4}{v_{\varepsilon}^3}\quad \text { for all } t \in(0, T_{\max, \varepsilon}) \text { and } \varepsilon \in(0,1).
\end{align*}
Combining \eqref{-3.22aa} and \eqref{0921-2309}, we conclude that
\begin{align*}
G_{\varepsilon}^{\prime}(t) +2 \int_{\Omega} v_{\varepsilon}^{-1}\left|\nabla v_{\varepsilon}\right|^2\left|D^2 \ln v_{\varepsilon}\right|^2 + b \int_{\Omega} v_{\varepsilon}\left|\nabla u_{\varepsilon}\right|^2 + \int_{\Omega} u_{\varepsilon} v_{\varepsilon}^{-3}\left|\nabla v_{\varepsilon}\right|^4  
\leqslant  4b \int_{\Omega} u_{\varepsilon}^{2} v_{\varepsilon}\left|\nabla v_{\varepsilon}\right|^2 
\end{align*}
for all $t \in(0, T_{\max, \varepsilon})$ and $\varepsilon \in(0,1)$, which implies \eqref{0922-1560}.
\end{proof}

We are now in a position to derive a uniform estimate for $\frac{|\nabla v_{\varepsilon }|^4}{v_{\varepsilon}^3}$ in $L^\infty((0, T_{\max, \varepsilon}), L^1(\Omega))$ independent of $\varepsilon$ applying the above energy-like functional differential inequality.
\begin{lem}\label{lemma-3.8}
Let $l > 1$ and assume that \eqref{assIniVal} holds.
Then we have
\begin{align}\label{-3.29}
\int_{\Omega} \frac{|\nabla v_{\varepsilon}|^4}{v_{\varepsilon}^3} \leqslant C \quad \text { for all } t \in(0, T_{\max, \varepsilon}) \text { and } \varepsilon \in(0,1)
\end{align}
and
\begin{align}\label{-3.29xaxa}
\int_0^{T_{\max, \varepsilon}} \int_{\Omega}  v_{\varepsilon} |\nabla v_{\varepsilon}|^2 \leqslant C \quad \text { for all } \varepsilon \in(0,1)
\end{align}
as well as
\begin{align}\label{-3.32} 
\int_0^{T_{\max, \varepsilon}}  \int_{\Omega} \frac{\left|\nabla v_{\varepsilon}\right|^6}{v_{\varepsilon}^5} \leqslant C \quad \text { for all } \varepsilon \in(0,1),
\end{align}
where $C$ is a positive constant depending on $-\int_{\Omega} \ln u_0$, $\int_{\Omega} u_0$, $\int_{\Omega} v_0$, $\int_{\Omega} |\nabla v_{0 }|^2$, $\int_{\Omega} \frac{|\nabla v_{0 }|^2}{v_0}$ and $\int_{\Omega} \frac{\left|\nabla v_{0}\right|^4}{v_{0}^3}$, but independent of $\varepsilon$. 
\end{lem}
\begin{proof}
We use Lemma \ref{lemma-3.6} to derive the estimates in this lemma. We put $m_*=\int_{\Omega}\left(u_0+1\right)+ \int_{\Omega} v_0$ and observe that 
\begin{align}\label{-2.8x}
\int_{\Omega} u_{\varepsilon}(t) \leqslant \int_{\Omega} (u_{0}+\varepsilon)+ \int_{\Omega} v_{0} \leqslant m_*, \quad \text { for all } t \in\left(0, T_{\max, \varepsilon}\right) 
\end{align}
according to \eqref{-2.8}. An application of \eqref{eq-6.4} with $p=1$ and $\eta =\frac{1}{8 }$ provides $c_1>0$ such that
\begin{align}\label{1008-0954}
4b  \int_\Omega u_{\varepsilon}^{2} v_{\varepsilon} |\nabla v_{\varepsilon}|^2
\leqslant & \frac{b}{2} \int_\Omega  v_{\varepsilon}|\nabla u_{\varepsilon}|^2 + \left\{c_1b \left\|v_{\varepsilon} \right\|_{L^{\infty}(\Omega)}+8 c_1 b  \left\|v_{\varepsilon} \right\|^3_{L^{\infty}(\Omega)}\right\} \cdot \int_{\Omega} u_{\varepsilon}^{2} v_{\varepsilon}  \cdot \int_\Omega \frac{|\nabla v_{\varepsilon}|^4}{v_{\varepsilon}^3} \nonumber\\
&+c_1 b  \left\|v_{\varepsilon} \right\|^2_{L^{\infty}(\Omega)} \cdot\left\{\int_\Omega u_{\varepsilon}\right\}^{3} \cdot \int_\Omega \frac{|\nabla v_{\varepsilon}|^4}{v_{\varepsilon}^3}+c_1 b  \left\|v_{\varepsilon} \right\|^2_{L^{\infty}(\Omega)}\cdot \int_{\Omega} u_{\varepsilon} v_{\varepsilon}.
\end{align}

Case I: $ 1 < l < 2$. Substituting \eqref{1008-0954} into \eqref{0922-1558}, thanks to \eqref{-2.9}, \eqref{-2.8x} and Young’s inequality, for all $t \in(0, T_{\max, \varepsilon})$ and $\varepsilon \in(0,1)$, we obtain 
\begin{align}\label{-3.33}
G_{\varepsilon}^{\prime}(t) +2 \int_{\Omega} v_{\varepsilon}^{-1}\left|\nabla v_{\varepsilon}\right|^2\left|D^2 \ln v_{\varepsilon}\right|^2+ \frac{b}{2} \int_{\Omega} v_{\varepsilon}\left|\nabla u_{\varepsilon}\right|^2 
\leqslant & B \int_{\Omega} u_{\varepsilon}^{2}  v_{\varepsilon} \cdot \int_\Omega \frac{|\nabla v_{\varepsilon}|^4}{v_{\varepsilon}^3} +B \int_\Omega \frac{|\nabla v_{\varepsilon}|^4}{v_{\varepsilon}^3}\nonumber\\
& +B \int_{\Omega} u_{\varepsilon} v_{\varepsilon} + \frac{4 b}{2-l}  \int_{\Omega} u^{3-l}_{\varepsilon} v_{\varepsilon} \nonumber\\
\leqslant & B \int_{\Omega} u_{\varepsilon}^{2} v_{\varepsilon}  \cdot \left(D \int_{\Omega} u_{\varepsilon}^{3-l}+\int_\Omega \frac{|\nabla v_{\varepsilon}|^4}{v_{\varepsilon}^3} \right)\nonumber\\
& +B \int_\Omega \frac{|\nabla v_{\varepsilon}|^4}{v_{\varepsilon}^3}+B \int_{\Omega} u_{\varepsilon} v_{\varepsilon} \nonumber\\
& + \frac{4 b}{2-l}  \int_{\Omega} u^{2}_{\varepsilon} v_{\varepsilon}  + \frac{4 b}{2-l}  \int_{\Omega} u_{\varepsilon} v_{\varepsilon}\nonumber\\
= &  B  \int_{\Omega} u_{\varepsilon}^{2}   \cdot G_{\varepsilon} (t)+B \int_\Omega \frac{|\nabla v_{\varepsilon}|^4}{v_{\varepsilon}^3}\nonumber\\
& +\left(B + \frac{4 b}{2-l}\right)\int_{\Omega} u_{\varepsilon} v_{\varepsilon} \nonumber\\
& + \frac{4 b}{2-l}  \int_{\Omega} u^{2}_{\varepsilon} v_{\varepsilon},
\end{align}
where $B=\max \left\{c_1 b\left\|v_{0} \right\|_{L^{\infty}(\Omega)}+8 c_1 b \left\|v_{0} \right\|^4_{L^{\infty}(\Omega)}, c_1b m_*^3 \left\|v_{0} \right\|^2_{L^{\infty}(\Omega)}, c_1b \left\|v_{0} \right\|^2_{L^{\infty}(\Omega)} \right\}$ and $D=\frac{4b}{(l-3)(l-2)}$. We set
\begin{align*}
Z_{1 \varepsilon}(t)= B\int_{\Omega} u_{\varepsilon}^{2} v_{\varepsilon} \quad \text { for all } t \in(0, T) \text { and } \varepsilon \in(0,1) 
\end{align*}
and
\begin{align*}
M_{1\varepsilon}(t) = & B \int_\Omega \frac{|\nabla v_{\varepsilon}|^4}{v_{\varepsilon}^3} + \left(B + \frac{4 b}{2-l}\right) \int_{\Omega} u_{\varepsilon} v_{\varepsilon} + \frac{4 b}{2-l}  \int_{\Omega} u^{2}_{\varepsilon}v_{\varepsilon}
\end{align*}
for all $t \in(0, T_{\max, \varepsilon})$ and $\varepsilon \in(0,1)$. Then we rewrite \eqref{-3.33} as follows
\begin{align*}
G_{\varepsilon}^{\prime}(t) \leqslant  Z_{1\varepsilon}(t) G_{\varepsilon}(t)+M_{1\varepsilon}(t) \quad \text { for all } t \in(0, T_{\max, \varepsilon}) \text { and } \varepsilon \in(0,1). 
\end{align*}
Integrating this differential inequality gives
\begin{align}\label{-3.34xx}
G_{\varepsilon}(t)  & \leqslant G_{\varepsilon}(0) e^{\int_0^t Z_{1 \varepsilon}(s) \d s}+\int_0^t e^{\int_s^t Z_{1 \varepsilon}(\sigma) \d \sigma} M_{1 \varepsilon}(s) \d s 
\end{align}
for all $t \in(0, T_{\max, \varepsilon})$ and $\varepsilon \in(0,1)$. Since 
\begin{align*}
\int_s^t Z_{1 \varepsilon}(\sigma) \d \sigma \leqslant  c_2 \quad \text { for all } t \in\left(0, T_{\max, \varepsilon}\right),~ s \in[0, t) \text {  and   } \varepsilon \in(0,1)
\end{align*}
by Lemma \ref{lemma-3.7}, and 
\begin{align*}
\int_0^t M_{1 \varepsilon}(s) \d s \leqslant  c_3 \quad \text { for all } t \in\left(0, T_{\max, \varepsilon} \right) \text { and } \varepsilon \in(0,1)
\end{align*}
due to \eqref{-2.10}, \eqref{-3.10} and \eqref{-3.25xxx}, from \eqref{-3.34xx} we conclude that 
\begin{align}\label{0908-1147}
G_{\varepsilon}(t) & \leqslant  \left\{\frac{4b}{(l-3)(l-2)} \int_{\Omega} (u_{0}+1)^{3-l}+\int_{\Omega} \frac{\left|\nabla v_{0}\right|^4}{v_{0}^3}\right\} \cdot e^{c_2}
+c_3 e^{c_2}
\end{align}
for all $t \in(0, T_{\max, \varepsilon})$ and $\varepsilon \in(0,1)$, which establishes \eqref{-3.29}. Similarly, we can derive \eqref{-3.29} for the cases where $2 < l < 3$ or $l>3$.

Case II: $l = 2$. Plugging \eqref{1008-0954} into \eqref{0922-1559} and using $\xi \ln \xi +\frac{1}{e}\geqslant 0$ and $\ln \xi - \xi \leqslant 0$ for all $\xi>0$, thanks to \eqref{-2.9}, \eqref{-2.8x} and Young’s inequality, for all $t \in(0, T_{\max, \varepsilon})$ and $\varepsilon \in(0,1)$, we obtain 
\begin{align}\label{-3.33aa}
G_{\varepsilon}^{\prime}(t) +2 \int_{\Omega} v_{\varepsilon}^{-1}\left|\nabla v_{\varepsilon}\right|^2\left|D^2 \ln v_{\varepsilon}\right|^2 + \frac{b}{2} \int_{\Omega} v_{\varepsilon}\left|\nabla u_{\varepsilon}\right|^2  \leqslant & B \int_{\Omega} u_{\varepsilon}^{2} v_{\varepsilon}  \cdot \int_\Omega \frac{|\nabla v_{\varepsilon}|^4}{v_{\varepsilon}^3} +B \int_\Omega \frac{|\nabla v_{\varepsilon}|^4}{v_{\varepsilon}^3} \nonumber\\
& +B \int_{\Omega} u_{\varepsilon} v_{\varepsilon}  + 4b \int_{\Omega} u_{\varepsilon} v_{\varepsilon}   +   4b \int_{\Omega} u_{\varepsilon}^{2} v_{\varepsilon}\nonumber\\
\leqslant & B \int_{\Omega} u_{\varepsilon}^{2} v_{\varepsilon}  \cdot \left(G_{\varepsilon} (t)+\frac{4b|\Omega|}{e} \right) \nonumber\\
&  +B \int_\Omega \frac{|\nabla v_{\varepsilon}|^4}{v_{\varepsilon}^3}+(B+4b) \int_{\Omega} u_{\varepsilon} v_{\varepsilon}     \nonumber\\
& +   4b \int_{\Omega} u_{\varepsilon}^{2} v_{\varepsilon} \nonumber\\
= &  B  \int_{\Omega} u_{\varepsilon}^{2} v_{\varepsilon}  \cdot G_{\varepsilon} (t)+B \int_\Omega \frac{|\nabla v_{\varepsilon}|^4}{v_{\varepsilon}^3} \nonumber\\
& +(B+4b) \int_{\Omega} u_{\varepsilon} v_{\varepsilon} \nonumber\\
& +4b \left(\frac{|\Omega|B}{e}+1\right) \int_{\Omega} u^2_{\varepsilon} v_{\varepsilon}
\end{align}
We set
\begin{align*}
Z_{2 \varepsilon}(t)= B\int_{\Omega} u_{\varepsilon}^{2}v_{\varepsilon} \quad \text { for all } t \in\left(0, T_{\max, \varepsilon}\right) \text { and } \varepsilon \in(0,1)
\end{align*}
and
\begin{align*}
M_{2 \varepsilon}(t) = B \int_\Omega \frac{|\nabla v_{\varepsilon}|^4}{v_{\varepsilon}^3} + (B+4b)\int_{\Omega} u_{\varepsilon} v_{\varepsilon}  +4b\left(1+ \frac{|\Omega|B}{e}\right) \int_{\Omega} u^2_{\varepsilon} v_{\varepsilon}
\end{align*}
for all $t \in(0, T_{\max, \varepsilon})$ and $\varepsilon \in(0,1)$. Then we rewrite \eqref{-3.33aa} as follows
\begin{align*}
G_{\varepsilon}^{\prime}(t) \leqslant  Z_{2 \varepsilon}(t) G_{\varepsilon}(t)+M_{2 \varepsilon}(t) \quad \text { for all } t \in(0, T_{\max, \varepsilon}) \text { and } \varepsilon \in(0,1). 
\end{align*}
Integrating this differential inequality gives
\begin{align}\label{-3.34}
G_{\varepsilon}(t)  & \leqslant G_{\varepsilon}(0) e^{\int_0^t Z_{2 \varepsilon}(s) \d s}+\int_0^t e^{\int_s^t Z_{2 \varepsilon}(\sigma) \d \sigma} M_{2 \varepsilon}(s) \d s
\end{align}
for all $t \in(0, T_{\max, \varepsilon})$ and $\varepsilon \in(0,1)$. Since 
\begin{align*}
\int_s^t Z_{2 \varepsilon}(\sigma) \d \sigma \leqslant  c_4 \quad \text { for all } t \in\left(0, T_{\max, \varepsilon}\right),~ s \in[0, t) \text {  and   } \varepsilon \in(0,1)
\end{align*}
by Lemma \ref{lemma-3.7}, and 
\begin{align*}
\int_0^t M_{2 \varepsilon}(s) \d s \leqslant  c_5 \quad \text { for all } t \in\left(0, T_{\max, \varepsilon}\right) \text { and } \varepsilon \in(0,1)
\end{align*}
due to \eqref{-2.10}, \eqref{-3.10} and \eqref{-3.25xxx}, from \eqref{-3.34} we conclude that 
\begin{align}\label{0908-1147a}
G_{\varepsilon}(t) & \leqslant  \left\{4b \int_{\Omega} (u_{0}+1)\ln (u_{0}+1)+\int_{\Omega} \frac{\left|\nabla v_{0}\right|^4}{v_{0}^3}\right\} \cdot e^{c_4}
+c_5 e^{c_4}
\end{align}
for all $t \in(0, T_{\max, \varepsilon})$ and $\varepsilon \in(0,1)$, which establishes \eqref{-3.29}. In a similar manner, we can derive \eqref{-3.29} for the case where $l = 3$.

By a direct integration in \eqref{-3.33} or \eqref{-3.33aa}, thanks to \eqref{-2.10}, \eqref{-3.10}, \eqref{-3.25xxx}  and \eqref{0908-1147} or \eqref{0908-1147a}, we see that
\begin{align*}
\int_0^{T_{\max, \varepsilon}} \int_{\Omega} v_{\varepsilon}^{-1}\left|\nabla v_{\varepsilon}\right|^2\left|D^2 \ln v_{\varepsilon}\right|^2 \leqslant C\quad \text { for all } \varepsilon \in(0,1)
\end{align*}
and
\begin{align*}
\int_0^{T_{\max, \varepsilon}} \int_{\Omega} v_{\varepsilon}\left|\nabla u_{\varepsilon}\right|^2  \leqslant C\quad \text { for all } \varepsilon \in(0,1).
\end{align*}
Combining these two inequalities with \eqref{-3.23a}, we show that
\begin{align*}
\int_0^{T_{\max, \varepsilon}} \int_{\Omega} \frac{\left|\nabla v_{\varepsilon}\right|^6}{v_{\varepsilon}^5} \leqslant C \quad \text { for all } \varepsilon \in(0,1).
\end{align*}
We complete the proof.
\end{proof}

In view of \eqref{0704-0024} and \eqref{-3.29}, we can derive the following lemma.
\begin{lem}\label{lemma-3.9x}
Let $1<l \leqslant 3$ and assume that \eqref{assIniVal} holds.
Then for all $p \geqslant 2$ we have 
\begin{align}\label{-3.36}
\int_{\Omega} u_{\varepsilon}^p(t) \leqslant C(p) \quad \text { for all } t \in(0, T_{\max, \varepsilon}) \text { and } \varepsilon \in(0,1)
\end{align}
\begin{align}\label{4.7-6111}
\int_0^{T_{\max, \varepsilon}} \int_{\Omega} u_{\varepsilon}^{p} v_{\varepsilon} \leqslant  C(p) \quad \text { for all } \varepsilon \in(0,1)
\end{align}
and
\begin{align}\label{-3.37}
\int_0^{T_{\max, \varepsilon}} \int_{\Omega} u_{\varepsilon}^{p+l-3} v_{\varepsilon} \left|\nabla u_{\varepsilon}\right|^2 \leqslant  C(p) \quad \text { for all }  \varepsilon \in(0,1),
\end{align}
where $C(p)$ is a positive constant depending on $-\int_{\Omega} \ln u_0$, $\int_{\Omega} u_0$, $\int_{\Omega} v_0$, $\int_{\Omega} |\nabla v_{0 }|^2$, $\int_{\Omega} \frac{|\nabla v_{0 }|^2}{v_0}$ and $\int_{\Omega} \frac{\left|\nabla v_{0}\right|^4}{v_{0}^3}$, but independent of $\varepsilon$. 
\end{lem}

\begin{proof}

We employ \eqref{-3.29} to fix $c_1 > 0$ such that
\begin{align}\label{0904-923a}
\int_{\Omega} \frac{|\nabla v_{\varepsilon}|^4}{v_{\varepsilon}^3} \leqslant c_1.
\end{align}

Case I: $1<l\leqslant 2$. Plugging \eqref{-2.8x} and \eqref{0904-923a} into \eqref{0704-0024}, and using \eqref{-2.9}, \eqref{jia-1} and Young's inequality, we obtain 
\begin{align}\label{-3.43}
& \dt \int_{\Omega} u_{\varepsilon}^p 
 + \frac{p(p-1)}{4} \int_{\Omega} u_{\varepsilon}^{p+l-3} v_{\varepsilon} \left|\nabla 
u_{\varepsilon}\right|^2  \nonumber\\
\leqslant &   A \int_{\Omega} u_{\varepsilon}^{p+1} v_{\varepsilon}  \cdot \int_\Omega \frac{|\nabla v_{\varepsilon}|^4}{v_{\varepsilon}^3} +  A \int_{\Omega}v_{\varepsilon} \cdot \int_\Omega \frac{|\nabla v_{\varepsilon}|^4}{v_{\varepsilon}^3} \nonumber\\
& + A\left\{\int_{\Omega} \frac{u_{\varepsilon}}{v_{\varepsilon}}|\nabla v_{\varepsilon}|^2
  +\int_{\Omega} \frac{v_{\varepsilon}}{u_{\varepsilon}}|\nabla u_{\varepsilon}|^2
  +\int_{\Omega} u_{\varepsilon} v_{\varepsilon}  \right\} \cdot \int_{\Omega} u_{\varepsilon}^p \nonumber\\
& + A m_*^{2 p+2l-3} \cdot \int_\Omega \frac{|\nabla v_{\varepsilon}|^4}{v_{\varepsilon}^3}+ A \int_{\Omega} u_{\varepsilon} v_{\varepsilon}  \nonumber\\
\leqslant &   \left(c_1 + 1\right) A \cdot \left\{\int_{\Omega} \frac{u_{\varepsilon}}{v_{\varepsilon}}|\nabla v_{\varepsilon}|^2
  +\int_{\Omega} \frac{v_{\varepsilon}}{u_{\varepsilon}}|\nabla u_{\varepsilon}|^2
  +\int_{\Omega} u_{\varepsilon} v_{\varepsilon}  \right\} \cdot \int_{\Omega} u_{\varepsilon}^p  \nonumber\\
& + A \left( m_*^{2 p+2l-3}+ \left\|v_{0} \right\|^2_{L^{\infty}(\Omega)}|\Omega|\right) \cdot \int_\Omega \frac{|\nabla v_{\varepsilon}|^4}{v_{\varepsilon}^3}+ A \int_{\Omega} u_{\varepsilon} v_{\varepsilon}  
\end{align}
for all $t \in(0, T_{\max, \varepsilon})$ and $\varepsilon \in(0,1)$. For each $\varepsilon \in(0,1)$, we set
\begin{align*}
Y_{1 \varepsilon}(t) = \int_{\Omega} u_{\varepsilon}^p
\end{align*}
and
\begin{align*}
H_{1 \varepsilon}(t)= \left(c_1 + 1\right) A \cdot \left\{\int_{\Omega} \frac{u_{\varepsilon}}{v_{\varepsilon}}|\nabla v_{\varepsilon}|^2+\int_{\Omega} \frac{v_{\varepsilon}}{u_{\varepsilon}}|\nabla u_{\varepsilon}|^2+ \int_{\Omega} u_{\varepsilon} v_{\varepsilon} \right\}
\end{align*}
as well as
\begin{align*}
M_{1 \varepsilon}(t)= A \left( m_*^{2 p+2l-3}+ \left\|v_{0} \right\|^2_{L^{\infty}(\Omega)}|\Omega|\right) \cdot \int_\Omega \frac{|\nabla v_{\varepsilon}|^4}{v_{\varepsilon}^3}+ A \int_{\Omega} u_{\varepsilon} v_{\varepsilon}.
\end{align*}
Thus \eqref{-3.43} becomes 
\begin{align*}
Y_{1 \varepsilon}^{\prime}(t) \leqslant  H_{1 \varepsilon}(t) Y_{1 \varepsilon}(t)+ M_{1 \varepsilon}(t), \quad \text { for all } t \in(0, T_{\max, \varepsilon}) \text { and } \varepsilon \in(0,1). 
\end{align*}
Integrating this differential inequality gives
\begin{align}\label{-3.44}
Y_{1 \varepsilon}(t) \leqslant Y_{1 \varepsilon}(0) e^{\int_0^t H_{1 \varepsilon}(s) \d s}+\int_0^t e^{\int_s^t H_{1 \varepsilon}(\sigma) \d \sigma} M_{1 \varepsilon}(s) \d s 
\end{align}
for all $t \in(0, T_{\max, \varepsilon})$ and $\varepsilon \in(0,1)$. Since 
\begin{align*}
\int_s^t H_{1 \varepsilon}(\sigma) \d \sigma \leqslant c_3(p)  \quad \text { for all } t \in\left(0, T_{\max, \varepsilon}\right), s \in[0, t) \text {, and } \varepsilon \in(0,1)
\end{align*}
by \eqref{-2.10}, \eqref{-3.5aa} and \eqref{-3.9}, and
\begin{align*}
\int_0^t M_{1 \varepsilon}(s) \d s \leqslant c_4(p) \quad \text { for all } t \in\left(0, T_{\max, \varepsilon}\right)\text { and } \varepsilon \in(0,1)
\end{align*}
due to \eqref{-2.10}  and \eqref{-3.10}, from \eqref{-3.44}, we have 
\begin{align}\label{jia-3}
\int_{\Omega} u_{\varepsilon}^p(t) \leqslant & e^{c_3(p) } \int_{\Omega} u_{0}^p + c_4(p) e^{c_3(p)} \quad \text { for all } t \in\left(0, T_{\max, \varepsilon}\right)\text { and } \varepsilon \in(0,1).
\end{align}
Collecting \eqref{-2.10}, \eqref{-3.5aa}, \eqref{-3.9}, \eqref{jia-1} and \eqref{jia-3}, we  see that
\begin{align*}
\int_0^{T_{\max, \varepsilon}} \int_{\Omega} u_{\varepsilon}^{p+1} v_{\varepsilon} \leqslant  C(p)\quad \text { for all }  \varepsilon \in(0,1). 
\end{align*}
This together with Young's inequality and \eqref{-2.10} shows that
\begin{align*}
\int_0^{T_{\max, \varepsilon}} \int_{\Omega} u_{\varepsilon}^{p} v_{\varepsilon} \leqslant \int_0^{T_{\max, \varepsilon}} \int_{\Omega} u_{\varepsilon}^{p+1} v_{\varepsilon}  + \int_0^{T_{\max, \varepsilon}} \int_{\Omega} u_{\varepsilon} v_{\varepsilon} \leqslant  C(p) \quad \text { for all }  \varepsilon \in(0,1).
\end{align*}
By a direct integration in \eqref{-3.43}, combining \eqref{-2.10}, \eqref{-3.5aa}, \eqref{-3.9}, \eqref{-3.10} with \eqref{jia-3}, we conclude that
\begin{align*}
\int_0^{T_{\max, \varepsilon}} \int_{\Omega} u_{\varepsilon}^{p+l-3} v_{\varepsilon} \left|\nabla 
u_{\varepsilon}\right|^2 \leqslant  C(p)\quad \text { for all }  \varepsilon \in(0,1).
\end{align*}

Case II: $2<l\leqslant 3$. Plugging \eqref{-2.8x} and \eqref{0904-923a} into \eqref{0704-0024}, and using \eqref{-2.9} and Young's inequality, we obtain 
\begin{align}\label{-3.43xax}
&\dt \int_{\Omega} u_{\varepsilon}^p 
+ \frac{p(p-1)}{4} \int_{\Omega} u_{\varepsilon}^{p+l-3} v_{\varepsilon} \left|\nabla 
u_{\varepsilon}\right|^2  \nonumber\\
\leqslant &   A \int_{\Omega} u_{\varepsilon}^{p+2} v_{\varepsilon}  \cdot \int_\Omega \frac{|\nabla v_{\varepsilon}|^4}{v_{\varepsilon}^3} +   A \int_{\Omega} v_{\varepsilon}  \cdot \int_\Omega \frac{|\nabla v_{\varepsilon}|^4}{v_{\varepsilon}^3} \nonumber\\
& + A\left\{\int_{\Omega} \frac{u_{\varepsilon}}{v_{\varepsilon}}|\nabla v_{\varepsilon}|^2
  +\int_{\Omega} \frac{v_{\varepsilon}}{u_{\varepsilon}}|\nabla u_{\varepsilon}|^2
  +\int_{\Omega} u_{\varepsilon} v_{\varepsilon}  \right\} \cdot \int_{\Omega} u_{\varepsilon}^p \nonumber\\
& + A m_*^{2 p+2l-3} \cdot \int_\Omega \frac{|\nabla v_{\varepsilon}|^4}{v_{\varepsilon}^3}+ A \int_{\Omega} u_{\varepsilon} v_{\varepsilon} \nonumber\\
\leqslant &   A c_1 \cdot \int_{\Omega} u_{\varepsilon}^{p+2} v_{\varepsilon}   + A  \left\{\int_{\Omega} \frac{u_{\varepsilon}}{v_{\varepsilon}}|\nabla v_{\varepsilon}|^2
  +\int_{\Omega} \frac{v_{\varepsilon}}{u_{\varepsilon}}|\nabla u_{\varepsilon}|^2
  +\int_{\Omega} u_{\varepsilon} v_{\varepsilon}  \right\} \cdot \int_{\Omega} u_{\varepsilon}^p \nonumber\\
& +  A \left( m_*^{2 p+2l-3}+ \left\|v_{0} \right\|^2_{L^{\infty}(\Omega)}|\Omega|\right)  \cdot \int_\Omega \frac{|\nabla v_{\varepsilon}|^4}{v_{\varepsilon}^3}+ A \int_{\Omega} u_{\varepsilon} v_{\varepsilon} 
\end{align}
for all $t \in(0, T_{\max, \varepsilon})$ and $\varepsilon \in(0,1)$. Using \cite[Lemma 6.1]{2024-JDE-Winkler} with $d=0$, \eqref{-2.8x} and Young's inequality, we infer that
\begin{align*}
\int_{\Omega} u_{\varepsilon}^{p+2} v_{\varepsilon} \leqslant & \Gamma(p)m_*  \cdot \int_{\Omega} u_{\varepsilon}^{p-1} v_{\varepsilon}|\nabla u_{\varepsilon}|^2  +\Gamma(p) m_*^2 \cdot \int_{\Omega} u_{\varepsilon}^p \cdot \int_{\Omega} \frac{|\nabla v_{\varepsilon}|^6}{v_{\varepsilon}^5} \nonumber\\
& +\Gamma(p) m_*^{p+1} \cdot \int_{\Omega} u_{\varepsilon} v_{\varepsilon}\nonumber\\
\leqslant & \frac{p(p-1)}{8}  \cdot \int_{\Omega} u_{\varepsilon}^{p+l-3} v_{\varepsilon}|\nabla u_{\varepsilon}|^2 +\frac{2 \Gamma^2(p)m_*^2 }{p(p-1)} \int_{\Omega} v_{\varepsilon}|\nabla u_{\varepsilon}|^2 \nonumber\\
&  +\Gamma(p) m_*^2 \cdot  \int_{\Omega} \frac{|\nabla v_{\varepsilon}|^6}{v_{\varepsilon}^5} \cdot  \int_{\Omega} u_{\varepsilon}^p+\Gamma(p) m_*^{p+1} \cdot \int_{\Omega} u_{\varepsilon} v_{\varepsilon}
\end{align*}
for all $t \in(0, T_{\max, \varepsilon})$ and $\varepsilon \in(0,1)$, where $\Gamma(p)$ is a some constant. Together with above inequality, \eqref{-3.43xax} yields 
\begin{align}\label{-3.43xaxxx}
&\dt \int_{\Omega} u_{\varepsilon}^p 
+ \frac{p(p-1)}{8} \int_{\Omega} u_{\varepsilon}^{p+l-3} v_{\varepsilon} \left|\nabla 
u_{\varepsilon}\right|^2  \nonumber\\
\leqslant &   A c_1  \Gamma(p) m_*^2 \cdot  \int_{\Omega} \frac{|\nabla v_{\varepsilon}|^6}{v_{\varepsilon}^5} \cdot  \int_{\Omega} u_{\varepsilon}^p  + A  \left\{\int_{\Omega} \frac{u_{\varepsilon}}{v_{\varepsilon}}|\nabla v_{\varepsilon}|^2
  +\int_{\Omega} \frac{v_{\varepsilon}}{u_{\varepsilon}}|\nabla u_{\varepsilon}|^2
  +\int_{\Omega} u_{\varepsilon} v_{\varepsilon}  \right\} \cdot \int_{\Omega} u_{\varepsilon}^p \nonumber\\
& +  A \left( m_*^{2 p+2l-3}+ \left\|v_{0} \right\|^2_{L^{\infty}(\Omega)}|\Omega|\right)  \cdot \int_\Omega \frac{|\nabla v_{\varepsilon}|^4}{v_{\varepsilon}^3}+ (A+\Gamma(p) m_*^{p+1}) \int_{\Omega} u_{\varepsilon} v_{\varepsilon} \nonumber\\
& +\frac{2 \Gamma^2(p)m_*^2 }{p(p-1)} \int_{\Omega} v_{\varepsilon}|\nabla u_{\varepsilon}|^2 \nonumber\\
\leqslant &   A (c_1  \Gamma(p) m_*^2+1)  \left\{\int_{\Omega} \frac{u_{\varepsilon}}{v_{\varepsilon}}|\nabla v_{\varepsilon}|^2
  +\int_{\Omega} \frac{v_{\varepsilon}}{u_{\varepsilon}}|\nabla u_{\varepsilon}|^2
  +\int_{\Omega} u_{\varepsilon} v_{\varepsilon}+\int_{\Omega} \frac{|\nabla v_{\varepsilon}|^6}{v_{\varepsilon}^5}  \right\} \cdot \int_{\Omega} u_{\varepsilon}^p \nonumber\\
& +  A \left( m_*^{2 p+2l-3}+ \left\|v_{0} \right\|^2_{L^{\infty}(\Omega)}|\Omega|\right)  \cdot \int_\Omega \frac{|\nabla v_{\varepsilon}|^4}{v_{\varepsilon}^3}+ (A+\Gamma(p) m_*^{p+1}) \int_{\Omega} u_{\varepsilon} v_{\varepsilon} \nonumber\\
& +\frac{2 \Gamma^2(p)m_*^2 }{p(p-1)} \int_{\Omega} v_{\varepsilon}|\nabla u_{\varepsilon}|^2  \quad \text { for all } t \in(0, T_{\max, \varepsilon}) \text { and } \varepsilon \in(0,1).
\end{align}
For each $\varepsilon \in(0,1)$, we set
\begin{align*}
Y_{2 \varepsilon}(t) = \int_{\Omega} u_{\varepsilon}^p
\end{align*}
and
\begin{align*}
H_{2 \varepsilon}(t)= A (c_1  \Gamma(p) m_*^2+1)  \left\{\int_{\Omega} \frac{u_{\varepsilon}}{v_{\varepsilon}}|\nabla v_{\varepsilon}|^2
  +\int_{\Omega} \frac{v_{\varepsilon}}{u_{\varepsilon}}|\nabla u_{\varepsilon}|^2
  +\int_{\Omega} u_{\varepsilon} v_{\varepsilon}+\int_{\Omega} \frac{|\nabla v_{\varepsilon}|^6}{v_{\varepsilon}^5}  \right\}
\end{align*}
as well as
\begin{align*}
M_{2 \varepsilon}(t)= &  A \left( m_*^{2 p+2l-3}+ \left\|v_{0} \right\|^2_{L^{\infty}(\Omega)}|\Omega|\right)  \cdot \int_\Omega \frac{|\nabla v_{\varepsilon}|^4}{v_{\varepsilon}^3}+ (A+\Gamma(p) m_*^{p+1}) \int_{\Omega} u_{\varepsilon} v_{\varepsilon} \nonumber\\
& +\frac{2 \Gamma^2(p)m_*^2 }{p(p-1)} \int_{\Omega} v_{\varepsilon}|\nabla u_{\varepsilon}|^2.
\end{align*}
Thus \eqref{-3.43xaxxx} becomes 
\begin{align*}
Y_{2 \varepsilon}^{\prime}(t) \leqslant  H_{2 \varepsilon}(t) Y_{2 \varepsilon}(t)+ M_{2 \varepsilon}(t), \quad \text { for all } t \in(0, T_{\max, \varepsilon}) \text { and } \varepsilon \in(0,1). 
\end{align*}
Integrating this differential inequality gives
\begin{align}\label{-3.44xzxz}
Y_{2 \varepsilon}(t) \leqslant Y_{2 \varepsilon}(0) e^{\int_0^t H_{2 \varepsilon}(s) \d s}+\int_0^t e^{\int_s^t H_{2 \varepsilon}(\sigma) \d \sigma} M_{2 \varepsilon}(s) \d s 
\end{align}
for all $t \in(0, T_{\max, \varepsilon})$ and $\varepsilon \in(0,1)$. Since 
\begin{align*}
\int_s^t H_{2 \varepsilon}(\sigma) \d \sigma \leqslant c_5(p)  \quad \text { for all } t \in\left(0, T_{\max, \varepsilon}\right), s \in[0, t) \text {, and } \varepsilon \in(0,1)
\end{align*}
by \eqref{-2.10}, \eqref{-3.5aa}, \eqref{-3.9} and \eqref{-3.32}, and
\begin{align*}
\int_0^t M_{2 \varepsilon}(s) \d s \leqslant c_6(p) \quad \text { for all } t \in\left(0, T_{\max, \varepsilon}\right)\text { and } \varepsilon \in(0,1)
\end{align*}
due to \eqref{-2.10}, \eqref{-3.10} and \eqref{-3.29xaxa}, from \eqref{-3.44xzxz}, we have 
\begin{align}\label{jia-3a}
\int_{\Omega} u_{\varepsilon}^p(t) \leqslant & e^{c_5(p) } \int_{\Omega} u_{0}^p + c_6(p) e^{c_5(p)} \quad \text { for all } t \in\left(0, T_{\max, \varepsilon}\right)\text { and } \varepsilon \in(0,1).
\end{align}
Collecting \eqref{-2.10}, \eqref{-3.5aa}, \eqref{-3.9}, \eqref{jia-1} and \eqref{jia-3a}, we  see that
\begin{align*}
\int_0^{T_{\max, \varepsilon}} \int_{\Omega} u_{\varepsilon}^{p+1} v_{\varepsilon} \leqslant  C(p)\quad \text { for all }  \varepsilon \in(0,1). 
\end{align*}
This together with Young's inequality and \eqref{-2.10} shows that
\begin{align*}
\int_0^{T_{\max, \varepsilon}} \int_{\Omega} u_{\varepsilon}^{p} v_{\varepsilon} \leqslant \int_0^{T_{\max, \varepsilon}} \int_{\Omega} u_{\varepsilon}^{p+1} v_{\varepsilon}  + \int_0^{T_{\max, \varepsilon}} \int_{\Omega} u_{\varepsilon} v_{\varepsilon} \leqslant  C(p) \quad \text { for all }  \varepsilon \in(0,1).
\end{align*}
By a direct integration in \eqref{-3.43xaxxx}, combining \eqref{-2.10}, \eqref{-3.5aa}, \eqref{-3.9}, \eqref{-3.10}, \eqref{-3.29xaxa}, \eqref{-3.32} with \eqref{jia-3a}, we conclude that
\begin{align*}
\int_0^{T_{\max, \varepsilon}} \int_{\Omega} u_{\varepsilon}^{p+l-3} v_{\varepsilon} \left|\nabla 
u_{\varepsilon}\right|^2 \leqslant  C(p)\quad \text { for all }  \varepsilon \in(0,1).
\end{align*}
We complete the proof.
\end{proof}

\subsection{Higher-Order Estimates. Proof of Theorem \ref{thm-1.1}}
On the basis of standard heat semigroup estimates \cite{2010-JDE-Winkler} we can obtain $L^{\infty}$ bounds for $\nabla v_{\varepsilon}$.

\begin{lem}\label{lemma-4.1}
Let $1<l \leqslant 3$ and assume that \eqref{assIniVal} holds.
Then we have
\begin{align}\label{-4.1}
\left\|v_{\varepsilon}(t)\right\|_{W^{1, \infty}(\Omega)} \leqslant C\quad \text { for all } t \in(0, T_{\max, \varepsilon}) \text { and } \varepsilon \in(0,1),
\end{align}
where $C$ is a positive constant independent of $\varepsilon$.
\end{lem}
\begin{proof}
According to the Neumann heat semigroup, fixing any $p > 1$ we can find $c_1>0$ such that 
\begin{align}\label{-4.2}
&\left\|\nabla v_{\varepsilon}(t)\right\|_{L^{\infty}(\Omega)} \nonumber\\
= &\left\|\nabla e^{t(\Delta-1)} v_0-\int_0^t \nabla e^{(t-s)(\Delta-1)}\left\{u_{\varepsilon}(s) v_{\varepsilon}(s)-v_{\varepsilon}(s)\right\} \d s\right\|_{L^{\infty}(\Omega)} \nonumber\\
\leqslant &  c_1\left\|v_0\right\|_{W^{1, \infty}(\Omega)}\nonumber\\
& +  c_1 \int_0^t\left(1+(t-s)^{-\frac{1}{2}-\frac{1}{p}}\right) e^{-(t-s)}\left\|u_{\varepsilon}(s) v_{\varepsilon}(s)-v_{\varepsilon}(s)\right\|_{L^p(\Omega)} \d s
\end{align}
for all $t \in\left(0, T_{\max, \varepsilon}\right)$ and $\varepsilon \in(0,1)$. From \eqref{-2.9} and \eqref{-3.36}, we can find $c_2>0$ such that
\begin{align*}
\left\|u_{\varepsilon}(s) v_{\varepsilon}(s)-v_{\varepsilon}(s)\right\|_{L^p(\Omega)} 
\leqslant & \left\|u_{\varepsilon}(s)\right\|_{L^p(\Omega)}\left\|v_{\varepsilon}(s)\right\|_{L^{\infty}(\Omega)}+|\Omega|^{\frac{1}{p}}\left\|v_{\varepsilon}(s)\right\|_{L^{\infty}(\Omega)} \\
\leqslant & c_2\left\|v_0\right\|_{L^{\infty}(\Omega)}+|\Omega|^{\frac{1}{p}}\left\|v_0\right\|_{L^{\infty}(\Omega)}.
\end{align*}
This together with \eqref{-4.2} yields \eqref{-4.1}.
\end{proof}

To establish the $L^\infty$ estimate of $u_{\varepsilon}$, we cite the following two lemmas (cf. \cite{2024-JDE-Winkler}).
\begin{lem}[{\cite[Lemma 6.2]{2024-JDE-Winkler}}]
\label{lemma-4.2}
Let $p_{\star}>2$. Then there exist $\kappa=\kappa\left(p_{\star}\right)>0$ and $K=K\left(p_{\star}\right)>0$ such that for any choice of $p \geqslant p_{\star}$ and $\eta \in(0,1]$, all $\varphi \in C^1(\overline{\Omega})$ and $\psi \in C^1(\overline{\Omega})$ fulfilling $\varphi>0$ and $\psi>0$ in $\overline{\Omega}$ satisfy
\begin{align*}
\int_{\Omega} \varphi^{p+1} \psi \leqslant & \eta \int_{\Omega} \varphi^{p-1} \psi|\nabla \varphi|^2+\eta \cdot\left\{\int_{\Omega} \varphi^{\frac{p}{2}}\right\}^{\frac{2(p+1)}{p}} \cdot \int_{\Omega} \frac{|\nabla \psi|^6}{\psi^5} \nonumber\\
& +K \eta^{-\kappa} p^{2 \kappa} \cdot\left\{\int_{\Omega} \varphi^{\frac{p}{2}}\right\}^2 \cdot \int_{\Omega} \varphi \psi.
\end{align*}
\end{lem}

\begin{lem}[{\cite[Lemma 6.3]{2024-JDE-Winkler}}]
\label{lemma-4.3}
Let $a \geqslant 1, b \geqslant 1, q \geqslant 0$ and $\left\{N_k\right\}_{k \in\{0,1,2,3, \ldots\}} \subset[1, \infty)$ be such that
\begin{align*}
N_k \leqslant a^k N_{k-1}^{2+q \cdot 2^{-k}}+b^{2^k} \quad \text { for all } k \geqslant 1.
\end{align*}
Then
\begin{align*}
\liminf _{k \rightarrow \infty} N_k^{\frac{1}{2^k}} \leqslant\left(2 \sqrt{2} a^3 b^{1+\frac{q}{2}} N_0\right)^{e^{\frac{q}{2}}}.
\end{align*}
\end{lem}

In \cite{2024-JDE-Winkler}, Winkler achieved the estimate of $L^{\infty}$-bounds for $u_{\varepsilon}$ in a doubly degenerate reaction-diffusion system \eqref{sys-regul}. With Lemmas~\ref{lemma-3.9x} and \ref{lemma-4.1} at hand, we also can derive such an estimate for the system \eqref{sys-regul} by employing a similar method.

\begin{lem}\label{lemma-4.4vbv}
Let $1<l \leqslant 3$ and assume that \eqref{assIniVal} holds.
Then we have 
\begin{align}\label{619-1648}
\left\|u_{\varepsilon}(t)\right\|_{L^{\infty}(\Omega)} \leqslant C\quad \text { for all } t \in(0, T_{\max, \varepsilon}) \text { and } \varepsilon \in(0,1), 
\end{align}
where $C$ is a positive constant depending on $-\int_{\Omega} \ln u_0$, $\int_{\Omega} u_{0} \ln u_{0}$, $\int_{\Omega} u_0$, $\int_{\Omega} v_0$, $\int_{\Omega} |\nabla v_{0 }|^2$, $\int_{\Omega} \frac{|\nabla v_{0 }|^2}{v_0}$ and $\int_{\Omega} \frac{\left|\nabla v_{0}\right|^4}{v_{0}^3}$, but independent of $\varepsilon$.
\end{lem}
\begin{proof}
Based on Lemma \ref{lemma-4.1}, we can find $c_1>0$ such that
\begin{align}\label{511-1312}
\left|\nabla v_{\varepsilon}(x, t)\right| \leqslant c_1 \quad \text { for all } x \in \Omega, t \in\left(0, T_{\max, \varepsilon}\right) \text { and } \varepsilon \in(0,1) \text {, }
\end{align}
For integers $k \geqslant 1$ we set $p_k=2^k p_0+1$ with $p_0>1$, and let
\begin{align}\label{511-1223-1}
N_{k, \varepsilon}(T)=1+\sup _{t \in(0, T)} \int_{\Omega} u_{\varepsilon}^{p_k}(t)< \infty, \quad T \in\left(0, T_{\max, \varepsilon}\right), k \in\{0,1,2, \ldots\}, \varepsilon \in(0,1).
\end{align}
Owing to \eqref{-3.36}, we can find $c_2>0$ such that 
\begin{align}\label{511-1223}
N_{0, \varepsilon}(T) \leqslant c_2 \quad \text { for all } T \in\left(0, T_{\max, \varepsilon}\right) \text { and } \varepsilon \in(0,1). 
\end{align}
To estimate $N_{k, \varepsilon}(T)$ for $k \geqslant 1, T \in\left(0, T_{\max, \varepsilon}\right)$ and $\varepsilon \in(0,1)$, using the first equation of the system \eqref{sys-regul}, according to \eqref{511-1312} and Young's inequality, we see that 
\begin{align}\label{511-1339}
\dt \int_{\Omega} u_{\varepsilon}^{p_k}= & -p_k \left(p_k-1\right) \int_{\Omega} u_{\varepsilon}^{p_k+l-3} v_{\varepsilon}\left|\nabla u_{\varepsilon}\right|^2+ p_k\left(p_k-1\right) \int_{\Omega} u_{\varepsilon}^{p_k+l-2} v_{\varepsilon} \nabla u_{\varepsilon} \cdot \nabla v_{\varepsilon}\nonumber \\
& +p_k  \int_{\Omega} u_{\varepsilon}^{p_k} v_{\varepsilon}\nonumber \\
\leqslant & -\frac{ p_k\left(p_k-1\right)}{2} \int_{\Omega} u_{\varepsilon}^{p_k+l-3} v_{\varepsilon}\left|\nabla u_{\varepsilon}\right|^2+\frac{ p_k\left(p_k-1\right)}{2} \int_{\Omega} u_{\varepsilon}^{p_k+l-1} v_{\varepsilon}\left|\nabla v_{\varepsilon}\right|^2 \nonumber\\
& +p_k  \int_{\Omega} u_{\varepsilon}^{p_k} v_{\varepsilon}\nonumber\\
\leqslant & -\frac{p_k\left(p_k-1\right)}{2} \int_{\Omega} u_{\varepsilon}^{p_k-1} v_{\varepsilon}\left|\nabla u_{\varepsilon}\right|^2+ p_k\left(p_k-1\right)  c^2_1 \int_{\Omega} u_{\varepsilon}^{p_k+1} v_{\varepsilon} \nonumber \\
& +p_k  \int_{\Omega}   u_{\varepsilon}^{p_k} v_{\varepsilon}\nonumber \\
\leqslant & -\frac{p_k\left(p_k-1\right)}{2} \int_{\Omega} u_{\varepsilon}^{p_k-1} v_{\varepsilon}\left|\nabla u_{\varepsilon}\right|^2\nonumber\\
& + \left\{ p_k\left(p_k-1\right)c^2_1 + p_k  \right\}\cdot \left\{\int_{\Omega} u_{\varepsilon}^{p_k+l-1} v_{\varepsilon}+\int_{\Omega}  u_{\varepsilon} v_{\varepsilon}\right\}.
\end{align}
Since $p_k\left(p_k-1\right) \leqslant p_k^2$, $p_k \leqslant p_k^2$ and $\frac{p_k\left(p_k-1\right)}{2} \geqslant \frac{p_k^2}{4}$, from \eqref{511-1339} we infer that
\begin{align}\label{511-1338}
\dt \int_{\Omega} u_{\varepsilon}^{p_k}+\frac{p_k^2}{4} \int_{\Omega} u_{\varepsilon}^{p_k+l-3} v_{\varepsilon}\left|\nabla u_{\varepsilon}\right|^2 \leqslant c_3 p_k^2 \int_{\Omega} u_{\varepsilon}^{p_k+l-1} v_{\varepsilon}+c_3 p_k^2 \int_{\Omega} u_{\varepsilon} v_{\varepsilon}
\end{align}
for all $t \in(0, T_{\max, \varepsilon})$, where $c_3=  c^2_1+1$. Choosing $p_{\star}=2 p_0> 2$ in Lemma \ref{lemma-4.2}, we can conclude that 
\begin{align*}
\int_{\Omega} u_{\varepsilon}^{p_k+l-1} v_{\varepsilon} \leqslant & \frac{1}{4 c_3} \int_{\Omega} u_{\varepsilon}^{p_k+l-3} v_{\varepsilon}\left|\nabla u_{\varepsilon}\right|^2+\frac{1}{4 c_3} \cdot\left\{\int_{\Omega} u_{\varepsilon}^{\frac{p_k+l-2}{2}}\right\}^{\frac{2(p_k+l-1)}{p_k+l-2}} \cdot \int_{\Omega} \frac{\left|\nabla v_{\varepsilon}\right|^6}{v_{\varepsilon}^5} \\
& +K \cdot\left(4 c_3\right)^\kappa \cdot (p_k+l-1)^{2 \kappa} \cdot\left\{\int_{\Omega} u_{\varepsilon}^{\frac{p_k+l-2}{2}}\right\}^2 \cdot \int_{\Omega} u_{\varepsilon} v_{\varepsilon} 
\end{align*}
for all $t \in(0, T_{\max, \varepsilon})$, where $\kappa=\kappa\left(p_{\star}\right)>0$ and $K=K\left(p_{\star}\right)>0$. According to \eqref{511-1223-1} and Young's inequality, for all $T \in\left(0, T_{\max, \varepsilon}\right)$ we can estimate
\begin{align*}
\int_{\Omega} u_{\varepsilon}^{\frac{p_k+l-2}{2}}\leqslant \int_{\Omega} u_{\varepsilon}^{\frac{p_k+1}{2}} + |\Omega| +1 =\int_{\Omega} u_{\varepsilon}^{p_{k-1}} + |\Omega| +1 \leqslant N_{k-1, \varepsilon}(T) + |\Omega| +1 
\end{align*}
for all $t \in(0, T)$. Since $(a+b)^s \leqslant 2^{s-1}\left(a^s+b^s\right)$ for all $a, b >0$ and $s \geqslant 1$, we then obtain
\begin{align*}
c_3 p_k^2 \int_{\Omega} u_{\varepsilon}^{p_k+1} v_{\varepsilon} \leqslant & \frac{p_k^2}{4} \int_{\Omega} u_{\varepsilon}^{p_k-1} v_{\varepsilon}\left|\nabla u_{\varepsilon}\right|^2+p_k^2 N_{k-1, \varepsilon}^{\frac{2(p_k+l-1)}{p_k+l-2}}(T) \int_{\Omega} \frac{\left|\nabla v_{\varepsilon}\right|^6}{v_{\varepsilon}^5} + p_k^2 (|\Omega|+1)^3 \int_{\Omega} \frac{\left|\nabla v_{\varepsilon}\right|^6}{v_{\varepsilon}^5} \\
& +2^{2\kappa+1} c_3^{\kappa+1} K (p_k+2)^{2 \kappa+2} N_{k-1, \varepsilon}^2(T) \int_{\Omega} u_{\varepsilon} v_{\varepsilon} + 2^{2\kappa+1} c_3^{\kappa+1} K (|\Omega| +1)^2   \int_{\Omega} u_{\varepsilon} v_{\varepsilon}
\end{align*}
for all $t \in(0, T)$. Since $p_k^2 \leqslant (p_k+2)^{2 \kappa+2}$ and $1 \leqslant N_{k-1, \varepsilon}^2(T) \leqslant N_{k-1, \varepsilon}^{\frac{2(p_k+l-1)}{p_k+l-2}}(T)$, from \eqref{511-1338} we deduce that with $c_4 =\max \left\{1, (|\Omega| +1)^3 , 2^{2\kappa+1} c_3^{\kappa+1} K, 2^{2\kappa+1} c_3^{\kappa+1} K (|\Omega| +1)^2\right\}$,
\begin{align*}
\dt \int_{\Omega} u_{\varepsilon}^{p_k} \leqslant 2 c_4 (p_k+2)^{2 \kappa+2} N_{k-1, \varepsilon}^{\frac{2\left(p_k+1\right)}{p_k}}(T) w_{\varepsilon}(t) 
\end{align*}
for all  $t \in(0, T)$,  any  $T \in\left(0, T_{\max, \varepsilon}\right)$  and each $\varepsilon \in(0,1)$, where $w_{\varepsilon}(t) = \int_{\Omega} \frac{\left|\nabla v_{\varepsilon}(t)\right|^6}{v_{\varepsilon}^5(t)}+\int_{\Omega} u_{\varepsilon}(t) v_{\varepsilon}(t)$.
Collecting \eqref{-2.10} and \eqref{-3.32}, we have with $c_5>0$, 
\begin{align*}
\int_0^{T} w_{\varepsilon}(t) \d t \leqslant c_5 \quad \text { for all } \varepsilon \in(0,1).
\end{align*}
An integration of this shows that 
\begin{align*}
\int_{\Omega} u_{\varepsilon}^{p_k} & \leqslant \int_{\Omega}\left(u_0+\varepsilon\right)^{p_k}+2 c_4 c_5 (p_k+2)^{2 \kappa+2} N_{k-1, \varepsilon}^{\frac{2\left(p_k+l-1\right)}{p_k+l-2}}(T) \\
& \leqslant|\Omega| \cdot\left\|u_0+1\right\|_{L^{\infty}(\Omega)}^{p_k}+2 c_4 c_5 (p_k+2)^{2 \kappa+2} N_{k-1, \varepsilon}^{\frac{2\left(p_k+l-1\right)}{p_k+l-2}}(T).
\end{align*}
for all  $t \in(0, T)$,  any  $T \in\left(0, T_{\max, \varepsilon}\right)$  and each $\varepsilon \in(0,1)$. In view of \eqref{511-1223}, this yields
\begin{align*}
N_{k, \varepsilon}({T}) & \leqslant 1+|\Omega| \cdot\left\|u_0+1\right\|_{L^{\infty}(\Omega)}^{p_k}+2 c_4 c_5 (p_k+2)^{2 \kappa+2} N_{k-1, \varepsilon}^{\frac{2\left(p_k+l-1\right)}{p_k+l-2}}(T) \\
& \leqslant s^{2^k}+d^k N_{k-1, \varepsilon}^{2+r \cdot 2^{-k}}({T}) \quad \text { for all } {T} \in\left(0, T_{\max, \varepsilon}\right) \text { and } \varepsilon \in(0,1)
\end{align*}
with
\begin{align*}
d=\left(\max \left\{ (2p_0+3) c_4 c_5 , 1\right\}\right)^{2 \kappa+2}, \quad 
r=\frac{2}{p_0} \text { and } \quad 
s=[2+2|\Omega| \cdot\left\|u_0+1\right\|_{L^{\infty}(\Omega)}]^{p_0},
\end{align*}
where the following estimate is used 
\begin{align*}
1+|\Omega| \cdot\left\|u_0+1\right\|_{L^{\infty}(\Omega)}^{p_k} \leqslant (1+|\Omega| \cdot\left\|u_0+1\right\|_{L^{\infty}(\Omega)})^{p_k}=[(2+2|\Omega| \cdot\left\|u_0+1\right\|_{L^{\infty}(\Omega)})^{p_0}]^{2^k}.
\end{align*}
Then it follows from Lemma \ref{lemma-4.3} and \eqref{511-1223} that
\begin{align*}
\left\|u_{\varepsilon}(t)\right\|_{L^{\infty}(\Omega)}^{p_0} & =\liminf _{k \rightarrow \infty}\left\{\int_{\Omega} u_{\varepsilon}^{p_k}(t)\right\}^{\frac{1}{2^k}} \\
& \leqslant \liminf _{k \rightarrow \infty} N_{k, \varepsilon}^{\frac{1}{2^k}}({T}) \\
& \leqslant\left(2 \sqrt{2} d^3 s^{1+\frac{r}{2}} c_2\right)^{e^{\frac{r}{2}}}
\end{align*}
for all $t \in(0, T)$, $T \in\left(0, T_{\max, \varepsilon}\right)$ and each $\varepsilon \in(0,1)$.
Taking $T \nearrow T_{\max, \varepsilon}$, we obtain \eqref{619-1648}.
\end{proof}

Based on Lemma \ref{lemma-4.4vbv}, we can demonstrate the global existence of $\left(u_{\varepsilon}, v_{\varepsilon}\right)$ for any $\varepsilon \in(0,1)$.
\begin{lem}\label{lemma-4.5}
Let $1<l \leqslant 3$ and assume that \eqref{assIniVal} holds. Then $T_{\max, \varepsilon}=+\infty$ for all $\varepsilon \in(0,1)$.
\end{lem}
\begin{proof}
This immediately follows from Lemma \ref{lemma-4.4vbv} when combined with \eqref{-2.7}.
\end{proof}

The $L^\infty$ estimate from Lemma \ref{lemma-4.4vbv} allows us to establish a lower bound for $v_{\varepsilon}$ using a comparison argument.
\begin{lem}\label{lemma-4.6xx}
Let $1<l \leqslant 3$ and let $T>0$, and assume that \eqref{assIniVal} holds.
Then we have  
\begin{align}\label{-4.8xx}
v_{\varepsilon}(x, t) \geqslant C(T)  \quad \text { for all } x \in \Omega, ~~t \in(0, T), \text { and  } \varepsilon \in(0,1),
\end{align}
where $C(T)$ is a positive constant depending on $v_0$, but independent of $\varepsilon$.
\end{lem}
\begin{proof}
Since
\begin{align*}
v_{\varepsilon t} \geqslant \Delta v_{\varepsilon}-c_1 v_{\varepsilon} \quad \text { in } \Omega \times(0, \infty) \quad \text { for all } \varepsilon \in(0,1),
\end{align*}
with $c_1=\sup _{\varepsilon \in(0,1)} \sup _{t>0}\left\|u_{\varepsilon}(t)\right\|_{L^{\infty}(\Omega)}$ being finite by Lemma \ref{lemma-4.4vbv}, from a comparison principle we obtain 
\begin{align*}
v_{\varepsilon}(x, t) \geqslant\left\{\inf _{\Omega} v_0\right\} \cdot e^{-c_1 t} \quad \text { for all } x \in \Omega, ~~t>0 \text { and } \varepsilon \in(0,1).
\end{align*}
We complete the proof.
\end{proof}

Since the boundedness of $u_{\varepsilon}$ and $v_{\varepsilon}$ asserted in Lemmas  \ref{lemma-4.1}, \ref{lemma-4.4vbv} and \ref{lemma-4.6xx}, the H\"{o}lder estimates of $u_{\varepsilon}$, $v_{\varepsilon}$ and $\nabla v_{\varepsilon}$ can be derived from standard parabolic regularity theory.
\begin{lem}\label{lemma-4.8}
Let $1<l \leqslant 3$ and let $T>0$, and assume that \eqref{assIniVal} holds. Then one can find $\theta_1=\theta(T) \in(0,1)$  such that 
\begin{align}\label{-4.12}
\left\|u_{\varepsilon}\right\|_{C^{\theta_1, \frac{\theta_1}{2}}(\overline{\Omega} \times[0, T])} \leqslant C_1(T) \quad \text { for all } \varepsilon \in(0,1)
\end{align}
and
\begin{align}\label{-4.13}
\left\|v_{\varepsilon}\right\|_{C^{\theta_1, \frac{\theta_1}{2}}(\overline{\Omega} \times[0, T])} \leqslant C_1(T) \quad \text { for all } \varepsilon \in(0,1),
\end{align}
where $C_1(T)$ is a positive constant independent of $\varepsilon$.
Moreover, for each $\tau>0$ and all $T>\tau$ one can also fix $\theta_2=\theta_2(\tau, T) \in(0,1)$ such that 
\begin{align}\label{-4.14}
\left\|v_{\varepsilon}\right\|_{C^{2+\theta_2, 1+\frac{\theta_2}{2}}(\overline{\Omega} \times [\tau, T])} \leqslant C_2(\tau, T) \quad \text { for all } \varepsilon \in(0,1),
\end{align}
where $C_2(T)>0$ is a positive constant independent of $\varepsilon$.
\end{lem}
\begin{proof}
It follows from Lemmas \ref{lemma-4.1}, \ref{lemma-4.4vbv}, \ref{lemma-4.6xx} and \eqref{assIniVal} as well as the H\"{o}lder regularity theory for parabolic equations (cf. \cite{1993-JDE-PorzioVespri}) that \eqref{-4.12} and \eqref{-4.13} hold. Combining the standard Schauder estimates (cf. \cite{1968-Ladyzen}) with \eqref{-4.12} and a cut-off argument, we can deduce \eqref{-4.14}.
\end{proof}

Based on the preparations above, we can now utilize a standard extraction procedure to construct a pair of limit functions $(u, v)$, which is proved to be a global bounded weak solution to the system \eqref{SYS:MAIN} as documented in Theorem \ref{thm-1.1}.

\begin{lem}\label{lemma-4.9}
Assume that the initial value $\left(u_0, v_0\right)$ satisfies \eqref{assIniVal}. Then there exist $(\varepsilon_j)_{j \in \mathbb{N}} \subset(0,1)$ as well as functions $u$ and $v$ which satisfy \eqref{solu:property2} with $u > 0$ a.e in $\Omega \times(0, \infty)$ and $v>0$ in $\overline{\Omega} \times[0, \infty)$ such that
\begin{flalign}
& u_{\varepsilon} \rightarrow u  \quad \text { in } C_{\mathrm{loc}}^0(\overline{\Omega} \times(0, \infty)) \text {, }\label{-4.15}\\
& v_{\varepsilon} \rightarrow v  \quad \text { in } C_{\mathrm{loc}}^0(\overline{\Omega} \times[0, \infty)) \text { and in } C_{\mathrm{loc}}^{2,1}(\overline{\Omega} \times(0, \infty)),\label{-4.16}\\
& \nabla v_{\varepsilon} \stackrel{*}{\rightharpoonup} \nabla v \quad \text { in } L^{\infty}(\Omega \times(0, \infty)),\label{-4.17}
\end{flalign}
as $\varepsilon=\varepsilon_j \searrow 0$, and that $(u, v)$ is a global weak solution of the system (\ref{SYS:MAIN}) as defined in Definition \ref{def-weak-sol}. 
Moreover, we have
\begin{align}\label{eq4.38}
\int_{\Omega} u_{0} \leqslant \int_{\Omega} u(t) \leqslant \int_{\Omega} u_{0}+ \int_{\Omega} v_{0}  \quad \text { for all } t>0
\end{align}
and
\begin{align}\label{eq4.38-1}
\|v(t)\|_{L^{\infty}(\Omega)} \leqslant\left\|v_0\right\|_{L^{\infty}(\Omega)} \quad \text { for all } t>0
\end{align}
as well as
\begin{align}\label{eq4.381}
\int_0^{\infty} \int_{\Omega} u v \leqslant \int_{\Omega} v_0.
\end{align}
\end{lem}
\begin{proof}
The existence of $\left\{\varepsilon_j\right\}_{j \in \mathbb{N}}$ and nonnegative functions $u$ and $v$ with the properties in \eqref{solu:property2} and \eqref{-4.15}-\eqref{-4.17} follows from Lemmas \ref{lemma-4.1} and \ref{lemma-4.8} by a diagonal extraction procedure. Moreover, $u \geqslant 0$ by nonnegativity of all the $u_{\varepsilon}$. Based on \eqref{-4.8xx} and \eqref{-4.16}, we obtain $v>0$ in $\overline \Omega \times[0, \infty)$. Using \eqref{-2.8} and \eqref{-4.15}, we obtain \eqref{eq4.38}. Using \eqref{-2.9} and \eqref{-4.16}, we also obtain \eqref{eq4.38-1}. Similarly, \eqref{eq4.381} results from \eqref{-2.10}, \eqref{-4.15} and \eqref{-4.16}.

We use \eqref{-3.37} to deduce that
\begin{align}\label{-4.23}
\left(u_{\varepsilon}^l\right)_{\varepsilon \in(0,1)} \text { is bounded in } L^2\left((0, T) ; W^{1, 2}(\Omega)\right) \quad \text { for all } T>0. 
\end{align}
The regularity requirements \eqref{-2.1} and \eqref{-2.2} in Definition \ref{def-weak-sol} become straightforward consequences of \eqref{-4.15}, \eqref{-4.16} and \eqref{-4.23}. Given $\varphi \in C_0^{\infty}(\overline{\Omega} \times[0, \infty))$ satisfying  $\frac{\partial \varphi}{\partial \nu}=0$ on $\partial \Omega \times(0, \infty)$, using \eqref{sys-regul} we see that
\begin{align}\label{ident-1.10-1}
-\int_0^{\infty} \int_{\Omega} u_{\varepsilon} \varphi_t-\int_{\Omega} u_0 \varphi(0)=&-\frac{1}{2} \int_0^{\infty} \int_{\Omega} v_{\varepsilon}  \nabla u_{\varepsilon}^{2} \cdot \nabla \varphi
+ \int_0^{\infty} \int_{\Omega} u_{\varepsilon}^{2} v_{\varepsilon} \nabla v_{\varepsilon} \cdot \nabla \varphi \nonumber\\
& +\int_0^{\infty} \int_{\Omega} u_{\varepsilon} v_{\varepsilon} \varphi
\end{align}
for all $\varepsilon \in(0,1)$. We apply \eqref{-4.15} to show that
$$
-\int_0^{\infty} \int_{\Omega} u_{\varepsilon} \varphi_t \rightarrow-\int_0^{\infty} \int_{\Omega} u \varphi_t
$$
as $\varepsilon=\varepsilon_j \searrow 0$. Based on \eqref{-4.15}, \eqref{-4.16}, \eqref{-4.17} and \eqref{-4.23}, we have
$$
-\frac{1}{l} \int_0^{\infty} \int_{\Omega} v_{\varepsilon}  \nabla u_{\varepsilon}^{l} \cdot \nabla \varphi \rightarrow -\frac{1}{l} \int_0^{\infty} \int_{\Omega} v \nabla u^{l} \cdot \nabla \varphi
$$
and
$$
\int_0^{\infty} \int_{\Omega} u_{\varepsilon}^{l} v_{\varepsilon} \nabla v_{\varepsilon} \cdot \nabla \varphi \rightarrow  \int_0^{\infty} \int_{\Omega} u^{l} v\nabla v \cdot \nabla \varphi
$$
and
$$
\int_0^{\infty} \int_{\Omega} u_{\varepsilon} \varphi \rightarrow \int_0^{\infty} \int_{\Omega} u \varphi
$$
as well as
$$
\int_0^{\infty} \int_{\Omega} u^2_{\varepsilon} \varphi \rightarrow \int_0^{\infty} \int_{\Omega} u^2 \varphi
$$
as $\varepsilon=\varepsilon_j \searrow 0$. Therefore, \eqref{ident-1.10-1} implies \eqref{-2.3}. Similarly, \eqref{-2.4} can be verified.
\end{proof}

\begin{proof}[Proof of Theorem \ref{thm-1.1}]
Theorem \ref{thm-1.1} is a direct consequence of Lemmas \ref{lemma-4.1}, \ref{lemma-4.4vbv}, \ref{lemma-4.5} and \ref{lemma-4.9}.
\end{proof}

\subsection{Proof of Theorem \ref{thm-1.1xzc}}
We will first provide a priori estimates of solutions to the approximate problem \eqref{sys-regul} using a similar method as in Section \ref{sect-3x}.

\begin{lem}\label{lem-1st-est-11}
Let $l=1$ and $T = \min \{\widetilde{T} , T_{\max,\varepsilon}\}$ for $\widetilde{T} \in \left(0, +\infty\right)$, and assume that \eqref{assIniVal} holds. 
Then we have 
\begin{align}\label{-3.5aaxz} 
\int_0^{T_{\max, \varepsilon}}\r \int_{\Omega} \frac{v_{\varepsilon}}{u_{\varepsilon}}|\nabla u_{\varepsilon }|^2  \leqslant C(T)  \quad \text { for all } \varepsilon \in(0,1),
\end{align}
where $C$ is a positive constant depending on $\int_{\Omega} u_0$ and $\int_{\Omega} |\nabla v_{0 }|^2$, but independent of $\varepsilon$.
\end{lem}
\begin{proof}
Multiply the second equation of \eqref{sys-regul} by $-\Delta v$ and integrating over $\Omega$ yield that
\begin{align}\label{1006-1659a}
\frac{1}{2} \frac{d}{d t} \int_{\Omega}|\nabla v_{\varepsilon}|^2+\int_{\Omega}|\Delta v_{\varepsilon}|^2 + \int_{\Omega}u_{\varepsilon} | \nabla v_{\varepsilon }|^2  =- \int_{\Omega}v_{\varepsilon} \nabla u_{\varepsilon }\cdot \nabla v_{\varepsilon }
\end{align}
for all $t \in(0, T_)$ and $\varepsilon \in(0,1)$. We multiply the first equation in \eqref{sys-regul} by $1+\ln u_{\varepsilon}$, integrate by parts and use \eqref{-2.9} to obtain
\begin{align}\label{1006-2240a}
\dt \int_{\Omega} u_{\varepsilon}\ln u_{\varepsilon}
= &  -\int_{\Omega} \frac{v_{\varepsilon}}{u_{\varepsilon}}|\nabla u_{\varepsilon }|^2+ \int_{\Omega} v_{\varepsilon} \nabla u_{\varepsilon }\cdot \nabla v_{\varepsilon }+\int_{\Omega}u_{\varepsilon} v_{\varepsilon} + \int_{\Omega}u_{\varepsilon} v_{\varepsilon} \ln u_{\varepsilon} \nonumber\\
\leqslant & -\int_{\Omega} \frac{v_{\varepsilon}}{u_{\varepsilon}}|\nabla u_{\varepsilon }|^2+\int_{\Omega} v_{\varepsilon} \nabla u_{\varepsilon }\cdot \nabla v_{\varepsilon } + \|v_0\|_{L^{\infty}(\Omega)} \int_{\Omega}u_{\varepsilon} \ln u_{\varepsilon}+\int_{\Omega}u_{\varepsilon} v_{\varepsilon}
\end{align}
for all $t \in(0, T)$ and $\varepsilon \in(0,1)$. Combining \eqref{1006-1659a} with \eqref{1006-2240a} and \eqref{-2.8}, \eqref{-2.9}, and using $\xi \ln \xi +\frac{1}{e}\geqslant 0$ for all $\xi>0$, we conclude that
\begin{align*}
\dt \left\{\int_{\Omega}u_{\varepsilon}\ln u_{\varepsilon}+  \frac{1}{2} \int_{\Omega} |\nabla v_{\varepsilon }|^2\right\} +\int_{\Omega} \frac{v_{\varepsilon}}{u_{\varepsilon}}|\nabla u_{\varepsilon }|^2
\leqslant &   \|v_0\|_{L^{\infty}(\Omega)} \int_{\Omega}u_{\varepsilon} \ln u_{\varepsilon} + \frac{1}{2} \int_{\Omega} |\nabla v_{\varepsilon }|^2 + \int_{\Omega}u_{\varepsilon} v_{\varepsilon} \\
\leqslant & \|v_0\|_{L^{\infty}(\Omega)} \left\{\int_{\Omega}u_{\varepsilon} \ln u_{\varepsilon}+ \frac{1}{2} \int_{\Omega} |\nabla v_{\varepsilon }|^2\right\}\\
& + \left\{\int_{\Omega}u_{\varepsilon}\ln u_{\varepsilon}+ \frac{1}{2} \int_{\Omega} |\nabla v_{\varepsilon }|^2\right\}  \\
& +  c_1 \|v_0\|_{L^{\infty}(\Omega)}  + \frac{|\Omega|}{e}\\
= &  \left(\|v_0\|_{L^{\infty}(\Omega)} +1\right) \cdot \left\{\int_{\Omega}u_{\varepsilon} \ln u_{\varepsilon}+ \frac{1}{2} \int_{\Omega} |\nabla v_{\varepsilon }|^2\right\}\\
&  +  c_1 \|v_0\|_{L^{\infty}(\Omega)}  + \frac{|\Omega|}{e}
\end{align*}
for all $t \in(0, T)$ and $\varepsilon \in(0,1)$, where $c_1=\int_{\Omega}\left(u_0+1\right)+ \int_{\Omega} v_0$. Integrating the above differential equality on $(0,t)$, we have
\begin{align}\label{1006-2249a}
\int_{\Omega}u_{\varepsilon}\ln u_{\varepsilon} + &  \frac{1}{2} \int_{\Omega} |\nabla v_{\varepsilon }|^2   \nonumber\\
\leqslant & c_4\equiv c_4(T)= \left\{\int_{\Omega}(u_{0}+1)\ln(u_{0}+1)+ \frac{1}{2} \int_{\Omega} |\nabla v_{0}|^2\right\} \cdot e^{c_2 T}+\frac{c_3}{c_2} e^{c_2 T} 
\end{align}
for all $t \in(0, T)$ and $\varepsilon \in(0,1)$, where $c_2=\|v_0\|_{L^{\infty}(\Omega)} +1$ and $c_3=c_1 \|v_0\|_{L^{\infty}(\Omega)}  + \frac{|\Omega|}{e}$ . From \eqref{1006-2249a}, we see that
\begin{align}\label{0914-1042}
\int_0^{t} \int_{\Omega} \frac{v_{\varepsilon}}{u_{\varepsilon}}|\nabla u_{\varepsilon }|^2 \leqslant c_4+c_3 T+c_2 c_4 T
\end{align}
for all $t \in(0, T)$ and $\varepsilon \in(0,1)$, which implies \eqref{-3.5aaxz}.
\end{proof}

Beginning with Lemma \ref{lem-1st-est-11} and utilizing the proof methods from Lemma \ref{lem-2nd-est} to Lemma \ref{lemma-4.1}, we can also derive the following two lemmas.
\begin{lem}\label{lemma-3.9xvb}
Let $l=1$ and $T = \min \{\widetilde{T}, T_{\max,\varepsilon}\}$ for $\widetilde{T} \in \left(0, +\infty\right)$, and assume that \eqref{assIniVal} holds. Then for all $p >2$, we have
\begin{align}\label{-3.36vb}
\int_{\Omega} u_{\varepsilon}^p(t) \leqslant C(p,T) \quad \text { for all } t \in(0, T) \text { and } \varepsilon \in(0,1)
\end{align}
and
\begin{align}\label{-3.37vb}
\int_0^{T} \int_{\Omega} u_{\varepsilon}^{p-2} \left|\nabla u_{\varepsilon}\right|^2 \leqslant  C(p,T) \quad \text { for all }  \varepsilon \in(0,1),
\end{align}
where $C(p,T)$ is some positive constant, but independent of $\varepsilon$. 
\end{lem}

\begin{lem}\label{lemma-4.1zxc}
Let $l=1$ and $T = \min \{\widetilde{T} , T_{\max,\varepsilon}\}$ for $\widetilde{T} \in \left(0, +\infty\right)$, and assume that \eqref{assIniVal} holds.
Then we have
\begin{align}\label{-4.1zxc}
\left\|v_{\varepsilon}(t)\right\|_{W^{1, \infty}(\Omega)} \leqslant C(T) \quad \text { for all } t \in(0, T) \text { and } \varepsilon \in(0,1),
\end{align}
where $C$ is a positive constant independent of $\varepsilon$.
\end{lem}

Following a standard procedure of changing variables in the second equation of \eqref{sys-regul}, we can determine the lower bound of $v_{\varepsilon}$.
\begin{lem}\label{lemma-4.6}
Let $l=1$ and $T = \min \{\widetilde{T} , T_{\max,\varepsilon}\}$ for $\widetilde{T} \in \left(0, +\infty\right)$, and assume that \eqref{assIniVal} holds. 
Then we have  
\begin{align}\label{-4.8xc}
v_{\varepsilon}(t) \geqslant C(T) \quad \text { for all } t \in(0, T) \text { and } \varepsilon \in(0,1).
\end{align}
where $C(T)$ is a positive constant depending on $v_0$, but independent of $\varepsilon$.
\end{lem}
\begin{proof}
Let
\begin{align*}
w_{\varepsilon}(x, t)=-\ln \frac{v_{\varepsilon}(x, t)}{\left\|v_0\right\|_{L^{\infty}(\Omega)}},
\end{align*}
then the second equation of system \eqref{sys-regul} with its corresponding boundary condition becomes
\begin{align*}
\begin{cases}w_{\varepsilon t}= \Delta w_{\varepsilon}-|\nabla w_{\varepsilon}|^2+  u_{\varepsilon} & x \in \Omega, t >0, \\ \frac{\partial w_{\varepsilon}}{\partial \nu}=0, & x \in \partial \Omega, t>0, \\ w_{\varepsilon}(x, 0)=-\ln \frac{v_0(x)}{\left\|v_0\right\|_{L^{\infty}(\Omega)}}, & x \in \Omega .\end{cases}
\end{align*}
Since
\begin{align*}
w_{\varepsilon t}=\Delta w_{\varepsilon}-|\nabla w_{\varepsilon}|^2+  u_{\varepsilon}\leqslant \Delta w_{\varepsilon}+ u_{\varepsilon} 
\end{align*}
for all $t \in(0, T)$ and $\varepsilon \in(0,1)$, we have
\begin{align*}
w_{\varepsilon}(t) \leqslant e^{t \Delta} w_0+\int_0^t e^{(t-s) \Delta}u_{\varepsilon}(s) ds \quad \text { for all } t \in(0, T)\cap\left(0, T_{\max , \varepsilon}\right) \text { and } \varepsilon \in(0,1). 
\end{align*}
Combining H\"{o}lder's inequality and the Neumann heat semigroup \cite{2010-JDE-Winkler}, we show with some $c_1>0$,
\begin{align*}
&\|w_{\varepsilon}(t)\|_{L^{\infty}(\Omega)} \\
\leqslant &  c_1\left\|w_0\right\|_{L^{\infty}(\Omega)}+c_1 \int_0^t\left(1+(t-s)^{-\frac{1}{p}}\right)\|u(s)\|_{L^{p}(\Omega)} \d s \\
\leqslant & c_1\left\|w_0\right\|_{L^{\infty}(\Omega)}+c_1\left(\int_0^t\|u_{\varepsilon}(s)\|_{L^{p}(\Omega)}^{p} \d s\right)^{\frac{1}{p}}\left(\int_0^t\left(1+(t-s)^{-\frac{1}{p}}\right)^\frac{p}{p-1} \d s\right)^{\frac{p-1}{p}}.
\end{align*}
Based on Lemma \ref{lemma-3.9xvb}, we can choose $p>2$ such that $\int_0^t (1+(t-s)^{-\frac{1}{p}})^{\frac{p}{p-1}} \d s  < \infty$ and
\begin{align*}
\int_0^t \int_{\Omega}u_{\varepsilon}^{p} \leqslant c_2
\end{align*}
for some $c_2=c_2(T)>0$. Therefore, we complete the proof of \eqref{-4.8xc}.
\end{proof}

We can now proceed to assert the time-dependent boundedness of $u_{\varepsilon}$.
\begin{lem}\label{lemma-4.4a}
Let $l=1$ and $T = \min \{\widetilde{T} , T_{\max,\varepsilon}\}$ for $\widetilde{T} \in \left(0, +\infty\right)$, and assume that \eqref{assIniVal} holds. Then we have
\begin{align}\label{619-1648a}
\left\|u_{\varepsilon}(t)\right\|_{L^{\infty}(\Omega)} \leqslant C(T)\quad \text { for all } t \in(0, T) \text { and } \varepsilon \in(0,1), 
\end{align}
where $C(T)$ is a positive constant independent of $\varepsilon$.
\end{lem}
\begin{proof}
We apply \eqref{-4.1zxc} and \eqref{-4.8xc} to find $c_1=c_1(T) > 0$ and $c_2=c_2(T) > 0$ such that
\begin{align*}
v_{\varepsilon} \geqslant c_1  \quad \text { for all } x \in \Omega, t \in(0, T)\text {, and } \varepsilon \in(0,1)
\end{align*}
and
\begin{align}\label{0929-1946a}
\left\|\nabla v_{\varepsilon}(t)\right\|_{L^{\infty}(\Omega)} \leqslant c_2 \quad \text { for all } t \in(0, T) \text { and } \varepsilon \in(0,1).
\end{align}
Since
\begin{align*}
u_{\varepsilon t}=\nabla \cdot\left(v_{\varepsilon} \nabla u_{\varepsilon}\right)
  - \nabla \cdot\left(u_{\varepsilon} v_{\varepsilon} \nabla v_{\varepsilon}\right)+ u_{\varepsilon} v_{\varepsilon}, \quad  x \in \Omega, ~~~t>0,
\end{align*}
applying \eqref{-2.9} and \eqref{-3.36vb} as well as \eqref{0929-1946a} implies that for each $q>1$, there exists $c_3=c_3(q, T)>0$ such that
\begin{align*}
\sup _{\varepsilon \in(0,1)} \sup _{t \in(0, T)}\left\{\left\|u_{\varepsilon}( t)v_{\varepsilon}( t\right\|_{L^q(\Omega)}+ \left\|\left(u_{\varepsilon} v_{\varepsilon} \nabla v_{\varepsilon}\right)(t)\right\|_{L^q(\Omega)}\right\} \leqslant c_3
\end{align*}
and finally, a Moser iteration result (cf. \cite{2012-JDE-TaoWinkler}) can imply \eqref{619-1648a}.
\end{proof}

Considering the above lemma and \eqref{-2.7}, we have the following result.
\begin{lem}\label{lemma-4.5aa}
Let $l = 1$ and assume that \eqref{assIniVal} holds. Then $T_{\max, \varepsilon}=+\infty$ for all $\varepsilon \in(0,1)$.
\end{lem}

Given the boundedness of $u_{\varepsilon}$ and $v_{\varepsilon}$ as established in Lemmas \ref{lemma-4.1zxc}, \ref{lemma-4.6}, and \ref{lemma-4.4a}, the Hölder estimates for $u_{\varepsilon}$, $v_{\varepsilon}$, and $\nabla v_{\varepsilon}$ can be derived using standard parabolic regularity theory.

\begin{lem}\label{lemma-4.8aa}
Let $l = 1$ and let $T>0$, and assume that \eqref{assIniVal} holds. Then one can find $\theta_1=\theta(T) \in(0,1)$  such that 
\begin{align*}
\left\|u_{\varepsilon}\right\|_{C^{\theta_1, \frac{\theta_1}{2}}(\overline{\Omega} \times[0, T])} \leqslant C_1(T) \quad \text { for all } \varepsilon \in(0,1)
\end{align*}
and
\begin{align*}
\left\|v_{\varepsilon}\right\|_{C^{\theta_1, \frac{\theta_1}{2}}(\overline{\Omega} \times[0, T])} \leqslant C_1(T) \quad \text { for all } \varepsilon \in(0,1),
\end{align*}
where $C_1(T)$ is a positive constant independent of $\varepsilon$.
Moreover, for each $\tau>0$ and all $T>\tau$ one can also fix $\theta_2=\theta_2(\tau, T) \in(0,1)$ such that 
\begin{align*}
\left\|v_{\varepsilon}\right\|_{C^{2+\theta_2, 1+\frac{\theta_2}{2}}(\overline{\Omega} \times [\tau, T])} \leqslant C_2(\tau, T) \quad \text { for all } \varepsilon \in(0,1),
\end{align*}
where $C_2(T)>0$ is a positive constant independent of $\varepsilon$.
\end{lem}
\begin{proof}
The proof follows that of Lemma \ref{lemma-4.8}.
\end{proof}
With the preparations above, we can use a standard extraction procedure to construct limit functions $(u, v)$, which are globally weak solutions to system \eqref{SYS:MAIN}, as shown in Theorem \ref{thm-1.1xzc}.

\begin{lem}\label{lemma-4.9aa}
Assume that the initial value $\left(u_0, v_0\right)$ satisfies \eqref{assIniVal}. Then there exist $(\varepsilon_j)_{j \in \mathbb{N}} \subset(0,1)$ as well as functions $u$ and $v$ which satisfy \eqref{solu:property1} with $u > 0$ a.e in $\Omega \times(0, \infty)$ and $v>0$ in $\overline{\Omega} \times[0, \infty)$ such that
\begin{flalign*}
& u_{\varepsilon} \rightarrow u  \quad \text { in } C_{\mathrm{loc}}^0(\overline{\Omega} \times(0, \infty)) \text {, }\\
& v_{\varepsilon} \rightarrow v  \quad \text { in } C_{\mathrm{loc}}^0(\overline{\Omega} \times[0, \infty)) \text { and in } C_{\mathrm{loc}}^{2,1}(\overline{\Omega} \times(0, \infty)),\\
& \nabla v_{\varepsilon} \stackrel{*}{\rightharpoonup} \nabla v \quad \text { in } L^{\infty}(\Omega \times(0, \infty)),
\end{flalign*}
as $\varepsilon=\varepsilon_j \searrow 0$, and that $(u, v)$ is a global weak solution of the system (\ref{SYS:MAIN}) as defined in Definition \ref{def-weak-sol}. 
\end{lem}
\begin{proof}
The proof parallels that of Lemma \ref{lemma-4.9}.
\end{proof}

We are now in a position to prove Theorem \ref{thm-1.1xzc}.

\begin{proof}[Proof of Theorem \ref{thm-1.1xzc}]
Theorem \ref{thm-1.1xzc} follows directly from Lemmas \ref{lemma-4.5aa} and \ref{lemma-4.9aa}.
\end{proof}

\section{The One-Dimensional Case. Proof of Theorem \ref{thm-1.1a}}\label{sect-5}

The following lemma concerns the functional differential inequality $\int_{\Omega} v_{\varepsilon}^{-\alpha}\left|v_{\varepsilon x}\right|^q$ for $q>2$ and $\alpha \in(0, q)$, which was initially proven in \cite[Lemma 4.1]{2021-TAMS-Winkler}.

\begin{lem}\label{lem-1009}
Let $q>2$ and $\alpha \in(0, q)$. Then we have 
\begin{align*}
& \frac{d}{d t} \int_{\Omega} v_{\varepsilon}^{-\alpha}\left|v_{\varepsilon x}\right|^q+\frac{1}{C} \int_{\Omega} v_{\varepsilon}^{-\alpha-2}\left|v_{\varepsilon x}\right|^{q+2} \leqslant C \int_{\Omega} u_{\varepsilon}^{\frac{q+2}{2}} v_{\varepsilon}^{q-\alpha} \quad \text { for all } t \in(0, T_{\max, \varepsilon}) \text { and } \varepsilon \in(0,1),
\end{align*}
where $C$ is some positive constant, but independent of $\varepsilon$.
\end{lem}

Building on Lemma \ref{lem-1009}, we are now able to establish time-independent $L^p$ estimates for $u_{\varepsilon}$ for any $p > 1$.

\begin{lem}\label{1012-2334}
Let $l \geqslant 1$ and assume that \eqref{assIniVal} holds.
Then for all $p>1$ we have
\begin{align}\label{-3.29c}
\int_{\Omega} u_{\varepsilon}^p(t) \leqslant C(p) \quad \text { for all } t \in(0, T_{\max, \varepsilon}) \text { and } \varepsilon \in(0,1),
\end{align}
where $C$ is a positive constant depending on $\int_{\Omega} u_0$, $\int_{\Omega} v_0$ and $\int_{\Omega}\left|v_0\right|^{-(q-1)}\left|v_{0 x}\right|^q$, but independent of $\varepsilon$. 
\end{lem}
\begin{proof}
Let
\begin{align}\label{1010-1355}
q=2(p+l-1).
\end{align}
This allows us to apply Lemma \ref{lem-1009} with $\alpha=q-1 \in(0, q)$ to obtain $c_1=c_1(q)>0$ and $c_2=c_2(q)>0$ such that
\begin{align}\label{1009-1933}
\frac{d}{d t} \int_{\Omega} v_{\varepsilon}^{-(q-1)}\left|v_{\varepsilon x}\right|^q+c_1 \int_{\Omega} v_{\varepsilon}^{-(q+1)}\left|v_{\varepsilon x}\right|^{q+2} \leqslant c_2 \int_{\Omega} u_{\varepsilon}^{\frac{q+2}{2}} v_{\varepsilon}
\end{align}
for all $t \in\left(0, T_{\max , \varepsilon}\right)$ and $\varepsilon \in(0,1)$. Next, a direct calculation using the first equation in \eqref{sys-regul} along with Young's inequality yields
\begin{align*}
\frac{1}{p} \frac{d}{d t} \int_{\Omega} u_{\varepsilon}^p + (p-1) \int_{\Omega} u_{\varepsilon}^{l+p-3} v_{\varepsilon} u_{\varepsilon x}^2= &  (p-1) \int_{\Omega} u_{\varepsilon}^{l+p-2} v_{\varepsilon} u_{\varepsilon x} v_{\varepsilon x}+  \int_{\Omega} u_{\varepsilon}^{p} v_{\varepsilon}\nonumber\\
\leqslant & \frac{(p-1)}{2} \int_{\Omega} u_{\varepsilon}^{l+p-3} v_{\varepsilon} u_{\varepsilon x}^2 \nonumber\\
& + \frac{(p-1)}{2} \int_{\Omega} u_{\varepsilon}^{l+p-1} v_{\varepsilon} v_{\varepsilon x}^2+ \int_{\Omega} u_{\varepsilon}^{p} v_{\varepsilon}
\end{align*}
for all $t \in\left(0, T_{\max , \varepsilon}\right)$ and $\varepsilon \in(0,1)$. Upon simple rearrangement together with \eqref{-2.9}, this leads to
\begin{align}\label{1010-1404}
\frac{1}{p} \frac{d}{d t} \int_{\Omega} u_{\varepsilon}^p+\frac{p-1}{2} \int_{\Omega} u_{\varepsilon}^{l+p-3} v_{\varepsilon} u_{\varepsilon x}^2 \leqslant c_3 \int_{\Omega} u_{\varepsilon}^{l+p-1} v_{\varepsilon}^{-1} v_{\varepsilon x}^2 + \int_{\Omega} u_{\varepsilon}^{p} v_{\varepsilon}
\end{align}
for all $t \in\left(0, T_{\max , \varepsilon}\right)$ and $\varepsilon \in(0,1)$, where $c_3=\frac{p-1}{2}\left\|v_0\right\|^2_{L^{\infty}(\Omega)}$. By Young's inequality, we derive
\begin{align}\label{1010-1405}
\int_{\Omega} u_{\varepsilon}^{l+p-1} v_{\varepsilon}^{-1} v_{\varepsilon x}^2 \leqslant \int_{\Omega} v_{\varepsilon}^{-(q+1)}\left|v_{\varepsilon x}\right|^{q+2}+\int_{\Omega} u_{\varepsilon}^{\frac{q+2}{2}} v_{\varepsilon}
\end{align}
and
\begin{align}\label{1010-1406}
\int_{\Omega} u_{\varepsilon}^{p} v_{\varepsilon} \leqslant \int_{\Omega} u_{\varepsilon} v_{\varepsilon} + \int_{\Omega} u_{\varepsilon}^{\frac{q+2}{2}} v_{\varepsilon}
\end{align}
for all $t \in\left(0, T_{\max , \varepsilon}\right)$ and $\varepsilon \in(0,1)$, where we have used \eqref{1010-1355}. We use Hölder inequality to see that
\begin{align}\label{1010-1728}
\left\|u_{\varepsilon}^{\frac{p+l-1}{2}} v_{\varepsilon}^{\frac{1}{2}}\right\|_{L^{\frac{2}{p+l-1}}(\Omega)}^2 & =\left\{\int_{\Omega}\left(u_{\varepsilon} v_{\varepsilon}\right)^{\frac{1}{p+l-1}} u_{\varepsilon}^{\frac{p+l-2}{p+l-1}} \right\}^{p+l-1} \nonumber\\
& \leqslant \left\{\int u_{\varepsilon}\right\}^{p+l-2} \cdot \int_{\Omega} u_{\varepsilon} v_{\varepsilon}\nonumber\\
& \leqslant c_4^{p+l-2} \int_{\Omega} u_{\varepsilon} v_{\varepsilon}
\end{align}
for all $t \in\left(0, T_{\max , \varepsilon}\right)$ and $\varepsilon \in(0,1)$, where $c_4=\int_{\Omega}\left(u_0+1\right)+ \int_{\Omega} v_0$ due to \eqref{-2.8}. To estimate the last integrand on the right-hand side of \eqref{1010-1405} and \eqref{1010-1406}, let
\begin{equation}\label{1010-1657}
\eta=\eta(p)=\min \left\{\frac{p-1}{2\eta\left(c_3+2 c_1^{-1} c_2 c_3+1\right)}, \frac{c_3}{\left(1+c_3+2 c_1^{-1} c_2 c_3\right)}\right\},
\end{equation}
and define
\begin{equation}\label{1010-1658}
\eta_1=\eta_1(p)=\min \left\{\eta, \frac{\eta}{(p+m-1)^2}, 1 \right\}.
\end{equation}
Using \eqref{1010-1728}, \eqref{1010-1405}, and Ehrling's lemma, we obtain 
\begin{align*}
\int_{\Omega} u_{\varepsilon}^{\frac{q+2}{2}} v_{\varepsilon}= & \int_{\Omega}\left(u_{\varepsilon}^{\frac{p+l-1}{2}} v_{\varepsilon}^{\frac{1}{2}}\right)^2 u_{\varepsilon} \nonumber\\
\leqslant & \left\|u_{\varepsilon}\right\|_{L^1(\Omega)} \cdot\left\|u_{\varepsilon}^{\frac{p+l-1}{2}} v_{\varepsilon}^{\frac{1}{2}}\right\|_{L^{\infty}(\Omega)}^2 \nonumber\\
\leqslant & \eta_1 \left\|\left(u_{\varepsilon}^{\frac{p+l-1}{2}} v_{\varepsilon}^{\frac{1}{2}}\right)_x\right\|_{L^2(\Omega)}^2+c_5\left\|u_{\varepsilon}\right\|_{L^1(\Omega)}\left\|u_{\varepsilon}^{\frac{p+l-1}{2}} v_{\varepsilon}^{\frac{1}{2}}\right\|_{L^{\frac{2}{p+l-1}}(\Omega)}^2 \nonumber\\
\leqslant & \frac{(p+l-1)^2}{2} \eta_1 \int_{\Omega} u_{\varepsilon}^{p+l-3} v_{\varepsilon} u_{\varepsilon x}^2+\frac{1}{2} \eta_1 \int_{\Omega} u_{\varepsilon}^{p+l-1} v_{\varepsilon}^{-1} v_{\varepsilon x}^2 \nonumber\\
& + c_5 c_4^{p+l-1}  \int_{\Omega} u_{\varepsilon} v_{\varepsilon} \nonumber\\
\leqslant & \frac{(p+l-1)^2}{2} \eta_1 \int_{\Omega} u_{\varepsilon}^{p+l-3} v_{\varepsilon} u_{\varepsilon x}^2+\frac{1}{2} \eta_1 \int_{\Omega} v_{\varepsilon}^{-(q+1)}\left|v_{\varepsilon x}\right|^{q+2} \nonumber\\
& +\frac{1}{2} \eta_1 \int_{\Omega} u_{\varepsilon}^{\frac{q+2}{2}} v_{\varepsilon} +c_6 \int_{\Omega} u_{\varepsilon} v_{\varepsilon}
\end{align*}
for all $t \in\left(0, T_{\max , \varepsilon}\right)$ and $\varepsilon \in(0,1)$, where $c_5=c_5(p)>0$ and $c_6= c_5 c_4^{p+l-1}>0$. By our choice of $\eta_1$ in \eqref{1010-1657}, this yields
\begin{align}\label{1010-1800}
\int_{\Omega} u_{\varepsilon}^{\frac{q+2}{2}} v_{\varepsilon} \leqslant \eta \int_{\Omega} u_{\varepsilon}^{p+l-3} v_{\varepsilon} u_{\varepsilon x}^2+\eta \int_{\Omega} v_{\varepsilon}^{-(q+1)}\left|v_{\varepsilon x}\right|^{q+2}+2 c_6 \int_{\Omega} u_{\varepsilon} v_{\varepsilon}
\end{align}
for all $t \in\left(0, T_{\max , \varepsilon}\right)$ and $\varepsilon \in(0,1)$. Combining this with \eqref{1010-1404}, \eqref{1010-1405} and \eqref{1010-1406} we see that
\begin{align}\label{1010-1900}
\frac{1}{p} \frac{d}{d t} \int_{\Omega} u_{\varepsilon}^p+\left\{\frac{p-1}{2}-\eta (c_3+1)\right\} \int_{\Omega} u_{\varepsilon}^{p+l-3} v_{\varepsilon} u_{\varepsilon x}^2 \leqslant &  [(1+\eta) c_3+ \eta] \int_{\Omega} v_{\varepsilon}^{-(q+1)}\left|v_{\varepsilon x}\right|^{q+2}\nonumber\\
& +(2 c_3 c_6+1) \int_{\Omega} u_{\varepsilon} v_{\varepsilon}
\end{align}
for all $t \in\left(0, T_{\max , \varepsilon}\right)$ and $\varepsilon \in(0,1)$. Together with \eqref{1010-1800}, \eqref{1009-1933} entails that
\begin{align}\label{1010-1905}
\frac{d}{d t} \int_{\Omega} v_{\varepsilon}^{-(q-1)}\left|v_{\varepsilon x}\right|^q+\left(c_1-\eta c_2\right) \int_{\Omega} v_{\varepsilon}^{-(q+1)}\left|v_{\varepsilon x}\right|^{q+2} \leqslant & \eta c_2 \int_{\Omega} u_{\varepsilon}^{p+l-3} v_{\varepsilon} u_{\varepsilon x}^2\nonumber\\
& +2 c_2 c_6 \int_{\Omega} u_{\varepsilon} v_{\varepsilon} 
\end{align}
for all $t \in\left(0, T_{\max , \varepsilon}\right)$ and $\varepsilon \in(0,1)$. Denoting
\begin{align}\label{1010-1909}
y_{\varepsilon}(t)=\frac{1}{p} \int_{\Omega}\left|u_{\varepsilon}(t)\right|^p+\frac{2 c_3}{c_1} \int_{\Omega}\left|v_{\varepsilon}(t)\right|^{-(q-1)}\left|v_{\varepsilon x}(t)\right|^q
\end{align}
for all $t \in\left(0, T_{\max , \varepsilon}\right)$ and $\varepsilon \in(0,1)$, we derive from \eqref{1010-1900} and \eqref{1010-1905} that
\begin{align}\label{1010-1912}
y_{\varepsilon}^{\prime}(t)+& \left\{\frac{p-1}{2}-\eta\left(c_3+2 c_1^{-1} c_2 c_3+1\right)\right\} \int_{\Omega} u_{\varepsilon}^{p+l-3} v_{\varepsilon} u_{\varepsilon x}^2 \nonumber\\
& +\left\{c_3-\eta\left(1+c_3+2 c_1^{-1} c_2 c_3\right)\right\} \int_{\Omega} v_{\varepsilon}^{-(q+1)}\left|v_{\varepsilon x}\right|^{q+2} \leqslant c_7 \int_{\Omega} u_{\varepsilon} v_{\varepsilon}
\end{align}
for all $t \in\left(0, T_{\max , \varepsilon}\right)$ and $\varepsilon \in(0,1)$, where $c_6=c_6(p)=2 c_3 c_6+4 c_1^{-1} c_2 c_3 c_6+1$. Thus, using the definition of $\eta$, it follows from \eqref{1010-1912} that
\begin{align*}
y_{\varepsilon}^{\prime}(t) \leqslant c_7 \int_{\Omega} u_{\varepsilon} v_{\varepsilon} \quad \text { for all } t \in\left(0, T_{\max , \varepsilon}\right) \text { and } \varepsilon \in(0,1)
\end{align*}
which, upon integration in time, thanks to \eqref{-2.10}, implies that
\begin{align*}
\int_{\Omega}\left|u_{\varepsilon}(t)\right|^p  \leqslant c_7 \int_{\Omega} v_{0} +\frac{1}{p} \int_{\Omega}\left|u_{0}\right|^p+\frac{2 c_3}{c_1} \int_{\Omega}\left|v_0\right|^{-(q-1)}\left|v_{0 x}\right|^q 
\end{align*}
for all $t \in\left(0, T_{\max , \varepsilon}\right)$ and $\varepsilon \in(0,1)$. The proof is thus completed.
\end{proof}

Using Lemma \ref{1012-2334} and a similar method as in Lemma \ref{lemma-4.1}, we can also derive $L^{\infty}$ bounds for $\nabla v_{\varepsilon}$.

\begin{lem}\label{lemma-4.1a}
Let $1 \geqslant 1$ and assume that \eqref{assIniVal} holds.
Then we have
\begin{align}\label{1013-1842}
\left\|v_{\varepsilon}(t)\right\|_{W^{1, \infty}(\Omega)} \leqslant C\quad \text { for all } t \in(0, T_{\max, \varepsilon}) \text { and } \varepsilon \in(0,1),
\end{align}
where $C$ is a positive constant independent of $\varepsilon$.
\end{lem}

By applying a standard variable change to the second equation of \eqref{sys-regul} and using a method similar to that in Lemma \ref{lemma-4.6}, we can establish the lower bound for $v_{\varepsilon}$.
\begin{lem}\label{lemma-4.6xx}
Let $l\geqslant1$ and assume that \eqref{assIniVal} holds. 
Then we have  
\begin{align*}
v_{\varepsilon}(t) \geqslant C(T) \quad \text { for all } t \in(0, T)\cap (0, T_{\max, \varepsilon})  \text { and } \varepsilon \in(0,1).
\end{align*}
where $C(T)$ is a positive constant depending on $v_0$, but independent of $\varepsilon$.
\end{lem}

We can now assert the following result using a method similar to that in Lemma \ref{lemma-4.4a}.
\begin{lem}\label{lemma-4.4ax}
Let $l\geqslant1$ and assume that \eqref{assIniVal} holds. Then we have
\begin{align*}
\left\|u_{\varepsilon}(t)\right\|_{L^{\infty}(\Omega)} \leqslant C(T)\quad \text { for all } t \in(0, T)\cap\left(0, T_{\max , \varepsilon}\right) \text { and } \varepsilon \in(0,1), 
\end{align*}
where $C(T)$ is a positive constant independent of $\varepsilon$.
\end{lem}

Based on the above lemma and \eqref{-2.7}, we have the following lemma.
\begin{lem}\label{lemma-4.5aa}
Let $l\geqslant1$ and assume that \eqref{assIniVal} holds.  Then $T_{\max, \varepsilon}=+\infty$ for all $\varepsilon \in(0,1)$.
\end{lem}

With $u_{\varepsilon}$ and $v_{\varepsilon}$ bounded as shown in Lemmas \ref{lemma-4.1a}, \ref{lemma-4.6xx}, and \ref{lemma-4.4ax}, we derive the following Hölder estimates for $u_{\varepsilon}$, $v_{\varepsilon}$, and $\nabla v_{\varepsilon}$ using standard parabolic regularity theory.

\begin{lem}\label{lemma-4.8aa}
Let $l\geqslant1$ and let $T>0$, and assume that \eqref{assIniVal} holds. Then one can find $\theta_1=\theta(T) \in(0,1)$  such that 
\begin{align*}
\left\|u_{\varepsilon}\right\|_{C^{\theta_1, \frac{\theta_1}{2}}(\overline{\Omega} \times[0, T])} \leqslant C_1(T) \quad \text { for all } \varepsilon \in(0,1)
\end{align*}
and
\begin{align*}
\left\|v_{\varepsilon}\right\|_{C^{\theta_1, \frac{\theta_1}{2}}(\overline{\Omega} \times[0, T])} \leqslant C_1(T) \quad \text { for all } \varepsilon \in(0,1),
\end{align*}
where $C_1(T)$ is a positive constant independent of $\varepsilon$.
Moreover, for each $\tau>0$ and all $T>\tau$ one can also fix $\theta_2=\theta_2(\tau, T) \in(0,1)$ such that 
\begin{align*}
\left\|v_{\varepsilon}\right\|_{C^{2+\theta_2, 1+\frac{\theta_2}{2}}(\overline{\Omega} \times [\tau, T])} \leqslant C_2(\tau, T) \quad \text { for all } \varepsilon \in(0,1),
\end{align*}
where $C_2(T)>0$ is a positive constant independent of $\varepsilon$.
\end{lem}

Using the preparations above, we apply a standard extraction procedure to construct limit functions $(u, v)$, which form a globally bounded weak solution to system \eqref{SYS:MAIN}, as shown in Theorem \ref{thm-1.1a}.

\begin{lem}\label{lemma-4.9bn}
Assume that the initial value $\left(u_0, v_0\right)$ satisfies \eqref{assIniVal}. Then there exist $(\varepsilon_j)_{j \in \mathbb{N}} \subset(0,1)$ as well as functions $u$ and $v$ which satisfy \eqref{solu:property3} with $u > 0$ a.e in $\Omega \times(0, \infty)$ and $v>0$ in $\overline{\Omega} \times[0, \infty)$ such that
\begin{flalign}
& u_{\varepsilon} \rightarrow u  \quad \text { in } C_{\mathrm{loc}}^0(\overline{\Omega} \times(0, \infty)) \text {, }\label{-4.15zx}\\
& v_{\varepsilon} \rightarrow v  \quad \text { in } C_{\mathrm{loc}}^0(\overline{\Omega} \times[0, \infty)) \text { and in } C_{\mathrm{loc}}^{2,1}(\overline{\Omega} \times(0, \infty)),\label{-4.16zx}\\
& \nabla v_{\varepsilon} \stackrel{*}{\rightharpoonup} \nabla v \quad \text { in } L^{\infty}(\Omega \times(0, \infty)),\nonumber
\end{flalign}
as $\varepsilon=\varepsilon_j \searrow 0$, and that $(u, v)$ is a global weak solution of the system (\ref{SYS:MAIN}) as defined in Definition \ref{def-weak-sol}. 
Moreover, we have
\begin{align*}
\int_{\Omega} u_{0} \leqslant \int_{\Omega} u(t) \leqslant \int_{\Omega} u_{0}+ \int_{\Omega} v_{0}  \quad \text { for all } t>0
\end{align*}
and
\begin{align*}
\|v(t)\|_{L^{\infty}(\Omega)} \leqslant\left\|v_0\right\|_{L^{\infty}(\Omega)} \quad \text { for all } t>0
\end{align*}
as well as
\begin{align*}
\int_0^{\infty} \int_{\Omega} u v \leqslant \int_{\Omega} v_0.
\end{align*}
\end{lem}
\begin{proof}
The proof follows a similar approach to that of Lemma \ref{lemma-4.9}.
\end{proof}

Collecting \cite[Lemma 5.2 and Corollary
5.3]{2021-TAMS-Winkler} and \cite[Lemmas 5.1 and 5.2]{2022-CPAA-LiWinkler}, we can complete the proof of the following lemma.

\begin{lem}\label{lemma-1010-x}
Let $1 \geqslant 1$ and assume that \eqref{assIniVal} holds.
Then we have
\begin{align}\label{1010-1937}
v_{\varepsilon}(x, t) \geqslant C \left\|v_{\varepsilon}(t)\right\|_{L^{\infty}(\Omega)} \quad \text { for all } x \in \Omega, t>0 \text { and } \varepsilon \in(0,1)
\end{align}
and
\begin{align}\label{1010-2029}
\left|v_{\varepsilon x}(x, t)\right| \leqslant C v_{\varepsilon}(x, t) \quad \text { for all } x \in \Omega, t>0 \text { and } \varepsilon \in(0,1)
\end{align}
as well as
\begin{align}\label{1010-2030}
\int_0^{\infty}\left\|v_{\varepsilon}(t)\right\|_{L^{\infty}(\Omega)} \d t \leqslant C \frac{\int_{\Omega} v_0}{\int_{\Omega} u_0} \quad \text { for all } \varepsilon \in(0,1)
\end{align}
where $C$ is a positive constant independent of $\varepsilon$.
\end{lem}

\begin{lem}
Let $1 \geqslant 1$ and assume that \eqref{assIniVal} holds. Assume that $\left(\varepsilon_j\right)_{j \in \mathbb{N}}$ is provided by Lemma \ref{lemma-4.9bn}. For each $\varepsilon \in\left(\varepsilon_j\right)_{j \in \mathbb{N}}$, define
\begin{align}
L_{\varepsilon}& =\int_0^{\infty}\left\|v_{\varepsilon}(t)\right\|_{L^{\infty}(\Omega)} \d t \nonumber\\
w_{\varepsilon}(x, \tau) & =u_{\varepsilon}\left(x, g_{\varepsilon}^{-1}(\tau)\right), \label{1010-2336} \quad x \in \overline{\Omega}, \tau \in[0,1)
\end{align}
and
\begin{align*}
g_{\varepsilon}(t)=\frac{1}{L_{\varepsilon}} \cdot \int_0^t\left\|v_{\varepsilon}(s)\right\|_{L^{\infty}(\Omega)} \d s
\end{align*}
Moreover, denoting $t=g_{\varepsilon}^{-1}(\tau)$, then for each $\varepsilon \in\left(\varepsilon_j\right)_{j \in \mathbb{N}}$ we let
\begin{align*}
a_{\varepsilon}(x, \tau) & =L_{\varepsilon} \cdot \frac{ v_{\varepsilon}(x, t)}{\left\|v_{\varepsilon}(t)\right\|_{L^{\infty}(\Omega)}} \quad \text { and } \quad b_{\varepsilon}(x, \tau)=L_{\varepsilon} \cdot \frac{ v_{\varepsilon}(x, t) v_{\varepsilon x}(x, t)}{\left\|v_{\varepsilon}(t)\right\|_{L^{\infty}(\Omega)}}, \nonumber\\
(x, \tau) & \in \Omega \times(0,1)
\end{align*}
Then we have
\begin{align}\label{1010-2327}
L_{\varepsilon} \rightarrow L =\int_0^{\infty}\|v(t)\|_{L^{\infty}(\Omega)} \d t, \quad \text { as } \varepsilon=\varepsilon_j \searrow 0
\end{align}
and there exists $C>1$ satisfying that
\begin{align}\label{1010-2333}
a_{\varepsilon}(x, \tau) \geqslant \frac{1}{C} \quad \text { for all } x \in \Omega, \tau \in(0,1), \text { and } \varepsilon \in\left(\varepsilon_j\right)_{j \in \mathbb{N}}
\end{align}
and
\begin{align}\label{1010-2334}
\left|b_{\varepsilon}(x, \tau)\right| \leqslant C \quad \text { for all } x \in \Omega, \tau \in(0,1) \text {, and } \varepsilon \in\left(\varepsilon_j\right)_{j \in \mathbb{N}}
\end{align}
as well as for each $\varepsilon \in\left(\varepsilon_j\right)_{j \in \mathbb{N}}$, the function $w_{\varepsilon}$ satisfies
\begin{align}\label{1010-2338}
\begin{cases}
w_{\varepsilon \tau}=\left(a_{\varepsilon}(x, \tau) w_{\varepsilon}^{l-1} w_{\varepsilon x}\right)_x-\left(b_{\varepsilon}(x, \tau) w_{\varepsilon}^{l}\right)_x+ a_{\varepsilon}(x, \tau) w_{\varepsilon}, x \in \Omega, \tau \in(0,1) \\
w_{\varepsilon x}=0, \quad x \in \partial \Omega, \quad \tau \in(0,1) \\
w_{\varepsilon}(x, 0)=u_0(x)+\varepsilon, \quad x \in \Omega
\end{cases}
\end{align}
\end{lem}
\begin{proof}
From \eqref{1010-2030}, we can infer that for any given $\eta>0$, there exist $t_0=t_0(\eta)>0$ and $\varepsilon^{\star}=\varepsilon^{\star}(\eta) \in(0,1)$ such that
\begin{align}\label{1010-2324}
\int_{t_0}^{\infty}\left\|v_{\varepsilon}(t)\right\|_{L^{\infty}(\Omega)} \d t \leqslant \frac{\eta}{3} \quad \text { for all } \varepsilon \in\left(\varepsilon_j\right)_{j \in \mathbb{N}} \text { satisfying } \varepsilon<\varepsilon^{\star}
\end{align}
Further applying \eqref{-4.16} and Fatou’s lemma, we obtain
\begin{align}\label{1010-2325}
\int_{t_0}^{\infty}\|v(t)\|_{L^{\infty}(\Omega)} \d t \leqslant \frac{\eta}{3}
\end{align}
From \eqref{-4.16}, it follows that we can find $\varepsilon^{\star \star}=\varepsilon^{\star \star}(\eta) \in(0,1)$ such that
\begin{align}\label{1010-2326}
\int_0^{t_0}\left\|v_{\varepsilon}(t)-v(t)\right\|_{L^{\infty}(\Omega)} d t \leqslant \frac{\eta}{3} \quad \text { for all } \varepsilon \in\left(\varepsilon_j\right)_{j \in \mathbb{N}} \text { satisfying } \varepsilon<\varepsilon^{\star \star}
\end{align}
Collecting \eqref{1010-2324}-\eqref{1010-2326}, we infer that whenever $\varepsilon \in\left(\varepsilon_j\right)_{j \in \mathbb{N}}$ is such that $\varepsilon<  \min \left\{\varepsilon_{\star}, \varepsilon_{\star \star}\right\} $ 
\begin{align*}
& \left|L_{\varepsilon}-L\right| \nonumber\\
\leqslant & \int_0^{t_0}\left\|v_{\varepsilon}(t)-v(t)\right\|_{L^{\infty}(\Omega)} \d t+\int_{t_0}^{\infty}\|v(t)\|_{L^{\infty}(\Omega)} \d t+\int_{t_0}^{\infty}\left\|v_{\varepsilon}(t)\right\|_{L^{\infty}(\Omega)} \d t  \nonumber\\
\leqslant & \eta
\end{align*}
which implies \eqref{1010-2327}. Thus, \eqref{1010-1937} and
\eqref{1010-2029}, combined with \eqref{1010-2327} and the positivity of each $L_{\varepsilon}$, result in \eqref{1010-2333} and \eqref{1010-2334}. Finally, using \eqref{1010-2336} and \eqref{sys-regul}, we immediately deduce \eqref{1010-2338}.
\end{proof}

\begin{lem}\label{lem-1013}
Let $1 \geqslant 1$ and assume that \eqref{assIniVal} holds. Then we have
\begin{align}\label{1010-2351}
\|u(t)\|_{L^{\infty}(\Omega)} \leqslant C \quad \text { for all } t>0
\end{align}
and
\begin{align}\label{1010-2352}
\left\|v(t)\right\|_{W^{1, \infty}(\Omega)} \leqslant C \quad \text { for all } t>0
\end{align}
\end{lem}
\begin{proof}
In view of Lemma \ref{lemma-4.9}, \eqref{1010-2336}, \eqref{1010-2333}, and \eqref{1010-2334}, the Moser-type argument used in \cite{2012-JDE-TaoWinkler} when applied to \eqref{1010-2338} allows us to find $c_1>0$ such that
$$
\left\|w_{\varepsilon}(\tau)\right\|_{L^{\infty}(\Omega)} \leqslant C \quad \text { for all } \tau \in(0,1) \text { and } \varepsilon \in(0,1)
$$
which implies
$$
\left\|u_{\varepsilon}(t)\right\|_{L^{\infty}(\Omega)} \leqslant C \quad \text { for all } t>0 \text { and } \varepsilon \in(0,1)
$$
due to \eqref{1010-2336}. From \eqref{-4.15zx}, \eqref{-4.16zx}, and \eqref{1013-1842}, we have \eqref{1010-2351} and \eqref{1010-2352}.
\end{proof}

\begin{proof}[Proof of Theorem \ref{thm-1.1a}]
Theorem \ref{thm-1.1a} follows directly from Lemmas \ref{lemma-4.9bn} and \ref{lem-1013}.
\end{proof}




\end{document}